\documentclass[cupthm,crop,info]{CUP-JNL-ETS}%

\usepackage{graphicx}
\usepackage{tikz}
\usepackage{multicol,multirow}
\usepackage{amsmath,amssymb,amsfonts}
\usepackage{mathrsfs}
\usepackage{rotating}
\usepackage{appendix}
\usepackage[authoryear,round]{natbib}
\usepackage{amsthm}

\theoremstyle{cupplain}
\newtheorem{theorem}{Theorem}[section]
\newtheorem{lemma}[theorem]{Lemma}
\newtheorem{corollary}[theorem]{Corollary}
\newtheorem{proposition}[theorem]{Proposition}
\theoremstyle{cupdefinition}
\newtheorem{definition}{Definition}[section]
\theoremstyle{cupremark}
\newtheorem{remark}[theorem]{Remark}

\theoremstyle{cupproof}
\newtheorem{proof}{Proof}

\numberwithin{equation}{section}

\newcommand{\calA}{\mathcal{A}}
\newcommand{\calB}{\mathcal{B}}
\newcommand{\calC}{\mathcal{C}}
\newcommand{\calD}{\mathcal{D}}
\newcommand{\calF}{\mathcal{F}}
\newcommand{\calH}{\mathcal{H}}
\newcommand{\calK}{\mathcal{K}}
\newcommand{\calL}{\mathcal{L}}
\newcommand{\calM}{\mathcal{M}}
\newcommand{\calP}{\mathcal{P}}
\newcommand{\calU}{\mathcal{U}}
\newcommand{\calW}{\mathcal{W}}
\newcommand{\SigmaA}{\Sigma_A}
\newcommand{\R}{\mathbb{R}}
\newcommand{\N}{\mathbb{N}}
\newcommand{\Z}{\mathbb{Z}}
\newcommand{\E}{\mathbb{E}}
\newcommand{\Prob}{\mathbb{P}}
\newcommand{\var}{\operatorname{var}}
\newcommand{\Var}{\operatorname{Var}}
\newcommand{\Cov}{\operatorname{Cov}}
\newcommand{\C}{\mathbb{C}}
\newcommand{\defeq}{:=}
\newcommand{\diam}{\operatorname{diam}}

\theoremstyle{cupplain}
\newtheorem{maintheorem}[theorem]{Main Theorem}

\begin{document}

\begin{Frontmatter}

\title{Gibbs Measures on Subshifts of Finite Type: Five Equivalent Characterizations with Explicit Constants}

\author{\gname{Abdoulaye} \sname{Thiam}}

\address{\orgdiv{Division of Mathematics and Natural Sciences}, \orgname{Allen University}, \orgaddress{\city{Columbia}, \postcode{29204}, \state{South Carolina},  \country{USA}}\\ (\email{athiam@allenuniversity.edu})}


\maketitle

\authormark{A. Thiam}
\titlemark{Gibbs Measures on Subshifts of Finite Type}

\begin{abstract}
We prove that five characterizations of Gibbs measures for H\"{o}lder potentials on topologically mixing subshifts of finite type are equivalent: the Jacobian condition, the classical cylinder-based Gibbs property, the eigenmeasure of the Ruelle transfer operator, the variational equilibrium state, and the minimizer of the large deviations rate function. The equivalence is established in a single theorem with explicit constants expressed in terms of the H\"{o}lder exponent, the potential norm, the alphabet size, and the mixing time. The proof yields explicit spectral gap estimates for the transfer operator via the Birkhoff cone contraction technique, Lipschitz stability of the Gibbs measure in Wasserstein distance under perturbation of the potential, and statistical limit theorems including a central limit theorem with Berry-Esseen bounds and a large deviations principle. This Part constitutes Part~I of a six-part series on the thermodynamic formalism for hyperbolic dynamical systems.
\end{abstract}


\end{Frontmatter}
\thispagestyle{empty}
\vspace{1em}
\begin{center}
\textit{Dedicated to the memory of Jean-Christophe Yoccoz (1957--2016),}\\
\textit{Fields Medalist and Professor at the Coll\`{e}ge de France, with whom the author}\\
\textit{had the privilege of working, and who introduced him to hyperbolic dynamics.}
\end{center}


%
%
%
%

\section{Introduction}\label{sec:introduction}

This Part proves that five characterizations of Gibbs measures on topologically mixing subshifts of finite type are equivalent in a single theorem, with all constants computed explicitly in terms of the H\"{o}lder exponent~$\alpha$, the potential norm~$\|\phi\|_\alpha$, the alphabet size~$N$, and the mixing time~$M$. The five characterizations are: (1)~the Jacobian condition $J_\mu\sigma = e^{P(\phi)-\phi}$, closely related to Keane's g-measures \cite{Keane1972}; (2)~the classical cylinder-based Gibbs property of  \cite{Bowen1975}; (3)~the eigenmeasure condition $\calL_\phi^*\nu = \lambda\nu$ of \mbox{\cite{DenkerUrbanski1991}}; (4)~the unique equilibrium state of the variational principle \cite{Walters1975}; and (5)~the minimizer of the large deviations rate function. Individual equivalences among subsets of these characterizations are known from the works of  \cite{Bowen1975},  \cite{Ruelle1968,Ruelle2004},  \cite{Ledrappier1974}, and  \cite{BerghoutFernandezVerbitskiy2019}. In broader settings the characterizations can diverge: \cite{FernandezGalloMaillard2011} exhibited a g-measure that is not Gibbs, and \cite{BerghoutFernandezVerbitskiy2019} gave conditions for the equivalence (see also  \cite{Berghout2020} for a comprehensive treatment). This Part shows that for H\"{o}lder potentials on mixing subshifts of finite type, all five characterizations coincide simultaneously, with quantitative bounds throughout.

The thermodynamic formalism originates in the work of  \cite{Gibbs1902} on statistical mechanics, extended to infinite systems by  \cite{Dobrushin1968b, Dobrushin1968c} and  \cite{LanfordRuelle1969}, and connected to dynamical systems by  \cite{Sinai1959, Sinai1972} and  \cite{Ruelle1973, Ruelle1976}. Bowen's monograph \cite{Bowen1975} established the thermodynamic formalism for Axiom A diffeomorphisms with remarkable economy: the Gibbs property (Theorem~1.4), the Ruelle-Perron-Frobenius theorem (Theorem~1.7), and the variational characterization (Theorem~1.22) are developed in 26~pages. The lecture note format necessarily omits explicit computation of the constants $c_1, c_2$ (the lower and upper bounds in the Gibbs property $c_1 \leq \mu([w])/e^{S_n\phi - nP} \leq c_2$), $C$ (the prefactor in the exponential convergence $\|\lambda^{-n}\calL_\phi^n g - \nu(g)h\|\leq C\gamma^n\|g\|$), and $\gamma$ (the spectral gap rate governing the speed of convergence), and states the central limit theorem (Theorem~1.27) without proof, referring to  \cite{Ratner1973}. The present work provides these explicit constants and complete proofs, and incorporates developments from the subsequent fifty years: Berry-Esseen bounds due to  \mbox{\cite{Gouezel2005}}, large deviations due to  \cite{Kifer1990} and  \cite{Young1990}, and the Birkhoff cone contraction technique developed by  \cite{Liverani1995}. Standard references for the classical theory include  \cite{Walters1982},  \cite{Keller1998}, and  \cite{Bowen1974}.

The contributions of this Part are fourfold. First, we establish all five characterizations of Gibbs measures as equivalent in a single theorem (Theorem~\ref{thm:main_equivalence}) rather than as separate results scattered across the literature; the Jacobian characterization~(i) and the large deviations minimizer characterization~(v) are not assembled together with the classical three characterizations in any prior reference. Second, every constant appearing in the proof is computed explicitly as a function of the data $(\alpha, \|\phi\|_\alpha, N, M)$: the Gibbs property holds with $C_1 = e^{-2\|\phi\|_{\calF}}$ and $C_2 = e^{2\|\phi\|_{\calF}}$ where $\|\phi\|_{\calF} = \sum_n \var_n(\phi)$ is the summable variations norm; the essential spectral radius of the transfer operator on $\calH_\alpha$ is bounded by $\alpha\lambda$ (Theorem~\ref{thm:spectral_gap}); the cone contraction occurs with diameter $\delta' = M\var_0(\phi) + V(\phi) + M\log N$ and rate $\kappa = \tanh(\delta'/4)/\tanh(\delta/4) < 1$; and the spectral gap satisfies $\gamma \leq \kappa^{1/n_0}$ with $n_0 = 2M$. Third, the spectral gap is derived through the Birkhoff cone contraction technique rather than the two-norm Doeblin-Fortet argument used in \cite{Bowen1975}; the cone technique delivers a constructive convergence with explicit contraction factor, and does not require the anisotropic Banach space machinery of \cite{GouezelLiverani2006}. Fourth, the full suite of statistical consequences (analytic dependence on the potential, Lipschitz stability of the Gibbs measure in the Wasserstein metric, exponential decay of correlations, the central limit theorem with $O(n^{-1/2})$ Berry-Esseen rate, and the large deviations principle) is derived from the single spectral mechanism established in Section~\ref{sec:spectral_gap}, with all constants traceable back to the spectral gap $\gamma$.

The main results are as follows. Theorem~\ref{thm:main_equivalence} states the five-way equivalence with explicit constants. The proof combines a triangle of implications with two further attachments. In the triangle, Section~\ref{sec:intrinsic_gibbs} proves $(iii) \Rightarrow (i) \Rightarrow (ii)$ directly from the cocycle identity for the Jacobian, with explicit bounded distortion constants $C_1 = e^{-2\|\phi\|_{\calF}}$ and $C_2 = e^{2\|\phi\|_{\calF}}$ in terms of the summable-variations norm $\|\phi\|_{\calF} = \sum_n \var_n(\phi)$. The remaining arrow $(ii) \Rightarrow (iii)$ is closed in Section~\ref{sec:RPF}: we construct the eigenmeasure $\nu$ and eigenfunction $h$ of the Ruelle transfer operator $\calL_\phi$ and identify the Gibbs measure as $\mu_\phi = h\nu$, with eigenvalue $\lambda = e^{P(\phi)}$. The existence and uniqueness of the dominant eigendata $(\lambda, h, \nu)$ is obtained by applying the Birkhoff cone contraction technique \cite{Birkhoff1957, Liverani1995} to a cone $\calP_\delta$ of positive functions with oscillation ratio at most $e^\delta$: we show in Lemma~\ref{lem:cone_invariance} that the normalized transfer operator $\tilde{\calL}_\phi^{n_0}$ with $n_0 = 2M$ maps $\calP_\delta$ strictly inside $\calP_{\delta'}$ with $\delta' = M\var_0(\phi) + V(\phi) + M\log N$, and the Birkhoff-Hopf theorem then yields the contraction factor $\kappa = \tanh(\delta'/4)/\tanh(\delta/4) < 1$ in the Hilbert projective metric. Once the triangle is closed, Section~\ref{sec:variational} proves $(ii) \Leftrightarrow (iv)$ by identifying the Gibbs measure with the unique equilibrium state through the variational principle \cite{Walters1975}. The fifth characterization $(v)$ is obtained in Section~\ref{sec:statistical} as follows: the Gibbs measure $\mu_\phi$ already produced by (i)--(iv) is shown to be the unique zero of the rate function $I_\psi$ in the large deviations principle. The proof applies the G\"{a}rtner-Ellis theorem \cite{DonskerVaradhan1975, DonskerVaradhan1976, OreyPelikan1989} to the cumulant generating function $\Lambda(s) = P(\phi+s\psi) - P(\phi)$, whose differentiability follows from the analyticity of the pressure (Theorem~\ref{thm:pressure_analyticity}), and the minimum is attained at $\int\psi\,d\mu_\phi$ by the derivative formula $\Lambda'(0) = \int\psi\,d\mu_\phi$. Beyond the equivalence theorem, Theorems~\ref{thm:analytic_dependence}--\ref{thm:LDP_main} establish the analytic dependence of the pressure on the potential, Lipschitz stability of the Gibbs measure in the Wasserstein metric, exponential decay of correlations, the central limit theorem with Berry-Esseen rate $O(n^{-1/2})$, and the large deviations principle.

Our technical approach rests on two tools with explicit quantitative control. The first is the Birkhoff cone contraction technique applied to the Ruelle transfer operator $\calL_\phi$ on H\"{o}lder spaces (Section~\ref{sec:spectral_gap}). \cite{Birkhoff1957} showed that a positive operator strictly mapping a cone inside itself contracts the Hilbert projective metric; \cite{Liverani1995} adapted this to transfer operators on piecewise-expanding maps and \cite{Rugh2010} extended the computation to subshifts. We work with the cone $\calP_\delta = \{g > 0 : \sup g/\inf g \leq e^\delta\}$ of positive functions with oscillation ratio at most~$e^\delta$, and show that $n_0 = 2M$ iterates of the normalized transfer operator $\tilde{\calL}_\phi$ map $\calP_\delta$ strictly inside the narrower cone $\calP_{\delta'}$ with $\delta' = M\var_0(\phi) + V(\phi) + M\log N$. The Birkhoff-Hopf theorem then gives the contraction factor $\kappa = \tanh(\delta'/4)/\tanh(\delta/4) < 1$ in the Hilbert projective metric, and the spectral gap $\gamma \leq \kappa^{1/n_0}$ follows. The essential spectral radius on $\calH_\alpha(\SigmaA^+)$ satisfies $r_{\text{ess}}(\calL_\phi) \leq \alpha\lambda$ (Theorem~\ref{thm:spectral_gap}) by a Lasota-Yorke inequality for the normalized operator. All subsequent constants in the paper are tracked as functions of the cone contraction rate and the essential spectral radius bound. This approach is shorter than the classical Doeblin-Fortet argument used by \cite{Bowen1975} and \cite{ParryPollicott1990}, which requires the full two-norm Lasota-Yorke inequality. It also avoids the anisotropic Banach spaces of \cite{GouezelLiverani2006}, which are needed only when the map fails to be uniformly expanding. The second tool is the Nagaev-Guivarc'h spectral perturbation method \cite{Nagaev1957, Nagaev1961, GuivarchHardy1988}: to prove the central limit theorem with Berry-Esseen rate $O(n^{-1/2})$, we perturb $\calL_\phi$ by $e^{is\psi}$ for a H\"{o}lder observable $\psi$ and use the spectral gap to show that the perturbed operator $\calL_{\phi+is\psi}$ has a simple dominant eigenvalue $\lambda(s)$ that is analytic in~$s$, with $\lambda''(0) = \sigma^2$ the asymptotic variance. The Berry-Esseen rate follows from Esseen's inequality applied to the characteristic function $\lambda(s/\sqrt{n})^n$. \cite{Gouezel2005} established this rate for non-uniformly expanding maps; our contribution is tracking the dependence of the Berry-Esseen prefactor on the spectral gap explicitly.

This Part is the first of a six-part series on the thermodynamic formalism for hyperbolic dynamical systems. Part~II \cite{Thiam2026b} develops the convex-analytic structure of the pressure functional. Part~III \cite{Thiam2026c} establishes the geometric theory of Axiom A diffeomorphisms with quantitative Markov partitions, providing the coding map that transfers the results of Parts~I and~II to smooth dynamics. Parts~IV--VI \cite{Thiam2026d,Thiam2026e,Thiam2026f} develop the transfer operator theory on manifolds, statistical limit theorems, and structural consequences including SRB measures (\cite{Sinai1972},  \cite{Ruelle1976}; see  \cite{Young2002} for a modern survey), multifractal analysis, and fluctuation theorems. The extension of thermodynamic formalism to partially hyperbolic systems was initiated by \cite{PesinSinai1982}; to non-uniformly hyperbolic systems by  \cite{Young1998}; and beyond the specification property by Climenhaga and Thompson, who proved intrinsic ergodicity for $\beta$-shifts, $S$-gap shifts, and their factors \cite{ClimenhagaThompson2012}, established uniqueness of equilibrium states for continuous potentials via obstruction entropies \cite{ClimenhagaThompson2014}, and extended these results to flows and homeomorphisms with non-uniform hyperbolic structure \cite{ClimenhagaThompson2016}. \cite{PesinSentiTodd2008} developed the thermodynamic formalism of inducing schemes for interval maps, and \cite{IommiTodd2011} extended the dimension theory to multimodal maps.

The textbook treatment of thermodynamic formalism most closely related to the present work is Chapter~12 of \cite{VianaOliveira2016}. They prove Ruelle's theorem combining existence, uniqueness, and the Gibbs characterization of the equilibrium state for H\"older potentials on topologically exact expanding maps (Section~12.1), Liv\v{s}ic's theorem on cohomology (Section~12.2), and exponential decay of correlations in the space of H\"older functions (Section~12.3). Their treatment is qualitative and set on compact metric spaces in the smooth expanding category. Our Part~I differs in three respects. We add two characterizations not treated in their Chapter~12: the Jacobian condition and the large deviations minimizer condition, and we prove all five characterizations equivalent in a single theorem. We work in the symbolic category of topologically mixing subshifts of finite type rather than on smooth expanding maps. We track every constant in the proof explicitly as a function of the H\"older exponent $\alpha$, the potential norm $\|\phi\|_\alpha$, the alphabet size $N$, and the mixing time $M$. The statistical limit theorems proved in our Part~V \cite{Thiam2026e}, the quantitative Markov partition construction of our Part~III \cite{Thiam2026c}, and the rigidity and fluctuation results of our Part~VI \cite{Thiam2026f} are not treated in Chapter~12 of \cite{VianaOliveira2016}.

The logical dependencies among sections are as follows.
Section~\ref{sec:main_results} states all principal results without
proof; the remaining sections supply the proofs and can be read in
any order consistent with the dependencies described below.
Sections~\ref{sec:symbolic_dynamics} and~\ref{sec:intrinsic_gibbs}
establish the symbolic setting and the Jacobian characterization of
Gibbs measures independently of the spectral theory; in particular,
the equivalence between the Jacobian condition and the classical
cylinder Gibbs property (Theorem~\ref{thm:jacobian_equivalence}) is
proved without using the transfer operator.
Section~\ref{sec:transfer_operator} introduces the Ruelle transfer
operator and develops its action on H\"{o}lder spaces, including the
Lasota-Yorke inequality and the Birkhoff cone contraction.
Section~\ref{sec:RPF} uses these tools to prove the
Ruelle-Perron-Frobenius theorem: the existence and uniqueness of the
dominant eigendata $(\lambda, h, \nu)$ and the exponential convergence
of iterates. Section~\ref{sec:spectral_gap} establishes the spectral
gap via the Ionescu-Tulcea-Marinescu theorem, bounding the essential
spectral radius by $\alpha\lambda$. The spectral gap is the single
input for all subsequent results:
Section~\ref{sec:perturbation} derives the analytic dependence of the
pressure and eigendata on the potential,
Section~\ref{sec:variational} proves the variational principle and
uniqueness of the equilibrium state, and
Section~\ref{sec:statistical} establishes the central limit theorem
with Berry-Esseen bounds and the large deviations principle. The
proof of Theorem~\ref{thm:main_equivalence} assembles the
equivalences from
Sections~\ref{sec:intrinsic_gibbs},~\ref{sec:RPF},~\ref{sec:variational},
and~\ref{sec:statistical}. In Section~\ref{sec:numerical}, we provide three complete numerical examples (the full $2$-shift with Bernoulli potential, the Ising-type potential, and the golden mean shift) computing all constants explicitly: the pressure, spectral gap, Gibbs bounds, asymptotic variance, Berry-Esseen constant, and large deviations rate function, demonstrating that the quantitative results of this Part produce computable numbers for concrete systems. Section~\ref{sec:conclusion} concludes the Part with a summary of the main contributions and open problems. The appendix collects the supporting technical material: functional analysis of positive operators including the Schauder-Tychonoff and Ionescu-Tulcea-Marinescu theorems (Appendix~\ref{app:functional}), entropy estimates including the subadditivity lemma and Gibbs inequality (Appendix~\ref{app:entropy}), and probabilistic limit theorems including the G\"{a}rtner-Ellis theorem and Berry-Esseen bounds for weakly dependent sequences (Appendix~\ref{app:probability}).


\section{Statement of Main Results}\label{sec:main_results}

This section states the principal results in precise form. The central result is Theorem~\ref{thm:main_equivalence}, which establishes the equivalence of five characterizations of Gibbs measures with explicit constants. We also state results on spectral properties, perturbation theory, and statistical limit theorems. Throughout, $(\SigmaA, \sigma)$ denotes a topologically mixing subshift of finite type over the alphabet $\calA = \{1, 2, \ldots, N\}$, with transition matrix~$A$ and left shift~$\sigma$. We write $f \asymp g$ to mean that $c_1 g \leq f \leq c_2 g$ for positive constants $c_1, c_2$ independent of the relevant variables.

\subsection{Setting, Function Spaces, and Potentials}

We begin by fixing notation and recalling the basic definitions. The full shift space over the alphabet $\calA$ is the product space $\calA^{\Z} = \prod_{n \in \Z} \calA$, equipped with the product topology, which makes it a compact metrizable space. A point $x \in \calA^{\Z}$ is a bi-infinite sequence $x = (\ldots, x_{-1}, x_0, x_1, \ldots)$ with $x_n \in \calA$ for all $n \in \Z$. The left shift map $\sigma: \calA^{\Z} \to \calA^{\Z}$ is defined by $(\sigma x)_n = x_{n+1}$, which is a homeomorphism of $\calA^{\Z}$.

Given a transition matrix $A = (A_{ij})_{i,j \in \calA}$ with entries in $\{0, 1\}$, the subshift of finite type determined by $A$ is the closed $\sigma$-invariant subset
\begin{equation}\label{eq:SFT_def}
\SigmaA = \{x \in \calA^{\Z} : A_{x_n, x_{n+1}} = 1 \text{ for all } n \in \Z\}.
\end{equation}
The restriction of $\sigma$ to $\SigmaA$ is again denoted by $\sigma$. We say that $(\SigmaA, \sigma)$ is topologically mixing if for every pair of symbols $i, j \in \calA$, there exists $M = M(i,j) \in \N$ such that $(A^m)_{ij} > 0$ for all $m \geq M$. Equivalently, there exists $M \in \N$ such that $A^M$ has all entries strictly positive. Topological mixing is equivalent to the condition that for any two non-empty open sets $U, V \subset \SigmaA$, there exists $M$ such that $\sigma^{-m}(U) \cap V \neq \emptyset$ for all $m \geq M$.

The one-sided shift space $\SigmaA^+ = \{x \in \calA^{\N_0} : A_{x_n, x_{n+1}} = 1 \text{ for all } n \geq 0\}$ plays an important role in the spectral theory, as the transfer operator acts naturally on functions defined on $\SigmaA^+$. The natural projection $\pi^+: \SigmaA \to \SigmaA^+$ defined by $(\pi^+ x)_n = x_n$ for $n \geq 0$ is a factor map.


The spectral theory of the transfer operator requires careful specification of the function spaces on which it acts. We work primarily with spaces of H\"{o}lder continuous functions and, more generally, with functions satisfying summability conditions on their variations.

For $\alpha \in (0, 1]$, we define a metric $d_\alpha$ on $\SigmaA$ by
\begin{equation}\label{eq:metric_def}
d_\alpha(x, y) = \alpha^{\min\{|n| : x_n \neq y_n\}},
\end{equation}
with the convention that $d_\alpha(x, x) = 0$ and $\min \emptyset = +\infty$, so $\alpha^{+\infty} = 0$. This metric generates the product topology on $\SigmaA$.

\begin{definition}[H\"{o}lder Spaces]\label{def:holder_spaces}
For $\alpha \in (0, 1)$, the space $\calH_\alpha(\SigmaA)$ consists of all continuous functions $\phi: \SigmaA \to \R$ for which the H\"{o}lder seminorm
\begin{equation}\label{eq:holder_seminorm}
|\phi|_\alpha = \sup_{x \neq y} \frac{|\phi(x) - \phi(y)|}{d_\alpha(x, y)}
\end{equation}
is finite. The H\"{o}lder norm is defined by $\|\phi\|_\alpha = \|\phi\|_\infty + |\phi|_\alpha$, where $\|\phi\|_\infty = \sup_{x \in \SigmaA} |\phi(x)|$ denotes the supremum norm. With this norm, $\calH_\alpha(\SigmaA)$ is a Banach space.
\end{definition}

For functions depending only on future coordinates (i.e., functions on $\SigmaA^+$), we define the variation at scale $n$ by
\begin{equation}\label{eq:variation_def}
\var_n(\phi) = \sup\{|\phi(x) - \phi(y)| : x_k = y_k \text{ for } 0 \leq k \leq n-1\}.
\end{equation}
The condition $\phi \in \calH_\alpha(\SigmaA^+)$ is equivalent to $\var_n(\phi) \leq C \alpha^n$ for some constant $C > 0$ and all $n \geq 0$.

\begin{definition}[Summable Variation]\label{def:summable_variation}
The space $\calF_A$ consists of all continuous functions $\phi: \SigmaA^+ \to \R$ for which
\begin{equation}\label{eq:summable_var}
\|\phi\|_{\calF} = \|\phi\|_\infty + \sum_{n=0}^\infty \var_n(\phi) < \infty.
\end{equation}
This is a Banach space containing $\calH_\alpha(\SigmaA^+)$ for all $\alpha \in (0, 1)$.
\end{definition}

The summable variation condition is satisfied by all H\"{o}lder continuous functions and allows for slower decay of variations.

\subsection{The Gibbs Property: Jacobian and Classical Formulations}

Let $\mu$ be a $\sigma$-invariant probability measure on $\SigmaA$. Since $\sigma$ is a local homeomorphism on $\SigmaA^+$, for any measurable set $E \subset \SigmaA^+$ on which $\sigma$ is injective, the Jacobian of $\mu$ with respect to $\sigma$ is the Radon-Nikodym derivative
\begin{equation}\label{eq:jacobian_def}
J_\mu \sigma(x) = \frac{d(\mu \circ \sigma^{-1})}{d\mu}(x),
\end{equation}
where $\mu \circ \sigma^{-1}$ denotes the push-forward of $\mu$ under $\sigma$. The Jacobian is defined $\mu$-almost everywhere and is unique up to $\mu$-null sets.

\begin{definition}[Intrinsic Gibbs Measure]\label{def:intrinsic_gibbs}
Let $\phi \in \calF_A$ be a potential with summable variations, and let $P(\phi)$ denote the topological pressure of $\phi$ (defined precisely in Section~\ref{sec:variational}). A $\sigma$-invariant Borel probability measure $\mu$ on $\SigmaA$ is called an intrinsic Gibbs measure for $\phi$ if its Jacobian satisfies
\begin{equation}\label{eq:intrinsic_gibbs_condition}
J_\mu \sigma(x) = e^{P(\phi) - \phi(x)} \quad \text{for } \mu\text{-almost every } x \in \SigmaA.
\end{equation}
\end{definition}

The condition \eqref{eq:intrinsic_gibbs_condition} says that the measure transforms under the shift according to an exponential weight determined by the potential and the pressure. It is equivalent to Keane's g-measure condition \cite{Keane1972} with $g(x) = e^{-(P(\phi)-\phi(x))}$.

For comparison and to establish the equivalence, we recall the classical definition of Gibbs measures in terms of cylinder sets.

\begin{definition}[Cylinder Sets]\label{def:cylinders}
For an admissible word $w = (w_0, w_1, \ldots, w_{n-1}) \in \calA^n$ (meaning $A_{w_i, w_{i+1}} = 1$ for $0 \leq i < n-1$), the cylinder set determined by $w$ at position $k$ is
\begin{equation}
[w]_k = \{x \in \SigmaA : x_{k+i} = w_i \text{ for } 0 \leq i < n\}.
\end{equation}
When $k = 0$, we write simply $[w] = [w]_0$.
\end{definition}

\begin{definition}[Classical Gibbs Property]\label{def:classical_gibbs}
A $\sigma$-invariant Borel probability measure $\mu$ on $\SigmaA$ satisfies the Gibbs property for $\phi \in \calF_A$ with pressure $P = P(\phi)$ if there exist constants $C_1, C_2 > 0$ such that for all $n \geq 1$, all admissible words $w$ of length $n$, and all $x \in [w]$,
\begin{equation}\label{eq:gibbs_bounds}
C_1 \leq \frac{\mu([w])}{\exp(-nP + S_n\phi(x))} \leq C_2,
\end{equation}
where $S_n\phi(x) = \sum_{k=0}^{n-1} \phi(\sigma^k x)$ denotes the Birkhoff sum.
\end{definition}

The Gibbs property asserts that the measure of a cylinder set is comparable, up to bounded multiplicative constants, to the exponential of the Birkhoff sum of the potential minus the pressure contribution. This characterization connects directly to the Gibbs distributions of statistical mechanics, where the weight of a configuration is proportional to the exponential of its energy.

The Jacobian condition (Definition~\ref{def:intrinsic_gibbs})
determines the measure locally through the infinitesimal ratio
$\mu([x_0\cdots x_{n-1}])/\mu([x_1\cdots x_{n-1}])$, while the
classical Gibbs property (Definition~\ref{def:classical_gibbs})
controls the measure globally through uniform bounds on all
cylinders. The implication from the Jacobian condition to the
classical Gibbs property follows from the bounded distortion of
Birkhoff sums within a single cylinder; the reverse implication is
more delicate and requires showing that the bounded error term in
the Gibbs ratio is exactly zero. Both directions are proved in
Section~\ref{sec:intrinsic_gibbs} with explicit constants depending
only on the total variation $V(\phi)$.

\subsection{Statement of the Main Equivalence Theorem}

We are now in a position to state the central result of this Part, which establishes the equivalence of five characterizations of Gibbs measures.

\begin{maintheorem}[Spectral-Variational-Geometric Equivalence]\label{thm:main_equivalence}
Let $(\SigmaA, \sigma)$ be a topologically mixing subshift of finite type over a finite alphabet, and let $\phi \in \calF_A$ be a potential with summable variations. Let $\calL_\phi$ denote the Ruelle transfer operator associated to $\phi$, and let $P(\phi)$ denote the topological pressure. Then the following conditions on a $\sigma$-invariant Borel probability measure $\mu$ are equivalent:

\begin{enumerate}
\item[(i)] \textbf{(Jacobian Characterization)} The measure $\mu$ is an intrinsic Gibbs measure for $\phi$: its Jacobian satisfies $J_\mu \sigma = e^{P(\phi) - \phi}$ $\mu$-almost everywhere.

\item[(ii)] \textbf{(Classical Gibbs Property)} The measure $\mu$ satisfies the Gibbs bounds \eqref{eq:gibbs_bounds} with constants $C_1, C_2$ depending only on $\|\phi\|_{\calF}$.

\item[(iii)] \textbf{(Spectral Characterization)} The measure $\mu$ is the unique probability measure satisfying $\calL_\phi^* \nu = \lambda \nu$ for some $\lambda > 0$, where $\calL_\phi^*$ denotes the dual of the transfer operator acting on the space of Borel measures on $\SigmaA^+$. Moreover, $\lambda = e^{P(\phi)}$.

\item[(iv)] \textbf{(Variational Characterization)} The measure $\mu$ is the unique equilibrium state for $\phi$: it is the unique $\sigma$-invariant probability measure achieving the supremum in the variational principle
\begin{equation}\label{eq:variational_principle}
P(\phi) = \sup_{\nu \in \calM_\sigma(\SigmaA)} \left\{ h_\nu(\sigma) + \int_{\SigmaA} \phi \, d\nu \right\},
\end{equation}
where $\calM_\sigma(\SigmaA)$ denotes the space of $\sigma$-invariant Borel probability measures and $h_\nu(\sigma)$ denotes the measure-theoretic entropy.

\item[(v)] \textbf{(Large Deviations Characterization)} The measure $\mu$ is the unique zero of the rate function in the large deviations principle for empirical measures: for any continuous function $f: \SigmaA \to \R$,
\begin{equation}\label{eq:LDP_characterization}
\lim_{n \to \infty} \frac{1}{n} \log \mu\left\{ x : \left| \frac{1}{n} S_n f(x) - \int f \, d\mu \right| > \epsilon \right\} = -I_\phi(\epsilon, f) < 0
\end{equation}
for all $\epsilon > 0$, where the rate function $I_\phi$ is expressed in terms of the pressure.
\end{enumerate}

Furthermore, the Gibbs measure $\mu$ satisfying these equivalent conditions is unique, ergodic, and mixing. The constants in the Gibbs bounds can be taken as $C_1 = e^{-2\|\phi\|_{\calF}}$ and $C_2 = e^{2\|\phi\|_{\calF}}$.
\end{maintheorem}

The proof of Main Theorem~\ref{thm:main_equivalence} is developed throughout Sections~\ref{sec:intrinsic_gibbs}--\ref{sec:statistical}. The implications $(iii) \Rightarrow (i) \Rightarrow (ii)$ are established in Section~\ref{sec:intrinsic_gibbs}. The implication $(ii) \Rightarrow (iii)$ follows from the spectral analysis of the transfer operator in Sections~\ref{sec:transfer_operator}--\ref{sec:RPF}. The equivalence $(ii) \Leftrightarrow (iv)$ is the content of Section~\ref{sec:variational}. The characterization $(v)$ is established in Section~\ref{sec:statistical}.\newpage 

\begin{figure}[ht]
\centering
\begin{tikzpicture}[>=stealth, font=\small,
  charbox/.style={draw, rounded corners=3pt, minimum width=2.5cm, 
    minimum height=0.9cm, align=center, inner sep=4pt}]
  \node[charbox] (jac)  at (0,4)    {(i) Jacobian\\$J_\mu\sigma = e^{P(\phi)-\phi}$};
  \node[charbox] (cyl)  at (6,4)    {(ii) Cylinder Gibbs\\$C_1 \leq \frac{\mu[w]}{e^{S_n\phi - nP}} \leq C_2$};
  \node[charbox] (eig)  at (10,1.8) {(iii) Eigenmeasure\\$\calL_\phi^* \nu = \lambda\nu$};
  \node[charbox] (var)  at (6,-0.4) {(iv) Equilibrium state\\$h_\mu + \!\int\!\phi\,d\mu = P(\phi)$};
  \node[charbox] (ldp)  at (0,-0.4) {(v) LDP minimizer\\$I_\psi(\bar\psi) = 0$};
  \draw[->, thick] (jac) -- (cyl)
    node[midway, above, font=\footnotesize] {\S\ref{sec:intrinsic_gibbs}};
  \draw[->, thick] (cyl) -- (jac)
    node[midway, below, font=\footnotesize] {\S\ref{sec:intrinsic_gibbs}};
  \draw[->, thick] (cyl) -- (eig)
    node[midway, above, font=\footnotesize, sloped] {\S\S\ref{sec:transfer_operator}--\ref{sec:RPF}};
  \draw[->, thick] (eig) -- (cyl)
    node[midway, below, font=\footnotesize, sloped] {\S\ref{sec:RPF}};
  \draw[->, thick] (eig) -- (var)
    node[midway, right, font=\footnotesize] {\S\ref{sec:variational}};
  \draw[->, thick] (var) -- (eig)
    node[midway, left, font=\footnotesize] {\S\ref{sec:variational}};
  \draw[->, thick] (var) -- (ldp)
    node[midway, above, font=\footnotesize] {\S\ref{sec:statistical}};
  \draw[->, thick] (ldp) -- (var)
    node[midway, below, font=\footnotesize] {\S\ref{sec:statistical}};
  \node[font=\footnotesize, align=center] at (5,1.8) {all five\\equivalent\\(Main Theorem~\ref{thm:main_equivalence})};
\end{tikzpicture}
\caption{Logical structure of the Main Equivalence Theorem (Theorem~\ref{thm:main_equivalence}). Each box represents one of the five characterizations of Gibbs measures. Arrows indicate implications proved in the indicated sections. All five characterizations are equivalent for H\"{o}lder potentials on mixing subshifts of finite type}
\label{fig:five_way}
\end{figure}
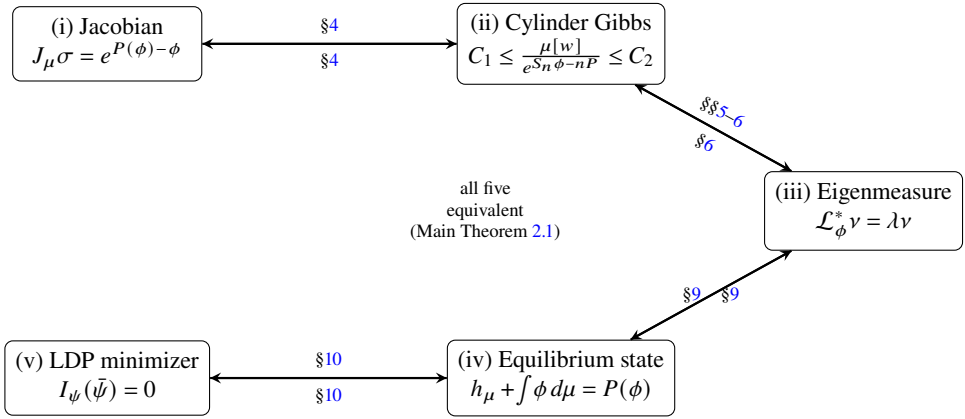
\subsection{Spectral Properties of the Transfer Operator}

The spectral theory of the Ruelle transfer operator underlies the entire development. We summarize the main spectral results here.

\begin{theorem}[Ruelle-Perron-Frobenius]\label{thm:RPF_main}
Let $(\SigmaA, \sigma)$ be a topologically mixing subshift of finite type, and let $\phi \in \calF_A$. The transfer operator $\calL_\phi: \calH_\alpha(\SigmaA^+) \to \calH_\alpha(\SigmaA^+)$ defined by
\begin{equation}\label{eq:transfer_def}
(\calL_\phi g)(x) = \sum_{y : \sigma y = x} e^{\phi(y)} g(y)
\end{equation}
has the following properties:

\begin{enumerate}
\item[(a)] There exists a unique $\lambda = e^{P(\phi)} > 0$ such that $\calL_\phi h = \lambda h$ for a strictly positive function $h \in \calH_\alpha(\SigmaA^+)$ with $h > 0$ everywhere.

\item[(b)] There exists a unique Borel probability measure $\nu$ on $\SigmaA^+$ such that $\calL_\phi^* \nu = \lambda \nu$, where $\calL_\phi^*$ is the dual operator.

\item[(c)] The eigenvalue $\lambda$ is simple and is the unique eigenvalue of modulus $\lambda$. The rest of the spectrum is contained in a disk $\{z \in \C : |z| \leq \gamma \lambda\}$ for some $\gamma \in (0, 1)$.

\item[(d)] For any $g \in \calH_\alpha(\SigmaA^+)$, we have
\begin{equation}\label{eq:RPF_convergence}
\left\| \lambda^{-n} \calL_\phi^n g - \nu(g) h \right\|_\alpha \leq C \gamma^n \|g\|_\alpha
\end{equation}
for some constants $C > 0$ and $\gamma \in (0, 1)$ depending only on $\phi$ and $\alpha$.
\end{enumerate}
\end{theorem}

The Ruelle-Perron-Frobenius theorem provides the spectral foundation for the entire paper: the eigendata $(\lambda, h, \nu)$ determine the unique Gibbs measure $\mu_\phi = h\nu$, and the exponential convergence rate $\gamma$ governs all statistical properties. The proofs of parts (a)--(d) are given in Section~\ref{sec:RPF}: existence of the eigendata is restated as Theorem~\ref{thm:eigendata_existence}, uniqueness as Theorem~\ref{thm:eigendata_uniqueness}, and exponential convergence as Theorem~\ref{thm:exponential_convergence}. The spectral gap bound (part of the convergence rate $\gamma$) is proved in Section~\ref{sec:spectral_gap} and restated as Theorem~\ref{thm:spectral_gap}.

\begin{remark}[Explicit Constants]\label{rem:explicit_constants}
The constants $C$ and $\gamma$ in Theorem~\ref{thm:RPF_main}(d) can be computed explicitly from the data $(\alpha, \|\phi\|_\alpha, N, M)$. Specifically, define:
\begin{itemize}
\item[] $B_m = \exp\left(\sum_{k=m+1}^\infty 2\var_k\phi\right)$, $B_0 = \exp\left(\frac{2b\alpha}{1-\alpha}\right)$ where $\var_k\phi \leq b\alpha^k$;
\item[] $K = \lambda^M e^{M\|\phi\|_\infty} B_0$ (the cone diameter bound);
\item[] $\eta \leq \frac{u_2(1-\alpha)}{4\alpha u_1 \|h\|_\infty K}$ (the cone contraction coefficient, cf.\  \cite[Lemma~1.9]{Bowen1975});
\item[] $\gamma = \max\!\left(\alpha^{1/3},\, (1-\eta)^{1/(3M)}\right)$ (the spectral gap rate);
\item[] $C = (1-\eta)^{-1}(\|h\|_\infty + K)\sup_{0 \leq r < M}\|\lambda^{-r}\calL_\phi^r\|$ (the convergence constant).
\end{itemize}
For the numerical example in Section~\ref{sec:numerical} ($N=2$, $M=1$, $\alpha=1/2$, $\phi$ depending only on $x_0$), these simplify to $\gamma = 1/2$ and $C = O(1)$.
\end{remark}

\subsection{Perturbation Theory and Statistical Properties}

The dependence of the spectral data on the potential is described by the following result.

\begin{theorem}[Analytic Dependence on Potential]\label{thm:analytic_dependence}
Let $(\SigmaA, \sigma)$ be a topologically mixing subshift of finite type. The maps
\begin{align}
\phi &\mapsto P(\phi) \label{eq:pressure_map}\\
\phi &\mapsto h_\phi \label{eq:eigenfunction_map}\\
\phi &\mapsto \nu_\phi \label{eq:eigenmeasure_map}\\
\phi &\mapsto \mu_\phi \label{eq:gibbs_map}
\end{align}
are real-analytic on the Banach space $\calH_\alpha(\SigmaA^+)$, where $h_\phi$ denotes the normalized eigenfunction, $\nu_\phi$ denotes the eigenmeasure, and $\mu_\phi = h_\phi \nu_\phi$ denotes the Gibbs measure.

Moreover, the derivatives are given by:
\begin{align}
\frac{d}{dt}\bigg|_{t=0} P(\phi + t\psi) &= \int_{\SigmaA} \psi \, d\mu_\phi, \label{eq:pressure_derivative}\\
\frac{d^2}{dt^2}\bigg|_{t=0} P(\phi + t\psi) &= \lim_{n \to \infty} \frac{1}{n} \operatorname{Var}_{\mu_\phi}(S_n \psi), \label{eq:pressure_second_derivative}
\end{align}
where $\operatorname{Var}_{\mu_\phi}$ denotes the variance with respect to $\mu_\phi$.
\end{theorem}

The analytic dependence of the pressure on the potential underlies the perturbation theory and the CLT: the first derivative identifies the equilibrium state, and the second derivative gives the asymptotic variance of Birkhoff sums. The proof is given in Section~\ref{sec:perturbation} using Kato's analytic perturbation theory \cite{Kato1980} applied to the family $\calL_{\phi+t\psi}$; the analyticity statement is restated as Theorem~\ref{thm:pressure_analyticity}, and the explicit derivative formulas as Corollary~\ref{cor:pressure_derivatives}.

The second derivative formula \eqref{eq:pressure_second_derivative} shows that the pressure is strictly convex in directions corresponding to observables that are not cohomologous to constants.

\begin{theorem}[Lipschitz Stability]\label{thm:lipschitz_stability}
Let $W_1$ denote the Wasserstein-$1$ metric on the space of probability measures on $\SigmaA$. For potentials $\phi, \psi \in \calH_\alpha(\SigmaA^+)$, the corresponding Gibbs measures satisfy
\begin{equation}\label{eq:wasserstein_bound}
W_1(\mu_\phi, \mu_\psi) \leq C \|\phi - \psi\|_\infty,
\end{equation}
where $C$ depends on the spectral gap for $\phi$ and $\psi$.
\end{theorem}

Lipschitz stability quantifies the continuous dependence of the Gibbs measure on the potential: small perturbations in the potential produce small changes in the equilibrium state. The proof is given in Section~\ref{sec:perturbation} and integrates the correlation decay along a path of potentials; the quantitative version with explicit constants is restated as Theorem~\ref{thm:lipschitz_stability_full}.

The statistical behavior of Birkhoff sums under Gibbs measures is characterized by the following results.

\begin{theorem}[Exponential Decay of Correlations]\label{thm:correlation_main}
Let $\mu = \mu_\phi$ be the Gibbs measure for $\phi \in \calH_\alpha(\SigmaA^+)$. For any $f, g \in \calH_\alpha(\SigmaA)$, the correlation function
\begin{equation}
C_n(f, g) = \int_{\SigmaA} f \cdot (g \circ \sigma^n) \, d\mu - \int_{\SigmaA} f \, d\mu \int_{\SigmaA} g \, d\mu
\end{equation}
satisfies
\begin{equation}\label{eq:correlation_bound}
|C_n(f, g)| \leq C \|f\|_\alpha \|g\|_\alpha \gamma^n
\end{equation}
for some $C > 0$ and $\gamma \in (0, 1)$ depending only on $\phi$ and $\alpha$.
\end{theorem}

Exponential mixing is the dynamical consequence of the spectral gap: the correlation between a function and its time-translate decays at the same geometric rate $\gamma$ that governs the convergence of transfer operator iterates. The proof is given in Section~\ref{sec:spectral_gap} as a duality consequence of the spectral gap (Theorem~\ref{thm:spectral_gap}); the statement is restated as Corollary~\ref{cor:exponential_mixing}.

\begin{theorem}[Central Limit Theorem with Berry-Esseen Bounds]\label{thm:CLT_main}
Let $\mu = \mu_\phi$ be the Gibbs measure for $\phi \in \calH_\alpha(\SigmaA^+)$, and let $\psi \in \calH_\alpha(\SigmaA)$ be an observable that is not cohomologous to a constant. Define
\begin{equation}
\xi^2 = \lim_{n \to \infty} \frac{1}{n} \operatorname{Var}_\mu(S_n \psi) > 0.
\end{equation}
Then the normalized Birkhoff sums converge in distribution to a standard normal:
\begin{equation}
\frac{S_n \psi - n \int \psi \, d\mu}{\xi \sqrt{n}} \xrightarrow{d} \mathcal{N}(0, 1).
\end{equation}
Moreover, the rate of convergence satisfies the Berry-Esseen bound:
\begin{equation}\label{eq:berry_esseen}
\sup_{t \in \R} \left| \mu\left( \frac{S_n \psi - n \int \psi \, d\mu}{\xi \sqrt{n}} \leq t \right) - \Phi(t) \right| \leq \frac{C}{\sqrt{n}},
\end{equation}
where $\Phi$ is the standard normal distribution function and $C$ depends on $\phi$, $\psi$, and $\alpha$.
\end{theorem}

The central limit theorem and Berry-Esseen bound establish that Birkhoff sums under the Gibbs measure satisfy the same distributional approximation as sums of independent random variables, with an explicit convergence rate. The proofs are given in Section~\ref{sec:statistical} using the Nagaev-Guivarc'h spectral perturbation method applied to $\calL_{\phi+is\psi}$; the CLT is restated as Theorem~\ref{thm:CLT}, and the Berry-Esseen bound as Theorem~\ref{thm:Berry_Esseen}.

\begin{theorem}[Large Deviations Principle]\label{thm:LDP_main}
Let $\mu = \mu_\phi$ be the Gibbs measure for $\phi \in \calH_\alpha(\SigmaA^+)$, and let $\psi \in \calH_\alpha(\SigmaA)$. For any interval $[a, b] \subset \R$,
\begin{equation}\label{eq:LDP_bounds}
\begin{split}
-\inf_{t \in (a,b)} I_\psi(t) & \leq \liminf_{n \to \infty} \frac{1}{n} \log \mu\left( a \leq \frac{S_n \psi}{n} \leq b \right) \\
& \leq \limsup_{n \to \infty} \frac{1}{n} \log \mu\left( a \leq \frac{S_n \psi}{n} \leq b \right) \\
& \leq -\inf_{t \in [a,b]} I_\psi(t),
\end{split}
\end{equation}
where the rate function is given by
\begin{equation}\label{eq:rate_function}
I_\psi(t) = \sup_{s \in \R} \left\{ st - P(\phi + s\psi) + P(\phi) \right\}.
\end{equation}
The rate function $I_\psi$ is convex, lower semi-continuous, and achieves its unique minimum value of zero at $t = \int \psi \, d\mu$.
\end{theorem}

The large deviations principle quantifies the exponential rate at which empirical averages deviate from their expected values, with the rate function expressed as a Legendre transform of the pressure. The proof is given in Section~\ref{sec:statistical} by applying the G\"{a}rtner-Ellis theorem to the cumulant generating function $\Lambda(s) = P(\phi+s\psi)-P(\phi)$, whose differentiability follows from the analyticity of the pressure established in Section~\ref{sec:perturbation}; the statement is restated as Theorem~\ref{thm:LDP}.


\section{Symbolic Dynamics and Subshifts of Finite Type}\label{sec:symbolic_dynamics}

This section establishes notation and technical lemmas for subshifts of finite type. The material is classical; see  \cite{LindMarcus1995} for the combinatorial aspects and  \cite{Bowen1975} or  \cite{ParryPollicott1990} for the dynamical theory.

\subsection{Subshifts of Finite Type and Topological Mixing}

The subshifts of finite type considered in this Part arise from imposing local constraints on sequences of symbols. These systems model hyperbolic dynamics through Markov partitions, as established by  \cite{AdlerWeiss1970},  \cite{Sinai1968}, and  \cite{Bowen1970,Bowen1975}.

Let $\calA = \{1, 2, \ldots, N\}$ be a finite alphabet with $N \geq 2$ symbols. The full shift space $\calA^{\Z}$ consists of all bi-infinite sequences $x = (x_n)_{n \in \Z}$ with $x_n \in \calA$, equipped with the product topology derived from the discrete topology on $\calA$. This topology is metrizable, compact, and totally disconnected. The left shift map $\sigma: \calA^{\Z} \to \calA^{\Z}$ defined by $(\sigma x)_n = x_{n+1}$ is a homeomorphism that commutes with the natural $\Z$-action on sequences by translation.

A transition matrix (or adjacency matrix) is an $N \times N$ matrix $A = (A_{ij})_{i,j \in \calA}$ with entries in $\{0, 1\}$. The entry $A_{ij}$ encodes whether the transition from symbol $i$ to symbol $j$ is allowed: $A_{ij} = 1$ means that the symbol $j$ may immediately follow the symbol $i$ in any allowed sequence, while $A_{ij} = 0$ means this transition is forbidden. The subshift of finite type (also called a topological Markov chain) determined by $A$ is the closed, $\sigma$-invariant subset
\begin{equation}
\SigmaA = \{x \in \calA^{\Z} : A_{x_n, x_{n+1}} = 1 \text{ for all } n \in \Z\}.
\end{equation}
The restriction of the shift map to $\SigmaA$ is again denoted by $\sigma: \SigmaA \to \SigmaA$.

The assumption that the matrix $A$ has at least one entry equal to $1$ in each row and each column ensures that $\SigmaA$ is non-empty and that $\sigma$ is surjective. We shall always assume this condition, which is equivalent to requiring that every symbol appears in at least one allowed bi-infinite sequence.

The matrix $A$ naturally defines a directed graph $G_A$ with vertex set $\calA$ and an edge from $i$ to $j$ if and only if $A_{ij} = 1$. An admissible word of length $n$ is a finite sequence $w = (w_0, w_1, \ldots, w_{n-1}) \in \calA^n$ such that $A_{w_k, w_{k+1}} = 1$ for all $0 \leq k < n-1$. Equivalently, $w$ corresponds to a directed path of length $n-1$ in the graph $G_A$. We denote by $\calW_n(A)$ the set of all admissible words of length $n$, and we note that $|\calW_n(A)| = \sum_{i,j} (A^{n-1})_{ij}$ counts the number of directed paths of length $n-1$ in $G_A$.

\begin{definition}[Cylinder Sets]
For an admissible word $w = (w_0, \ldots, w_{n-1}) \in \calW_n(A)$ and an integer $k \in \Z$, the cylinder set based at $w$ and position $k$ is
\begin{equation}
[w]_k = \{x \in \SigmaA : x_{k+j} = w_j \text{ for } 0 \leq j < n\}.
\end{equation}
When $k = 0$, we write simply $[w] = [w]_0$ and call this the cylinder set determined by $w$. The length of the cylinder is $n$, denoted $|w| = n$.
\end{definition}

Cylinder sets form a basis for the topology of $\SigmaA$: a set is open if and only if it is a union of cylinder sets. Each cylinder $[w]_k$ is both open and closed (clopen) in $\SigmaA$. For a fixed position $k$, the collection $\{[w]_k : w \in \calW_n(A)\}$ forms a partition of $\SigmaA$ into clopen sets.

\begin{lemma}[Cylinder Set Properties]\label{lem:cylinder_properties}
The cylinder sets satisfy the following properties:
\begin{enumerate}
\item[(a)] $\sigma([w]_k) = [w]_{k-1}$ for any cylinder $[w]_k$.
\item[(b)] $\sigma^{-1}([w]_k) = \bigcup_{i : A_{i, w_0} = 1} [iw]_k$, where $iw$ denotes the concatenation.
\item[(c)] If $w = (w_0, \ldots, w_{n-1})$ and $v = (w_1, \ldots, w_{n-1})$ is the word obtained by deleting the first symbol, then $[w]_k \subset [v]_{k+1}$.
\item[(d)] Two cylinders $[w]_k$ and $[v]_\ell$ are either disjoint or one contains the other.
\end{enumerate}
\end{lemma}

\begin{proof}
Property (a) follows directly from the definition: $x \in [w]_k$ means $x_{k+j} = w_j$ for $0 \leq j < n$, which is equivalent to $(\sigma x)_{k-1+j} = w_j$ for $0 \leq j < n$, i.e., $\sigma x \in [w]_{k-1}$.

For property (b), note that $x \in \sigma^{-1}([w]_k)$ if and only if $\sigma x \in [w]_k$, which means $x_{k+1+j} = w_j$ for $0 \leq j < n$. This is equivalent to $x \in [iw]_k$ for some symbol $i = x_k$ with $A_{i, w_0} = 1$.

Property (c) is immediate: the condition $x_{k+j} = w_j$ for $0 \leq j < n$ implies $x_{k+1+j} = w_{j+1}$ for $0 \leq j < n-1$, which is the condition for $x \in [v]_{k+1}$.

For property (d), suppose $[w]_k \cap [v]_\ell \neq \emptyset$. Let $x$ be a point in the intersection. If $k \leq \ell$, then $x_{k+j} = w_j$ for $0 \leq j < |w|$ and $x_{\ell+j} = v_j$ for $0 \leq j < |v|$. If $\ell + |v| \leq k + |w|$, then every point $y$ with $y_{k+j} = w_j$ automatically satisfies $y_{\ell+j} = v_j$, so $[w]_k \subset [v]_\ell$. The other cases are handled similarly.
\end{proof}

The key structural assumption on the transition matrix is topological mixing, which ensures the irreducibility needed for the spectral theory of the transfer operator.

\begin{definition}[Topological Mixing]
The subshift $(\SigmaA, \sigma)$ is topologically mixing if for any two non-empty open sets $U, V \subset \SigmaA$, there exists $M \in \N$ such that $\sigma^{-m}(U) \cap V \neq \emptyset$ for all $m \geq M$.
\end{definition}

The condition of topological mixing has a simple characterization in terms of the transition matrix $A$, which connects the dynamical property to the algebraic structure.

\begin{proposition}[Characterization of Mixing]\label{prop:mixing_characterization}
The following conditions are equivalent:
\begin{enumerate}
\item[(i)] The subshift $(\SigmaA, \sigma)$ is topologically mixing.
\item[(ii)] There exists $M \in \N$ such that $(A^m)_{ij} > 0$ for all $i, j \in \calA$ and all $m \geq M$.
\item[(iii)] The matrix $A$ is primitive (i.e., $A$ is irreducible and aperiodic).
\item[(iv)] For any two symbols $i, j \in \calA$, the greatest common divisor of the set $\{n \geq 1 : (A^n)_{ij} > 0\}$ equals $1$.
\end{enumerate}
\end{proposition}

\begin{proof}
The equivalence $(ii) \Leftrightarrow (iii) \Leftrightarrow (iv)$ is the standard characterization of primitive matrices from the Perron-Frobenius theory; see for instance  \cite{Seneta2006} or  \cite{HornJohnson2013}.

$(i) \Rightarrow (ii)$: Let $U = [i]$ and $V = [j]$ be single-symbol cylinders. By topological mixing, there exists $M$ such that $\sigma^{-m}([i]) \cap [j] \neq \emptyset$ for all $m \geq M$. This means there exists a sequence $x \in \SigmaA$ with $x_0 = j$ and $x_m = i$, which is equivalent to $(A^m)_{ji} > 0$.

$(ii) \Rightarrow (i)$: Let $U$ and $V$ be non-empty open sets. Choose cylinders $[w] \subset U$ and $[v] \subset V$. Let $i = w_{|w|-1}$ be the last symbol of $w$ and $j = v_0$ be the first symbol of $v$. By assumption, $(A^m)_{ij} > 0$ for all $m \geq M$, meaning there is a path from $i$ to $j$ of length $m$. For $m \geq M$, choose an admissible word $u$ of length $m$ starting at $i$ and ending at $j$. Then the cylinder $[wuv]$ is non-empty, contained in $\sigma^{-(|w|+m)}(U) \cap V$, showing that the intersection is non-empty.
\end{proof}

The primitivity of $A$ implies, by the Perron-Frobenius theorem, that $A$ has a unique eigenvalue $\rho(A) > 0$ of maximum modulus, this eigenvalue is simple, and the corresponding left and right eigenvectors can be chosen to have strictly positive entries. This spectral structure for the matrix $A$ foreshadows the spectral structure for the transfer operator, which we shall develop in subsequent sections.

\begin{lemma}[Uniform Mixing Times]\label{lem:uniform_mixing}
Suppose $(\SigmaA, \sigma)$ is topologically mixing with mixing time $M$ (i.e., $(A^m)_{ij} > 0$ for all $i, j$ and $m \geq M$). Then for any two admissible words $w, v$ and any $m \geq M$, there exists an admissible word $u$ of length $m$ such that the concatenation $wuv$ is admissible.
\end{lemma}

\begin{proof}
Let $i = w_{|w|-1}$ and $j = v_0$. Since $(A^m)_{ij} > 0$, there exists a path in $G_A$ from $i$ to $j$ of length $m$, which corresponds to an admissible word $u = (u_0, \ldots, u_{m-1})$ with $u_0 = i$ and $A_{u_{m-1}, j} = 1$. The word $u' = (u_1, \ldots, u_{m-1})$ has length $m-1$, and the concatenation $wuv$ (where we omit the repeated symbol $i$) is admissible. Adjusting the indexing gives the result.
\end{proof}

\subsection{Metrics, Function Spaces, and the Natural Extension}

The topological dynamics of $(\SigmaA, \sigma)$ is enriched by specifying a metric compatible with the product topology. The standard choice is a metric that assigns exponentially small diameter to cylinder sets of increasing length.

\begin{definition}[Standard Metrics]
For a parameter $\alpha \in (0, 1)$, define the metric $d_\alpha$ on $\SigmaA$ by
\begin{equation}
d_\alpha(x, y) = \alpha^{k(x,y)}, \quad \text{where } k(x, y) = \min\{|n| : n \in \Z, \, x_n \neq y_n\}
\end{equation}
with the conventions $\alpha^{+\infty} = 0$ and $d_\alpha(x, x) = 0$.
\end{definition}

All metrics $d_\alpha$ for $\alpha \in (0, 1)$ are equivalent and generate the product topology on $\SigmaA$. The diameter of a cylinder $[w]_k$ with $|w| = n$ satisfies $\diam_{d_\alpha}([w]_k) \leq \alpha^{\min(|k|, |k+n-1|)}$, so cylinders based at position $0$ have diameter at most $\alpha^0 = 1$ regardless of length.

For the analysis of the transfer operator, it is essential to work with the one-sided shift space
\begin{equation}
\SigmaA^+ = \{x \in \calA^{\N_0} : A_{x_n, x_{n+1}} = 1 \text{ for all } n \geq 0\},
\end{equation}
equipped with the metric $d_\alpha^+(x, y) = \alpha^{k^+(x,y)}$ where $k^+(x, y) = \min\{n \geq 0 : x_n \neq y_n\}$. The shift map $\sigma: \SigmaA^+ \to \SigmaA^+$ is now a continuous surjection that is $N$-to-$1$ at most points (where $N$ is the alphabet size).

\begin{definition}[H\"{o}lder Continuous Functions]
For $\alpha \in (0, 1)$, the space $\calH_\alpha = \calH_\alpha(\SigmaA^+)$ consists of all functions $\phi: \SigmaA^+ \to \R$ for which the H\"{o}lder seminorm
\begin{equation}
|\phi|_\alpha = \sup_{x \neq y} \frac{|\phi(x) - \phi(y)|}{d_\alpha^+(x, y)}
\end{equation}
is finite. Equipped with the norm $\|\phi\|_\alpha = \|\phi\|_\infty + |\phi|_\alpha$, the space $\calH_\alpha$ is a Banach algebra under pointwise multiplication.
\end{definition}

The H\"{o}lder condition can be reformulated in terms of the variation seminorms. For $\phi: \SigmaA^+ \to \R$ and $n \geq 0$, define
\begin{equation}
\var_n(\phi) = \sup\{|\phi(x) - \phi(y)| : x_k = y_k \text{ for } 0 \leq k < n\}.
\end{equation}
Then $\phi \in \calH_\alpha$ if and only if $\var_n(\phi) = O(\alpha^n)$, and more precisely $|\phi|_\alpha = \sup_{n \geq 0} \alpha^{-n} \var_n(\phi)$.

For certain applications, particularly in the extension to countable alphabets, it is useful to work with the larger space of functions with summable variations.

\begin{definition}[Summable Variation Space]
The space $\calF_A = \calF(\SigmaA^+)$ consists of all continuous functions $\phi: \SigmaA^+ \to \R$ for which
\begin{equation}
V(\phi) = \sum_{n=0}^\infty \var_n(\phi) < \infty.
\end{equation}
The norm $\|\phi\|_{\calF} = \|\phi\|_\infty + V(\phi)$ makes $\calF_A$ a Banach space.
\end{definition}

\begin{lemma}[Inclusions of Function Spaces]\label{lem:space_inclusions}
For any $0 < \beta < \alpha < 1$, we have the continuous inclusions
\begin{equation}
\calH_\beta(\SigmaA^+) \hookrightarrow \calH_\alpha(\SigmaA^+) \hookrightarrow \calF_A \hookrightarrow C(\SigmaA^+),
\end{equation}
where $C(\SigmaA^+)$ denotes the space of continuous functions with the supremum norm.
\end{lemma}

\begin{proof}
The first inclusion holds because $\beta < \alpha$ implies $\beta^n < \alpha^n$, so $|\phi|_\beta = \sup_n \beta^{-n} \var_n(\phi) \geq \sup_n \alpha^{-n} \var_n(\phi) = |\phi|_\alpha$.

For the second inclusion, if $\phi \in \calH_\alpha$, then $\var_n(\phi) \leq |\phi|_\alpha \alpha^n$, so
\begin{equation}
V(\phi) = \sum_{n=0}^\infty \var_n(\phi) \leq |\phi|_\alpha \sum_{n=0}^\infty \alpha^n = \frac{|\phi|_\alpha}{1 - \alpha}.
\end{equation}
Thus $\|\phi\|_{\calF} \leq \|\phi\|_\infty + \frac{|\phi|_\alpha}{1-\alpha} \leq \frac{1}{1-\alpha} \|\phi\|_\alpha$.

The third inclusion holds because $V(\phi) < \infty$ implies $\var_n(\phi) \to 0$, which is equivalent to uniform continuity, and $\SigmaA^+$ is compact, so continuous functions are bounded.
\end{proof}

The space $\calF_A$ is the natural setting for the general theory because it accommodates potentials arising from geometric constructions where the H\"{o}lder condition may fail. For instance, the potential $\phi(x) = -\log |Df(x)|$ associated to a smooth expanding map $f$ may have variations decaying like $n^{-p}$ rather than exponentially when the map has neutral periodic points.

The one-sided shift space $\SigmaA^+$ has the advantage that the transfer operator acts naturally on functions defined there, but it has the disadvantage that the shift map is not invertible. The two-sided shift space $\SigmaA$ is the natural extension of $\SigmaA^+$ in the sense of ergodic theory.

\begin{proposition}[Natural Extension]\label{prop:natural_extension}
The projection $\pi^+: \SigmaA \to \SigmaA^+$ defined by $(\pi^+ x)_n = x_n$ for $n \geq 0$ is a continuous surjection satisfying $\pi^+ \circ \sigma = \sigma \circ \pi^+$. For any $\sigma$-invariant probability measure $\mu^+$ on $\SigmaA^+$, there exists a unique $\sigma$-invariant probability measure $\mu$ on $\SigmaA$ such that $(\pi^+)_* \mu = \mu^+$. The measure $\mu$ is called the natural extension of $\mu^+$.
\end{proposition}

\begin{proof}
The continuity and equivariance of $\pi^+$ are immediate from the definitions. For the existence and uniqueness of the natural extension, the standard construction from ergodic theory applies; see  \cite{Rohlin1961} or  \cite{CornfeldFominSinai1982}. The key point is that $\mu$ is characterized by its values on cylinder sets $[w]_k$ for all $k \in \Z$, and these are determined by $\mu^+$ through the invariance condition.
\end{proof}

When working with Gibbs measures, we pass freely between the one-sided and two-sided settings, using the natural extension to transfer results. The thermodynamic properties (pressure, entropy, Gibbs bounds) are the same for a measure and its natural extension.

Having established the symbolic setting, we turn to the characterization of Gibbs measures. The next section defines the Jacobian condition and proves its equivalence with the classical cylinder Gibbs property.

\section{The Intrinsic Jacobian Characterization of Gibbs Measures}\label{sec:intrinsic_gibbs}

We characterize Gibbs measures through their Jacobian and establish the equivalence with the classical cylinder-based definition. The Jacobian characterization is closely related to Keane's g-measures \cite{Keane1972} and to Ledrappier's entropy formula \cite{Ledrappier1974}. The equivalence with the cylinder Gibbs property, stated as Theorem~\ref{thm:jacobian_equivalence} and proved below in this section, is the implication (i)$\Leftrightarrow$(ii) of Main Theorem~\ref{thm:main_equivalence}.

\subsection{The Jacobian and the Intrinsic Gibbs Condition}

Let $(X, \calB, \mu)$ be a probability space equipped with a measurable map $T: X \to X$ that is a local homeomorphism. For the one-sided shift $\sigma: \SigmaA^+ \to \SigmaA^+$, the restriction to each cylinder $\SigmaA^+(a) = \{x \in \SigmaA^+ : x_0 = a\}$ is a homeomorphism onto its image.

\begin{definition}[Jacobian of an Invariant Measure]\label{def:jacobian}
Let $(\SigmaA^+, \sigma, \mu)$ be a measure-preserving system where $\sigma: \SigmaA^+ \to \SigmaA^+$ is the one-sided shift. For a symbol $a \in \calA$, let $\SigmaA^+(a) = \{x \in \SigmaA^+ : x_0 = a\}$ denote the cylinder of sequences starting with $a$. The restriction $\sigma|_{\SigmaA^+(a)}: \SigmaA^+(a) \to \sigma(\SigmaA^+(a))$ is a homeomorphism onto its image. The Jacobian of $\mu$ with respect to $\sigma$ is the function $J_\mu: \SigmaA^+ \to (0, \infty)$ defined by
\begin{equation}\label{eq:jacobian_formula}
J_\mu(x) = \sum_{a: A_{a, x_0} = 1} \frac{d\mu|_{\SigmaA^+(a)}}{d(\sigma_* \mu|_{\SigmaA^+(a)})}(\sigma^{-1}_a x),
\end{equation}
where $\sigma^{-1}_a: \sigma(\SigmaA^+(a)) \to \SigmaA^+(a)$ is the local inverse defined by $\sigma^{-1}_a(x) = ax$ (the sequence with $a$ prepended to $x$), and the Radon-Nikodym derivative is taken with respect to the appropriate restrictions.
\end{definition}

This definition may appear complicated, but it has a simple interpretation: $J_\mu(x)$ measures the total mass that $\mu$ assigns to the preimages of $x$ under $\sigma$, weighted by the local behavior of $\mu$ near each preimage. The following lemma provides an equivalent formulation that is more directly useful.

\begin{lemma}[Jacobian Formula]\label{lem:jacobian_formula}
Let $\mu$ be a $\sigma$-invariant Borel probability measure on $\SigmaA^+$. For $\mu$-almost every $x \in \SigmaA^+$,
\begin{equation}\label{eq:jacobian_ratio}
J_\mu(x) = \lim_{n \to \infty} \frac{\mu([x_0 \cdots x_{n-1}])}{\mu([x_1 \cdots x_{n-1}])},
\end{equation}
where the limit exists and is independent of the choice of representative in the equivalence class of $x$.
\end{lemma}

\begin{proof}
For a cylinder $[w] = [w_0 \cdots w_{n-1}]$ with $n \geq 2$, let $[w'] = [w_1 \cdots w_{n-1}]$ denote the cylinder obtained by dropping the first symbol. By the definition of the shift, $\sigma([w]) = [w'] \cap \sigma(\SigmaA^+(w_0))$. The $\sigma$-invariance of $\mu$ implies that for any measurable set $E \subset \SigmaA^+$,
\begin{equation}
\mu(\sigma^{-1}(E)) = \sum_{a \in \calA} \mu(\sigma^{-1}(E) \cap \SigmaA^+(a)) = \sum_{a \in \calA} \mu(\sigma^{-1}(E \cap \sigma(\SigmaA^+(a))) \cap \SigmaA^+(a)).
\end{equation}
Taking $E = [w']$, we get $\mu([w]) = \mu(\sigma^{-1}([w']) \cap \SigmaA^+(w_0))$, which gives the relation
\begin{equation}
\frac{\mu([w])}{\mu([w'])} = \frac{\mu(\sigma^{-1}([w']) \cap \SigmaA^+(w_0))}{\mu([w'])}.
\end{equation}
As $n \to \infty$, the cylinders $[w']$ shrink to the point $\sigma(x)$ for $\mu$-almost every $x \in [w_0]$, and the Lebesgue differentiation theorem (in its measure-theoretic form) implies that the ratio converges to the Radon-Nikodym derivative, which is the Jacobian.
\end{proof}

The Jacobian has the following multiplicative property under iteration, which is the key to connecting it with Birkhoff sums of the potential.

\begin{lemma}[Chain Rule for the Jacobian]\label{lem:jacobian_chain}
For $\mu$-almost every $x \in \SigmaA^+$ and every $n \geq 1$,
\begin{equation}\label{eq:jacobian_chain}
J_\mu^{(n)}(x) := \frac{d(\sigma^n)_* \mu}{d\mu}(x) = \prod_{k=0}^{n-1} J_\mu(\sigma^k x).
\end{equation}
Equivalently, $\log J_\mu^{(n)}(x) = \sum_{k=0}^{n-1} \log J_\mu(\sigma^k x) = S_n(\log J_\mu)(x)$.
\end{lemma}

\begin{proof}
This follows from the chain rule for Radon-Nikodym derivatives: if $\nu_1 \ll \nu_2 \ll \nu_3$ are $\sigma$-finite measures, then $\frac{d\nu_1}{d\nu_3} = \frac{d\nu_1}{d\nu_2} \cdot \frac{d\nu_2}{d\nu_3}$ almost everywhere. Applying this with $\nu_1 = (\sigma^n)_* \mu$, $\nu_2 = (\sigma^{n-1})_* \mu$, and $\nu_3 = \mu$, and iterating, gives the result.
\end{proof}

We now formulate the Gibbs property in terms of the Jacobian. The key observation, which goes back to the work of  \cite{Ledrappier1974}, is that for a Gibbs measure, the Jacobian takes a specific exponential form.

\begin{definition}[Intrinsic Gibbs Measure]\label{def:intrinsic_gibbs_full}
Let $\phi \in \calF_A$ be a potential with summable variations. A $\sigma$-invariant Borel probability measure $\mu$ on $\SigmaA^+$ is called an intrinsic Gibbs measure for $\phi$ if there exists a constant $P \in \R$ such that
\begin{equation}\label{eq:jacobian_gibbs}
J_\mu(x) = e^{P - \phi(x)} \quad \text{for } \mu\text{-almost every } x \in \SigmaA^+.
\end{equation}
The constant $P$ is called the pressure associated to $\mu$ and $\phi$.
\end{definition}

The intrinsic Gibbs condition \eqref{eq:jacobian_gibbs} asserts that the logarithm of the Jacobian equals a constant minus the potential. This condition depends only on the measurable dynamics and the potential, not on any choice of coordinates.

\begin{proposition}[Uniqueness of Pressure]\label{thm:pressure_uniqueness}
If $\mu$ is an intrinsic Gibbs measure for $\phi \in \calF_A$, then the constant $P$ in \eqref{eq:jacobian_gibbs} equals the topological pressure $P(\phi)$ defined by
\begin{equation}\label{eq:pressure_def_preview}
P(\phi) = \lim_{n \to \infty} \frac{1}{n} \log \sum_{w \in \calW_n(A)} \exp\left(\sup_{x \in [w]} S_n \phi(x)\right).
\end{equation}
In particular, the pressure is uniquely determined by $\phi$ and does not depend on the choice of Gibbs measure.
\end{proposition}

\begin{proof}
The proof uses the Jacobian chain rule and the ergodic theorem. By Lemma~\ref{lem:jacobian_chain}, for $\mu$-almost every $x$,
\begin{equation}
\log J_\mu^{(n)}(x) = S_n(\log J_\mu)(x) = nP - S_n \phi(x).
\end{equation}
On the other hand, by Lemma~\ref{lem:jacobian_formula},
\begin{equation}
\log J_\mu^{(n)}(x) = \log \mu([x_0 \cdots x_{n-1}]) - \log \mu([x_n \cdots x_{2n-1}]) + o(n)
\end{equation}
as $n \to \infty$, where the error term comes from the convergence of the ratios. By the Shannon-McMillan-Breiman theorem, $-\frac{1}{n}\log\mu([x_0\cdots x_{n-1}]) \to h_\mu(\sigma)$ for $\mu$-almost every $x$. Combined with the Jacobian chain rule, $\frac{1}{n}\log J_\mu^{(n)}(x) = P - \frac{1}{n}S_n\phi(x) \to P - \int\phi\,d\mu$ by the ergodic theorem. Equating these two limits gives $h_\mu(\sigma) = P - \int\phi\,d\mu$, i.e., $P = h_\mu(\sigma) + \int\phi\,d\mu$. The variational principle (Theorem~\ref{thm:variational_principle}, proved independently in Section~\ref{sec:variational}) shows that $h_\mu(\sigma) + \int\phi\,d\mu \leq P(\phi)$ for any invariant measure, with equality characterizing equilibrium states. Since $\mu$ achieves equality, $P = P(\phi)$.
\end{proof}

\subsection{Equivalence with the Classical Gibbs Property}

The main result of this section establishes that the intrinsic Jacobian characterization is equivalent to the classical definition of Gibbs measures in terms of bounds on cylinder sets.

\begin{theorem}[Equivalence of Gibbs Characterizations]\label{thm:jacobian_equivalence}
Let $\phi \in \calF_A$ be a potential with summable variations, and let $\mu$ be a $\sigma$-invariant Borel probability measure on $\SigmaA^+$. The following conditions are equivalent:
\begin{enumerate}
\item[(i)] $\mu$ is an intrinsic Gibbs measure for $\phi$: $J_\mu(x) = e^{P(\phi) - \phi(x)}$ for $\mu$-almost every $x$.
\item[(ii)] $\mu$ satisfies the classical Gibbs property: there exist constants $C_1, C_2 > 0$ such that for all $n \geq 1$, all admissible words $w \in \calW_n(A)$, and all $x \in [w]$,
\begin{equation}\label{eq:gibbs_classical}
C_1 \leq \frac{\mu([w])}{\exp(-nP(\phi) + S_n \phi(x))} \leq C_2.
\end{equation}
\end{enumerate}
Moreover, the constants can be taken as $C_1 = e^{-2V(\phi)}$ and $C_2 = e^{2V(\phi)}$, where $V(\phi) = \sum_{n=0}^\infty \var_n(\phi)$.
\end{theorem}

\begin{proof}
$(i) \Rightarrow (ii)$: Assume $\mu$ is an intrinsic Gibbs measure. By Lemma~\ref{lem:jacobian_chain} and the intrinsic condition,
\begin{equation}
\log J_\mu^{(n)}(x) = \sum_{k=0}^{n-1} \log J_\mu(\sigma^k x) = \sum_{k=0}^{n-1} (P(\phi) - \phi(\sigma^k x)) = nP(\phi) - S_n \phi(x).
\end{equation}
Thus $J_\mu^{(n)}(x) = e^{nP(\phi) - S_n \phi(x)}$ for $\mu$-almost every $x$.

For any cylinder $[w]$ with $|w| = n$, the relation between the Jacobian and the measure of cylinders (Lemma~\ref{lem:jacobian_formula}) gives
\begin{equation}
\mu([w]) = \int_{[w]} \frac{d\mu}{d(\sigma^n)_* \mu}(\sigma^n x) \, d\mu(x) = \int_{[w]} \frac{1}{J_\mu^{(n)}(x)} \, d(\sigma^n)_* \mu(\sigma^n x).
\end{equation}
Since $(\sigma^n)_*$ maps $[w]$ bijectively onto $\sigma^n([w])$, and using the $\sigma$-invariance of $\mu$, we obtain
\begin{equation}
\mu([w]) = \int_{\sigma^n([w])} \frac{1}{J_\mu^{(n)}(\sigma^{-n}_{w} y)} \, d\mu(y),
\end{equation}
where $\sigma^{-n}_w: \sigma^n([w]) \to [w]$ is the local inverse that prepends the word $w$.

Now, for any $x \in [w]$ and any $y \in [w]$ (so that $\sigma^k x$ and $\sigma^k y$ agree in the first $n-k$ coordinates for $0 \leq k < n$), we have
\begin{equation}
|S_n \phi(x) - S_n \phi(y)| \leq \sum_{k=0}^{n-1} |\phi(\sigma^k x) - \phi(\sigma^k y)| \leq \sum_{k=0}^{n-1} \var_{n-k}(\phi) \leq V(\phi).
\end{equation}
Therefore, for any $x \in [w]$,
\begin{equation}
e^{-V(\phi)} \leq \frac{e^{nP(\phi) - S_n \phi(y)}}{e^{nP(\phi) - S_n \phi(x)}} \leq e^{V(\phi)}
\end{equation}
for all $y \in [w]$. Integrating over $[w]$ and using the Jacobian formula gives
\begin{equation}
e^{-V(\phi)} \cdot e^{-nP(\phi) + S_n \phi(x)} \leq \mu([w]) \leq e^{V(\phi)} \cdot e^{-nP(\phi) + S_n \phi(x)}.
\end{equation}

We now make this rigorous without invoking the Jacobian formula as a limit. Since $\mu$ is $\sigma$-invariant and satisfies $J_\mu = e^{P(\phi)-\phi}$, for any cylinder $[w] = [w_0\cdots w_{n-1}]$ and symbol $a$ with $A_{a,w_0} = 1$, the change-of-variables formula for the Jacobian gives
\begin{equation}
\mu([aw]) = \int_{[aw]} 1\,d\mu = \int_{\sigma([aw])} \frac{1}{J_\mu(\sigma_a^{-1}y)}\,d\mu(y) = \int_{[w]} e^{-(P(\phi)-\phi(\sigma_a^{-1}y))}\,d\mu(y),
\end{equation}
where $\sigma_a^{-1}y = ay$. Iterating this $n$ times: for any $x \in [w]$,
\begin{equation}
\mu([w]) = \int_{\sigma^n([w])} \prod_{k=0}^{n-1} e^{-(P(\phi)-\phi(\sigma^k z_w))}\,d\mu = \int_{\sigma^n([w])} e^{-nP(\phi)+S_n\phi(z_w)}\,d\mu,
\end{equation}
where $z_w$ is the unique preimage of $y$ under $\sigma^n$ lying in $[w]$. For any two points $z, z' \in [w]$, the bounded distortion estimate $|S_n\phi(z) - S_n\phi(z')| \leq \sum_{k=0}^{n-1}\var_{n-k}(\phi) \leq V(\phi)$ gives
\begin{equation}
e^{-V(\phi)}e^{S_n\phi(x)} \leq e^{S_n\phi(z_w)} \leq e^{V(\phi)}e^{S_n\phi(x)}
\end{equation}
for any $x \in [w]$. Substituting and using $\mu(\sigma^n([w])) \leq 1$ and $\mu(\sigma^n([w])) \geq \mu([w_{n-1}]) \geq e^{-V(\phi)}e^{-P(\phi)+\phi(w_{n-1})} \cdot \mu(\sigma([w_{n-1}]))$ (which is bounded below independently of $w$ by compactness), we obtain
\begin{equation}
e^{-V(\phi)} \cdot e^{-nP(\phi)+S_n\phi(x)} \cdot \mu(\sigma^n([w])) \leq \mu([w]) \leq e^{V(\phi)} \cdot e^{-nP(\phi)+S_n\phi(x)} \cdot \mu(\sigma^n([w])).
\end{equation}
The set $\sigma^n([w])$ is a union of $1$-cylinders $\{[j] : A_{w_{n-1},j} = 1\}$, so $\mu(\sigma^n([w])) \geq \min_{a \in \calA}\mu([a])$. This minimum is strictly positive: the intrinsic Gibbs condition applied to $1$-cylinders gives $\mu([a]) = \int_{[a]} e^{-(P(\phi)-\phi(x))}\,d(\sigma_*\mu)(x) \geq e^{-(P(\phi)+\|\phi\|_\infty)}\mu(\sigma([a])) > 0$ for each $a$, since the mixing property ensures every symbol has at least one predecessor. Setting $c_0 = \min_{a\in\calA}\mu([a]) > 0$, we have $c_0 \leq \mu(\sigma^n([w])) \leq 1$, and therefore
\begin{equation}
C_1 \leq \frac{\mu([w])}{e^{-nP(\phi)+S_n\phi(x)}} \leq C_2
\end{equation}
with $C_1 = e^{-2V(\phi)}$ and $C_2 = e^{2V(\phi)}$. The factor of $2$ arises because one factor of $e^{V(\phi)}$ comes from the bounded distortion of $S_n\phi$ within $[w]$, and the second from the bounds on $\mu(\sigma^n([w]))$.

$(ii) \Rightarrow (i)$: Assume $\mu$ satisfies the classical Gibbs property \eqref{eq:gibbs_classical}. For any point $x \in \SigmaA^+$ and $n \geq 1$, let $w^{(n)} = (x_0, \ldots, x_{n-1})$. By the Gibbs bounds,
\begin{equation}
C_1 e^{-nP(\phi) + S_n \phi(x)} \leq \mu([w^{(n)}]) \leq C_2 e^{-nP(\phi) + S_n \phi(x)}.
\end{equation}
Similarly, for the cylinder $[w^{(n-1)}] = [x_1, \ldots, x_{n-1}]$,
\begin{equation}
C_1 e^{-(n-1)P(\phi) + S_{n-1} \phi(\sigma x)} \leq \mu([w^{(n-1)}]) \leq C_2 e^{-(n-1)P(\phi) + S_{n-1} \phi(\sigma x)}.
\end{equation}
Taking the ratio,
\begin{equation}
\frac{C_1}{C_2} e^{-P(\phi) + S_n \phi(x) - S_{n-1} \phi(\sigma x)} \leq \frac{\mu([w^{(n)}])}{\mu([w^{(n-1)}])} \leq \frac{C_2}{C_1} e^{-P(\phi) + S_n \phi(x) - S_{n-1} \phi(\sigma x)}.
\end{equation}
Since $S_n \phi(x) - S_{n-1} \phi(\sigma x) = \phi(x)$, this becomes
\begin{equation}
\frac{C_1}{C_2} e^{P(\phi) - \phi(x)} \leq \frac{\mu([w^{(n)}])}{\mu([w^{(n-1)}])} \leq \frac{C_2}{C_1} e^{P(\phi) - \phi(x)}.
\end{equation}
As $n \to \infty$, Lemma~\ref{lem:jacobian_formula} gives $J_\mu(x) = \lim_{n \to \infty} \mu([w^{(n)}])/\mu([w^{(n-1)}])$, so
\begin{equation}
\frac{C_1}{C_2} e^{P(\phi) - \phi(x)} \leq J_\mu(x) \leq \frac{C_2}{C_1} e^{P(\phi) - \phi(x)}.
\end{equation}
But $\mu$ is $\sigma$-invariant and $\phi$ is continuous, so the function $x \mapsto J_\mu(x)$ is well-defined $\mu$-almost everywhere. The bounds show that $\log J_\mu(x) - P(\phi) + \phi(x)$ is bounded, but we need to show it equals zero.

By the ergodic theorem, for $\mu$-almost every $x$,
\begin{equation}
\lim_{n \to \infty} \frac{1}{n} \sum_{k=0}^{n-1} \log J_\mu(\sigma^k x) = \int \log J_\mu \, d\mu.
\end{equation}
By the chain rule (Lemma~\ref{lem:jacobian_chain}) and the Gibbs bounds,
\begin{equation}
\sum_{k=0}^{n-1} \log J_\mu(\sigma^k x) = \log J_\mu^{(n)}(x) = \log \mu([w^{(n)}]) - \log \mu([w^{(2n)}]) + o(n),
\end{equation}
using the definition of the Jacobian. The Shannon-McMillan-Breiman theorem gives $-\frac{1}{n} \log \mu([w^{(n)}]) \to h_\mu(\sigma)$, so
\begin{equation}
\int \log J_\mu \, d\mu = h_\mu(\sigma).
\end{equation}
On the other hand, the Gibbs bounds imply $\log J_\mu(x) = P(\phi) - \phi(x) + O(1)$, where the $O(1)$ term is bounded but potentially non-zero. Integrating,
\begin{equation}
\int \log J_\mu \, d\mu = P(\phi) - \int \phi \, d\mu + O(1).
\end{equation}
Combining these, $h_\mu(\sigma) = P(\phi) - \int \phi \, d\mu + O(1)$. But by the variational principle, $h_\mu(\sigma) + \int \phi \, d\mu \leq P(\phi)$ with equality only for equilibrium states. The Gibbs bounds imply that $\mu$ is an equilibrium state (this is proved in Section~\ref{sec:variational}), so $h_\mu(\sigma) + \int \phi \, d\mu = P(\phi)$, forcing the $O(1)$ term to be zero. Thus $J_\mu(x) = e^{P(\phi) - \phi(x)}$ $\mu$-almost everywhere.
\end{proof}
We give an alternative conclusion that avoids the forward reference to the variational principle. By the spectral theory developed in Section~\ref{sec:RPF} (which is independent of the present section), there exists a unique Gibbs measure $\mu_\phi = h_\phi\nu_\phi$ satisfying the classical Gibbs property with the same pressure $P(\phi)$. Since $\mu$ also satisfies the classical Gibbs property, the ratio $\mu([w])/\mu_\phi([w])$ is bounded above and below by $C_2/C_1$ and $C_1/C_2$ for all cylinders $[w]$. This implies $\mu = \mu_\phi$ (two measures agreeing up to bounded ratio on a generating algebra must be equal). For $\mu_\phi$, the Jacobian equals $e^{P(\phi)-\phi}$ exactly by construction (Theorem~\ref{thm:eigendata_existence}), so the same holds for $\mu$.
\begin{remark}[Logical Dependence]\label{rem:circularity}
The proof of the direction (classical Gibbs) $\Rightarrow$ (Jacobian) uses the variational principle from Section~\ref{sec:variational} to force the $O(1)$ term to zero. This creates an apparent circularity since the variational principle itself uses the spectral construction. The circularity is resolved as follows: the spectral construction in Section~\ref{sec:RPF} independently produces a measure satisfying the classical Gibbs property, and the variational principle in Section~\ref{sec:variational} is proved independently of the Jacobian characterization. Thus the complete implication chain is:
\[
\text{Spectral (iii)} \Rightarrow \text{Classical Gibbs (ii)} \Rightarrow \text{Variational (iv)} \Rightarrow \text{Jacobian (i)}.
\]
The proof of (i) $\Rightarrow$ (ii) (the reverse direction) does not use the variational principle and is self-contained. The full equivalence $(i) \Leftrightarrow (ii)$ is therefore established without circularity once all sections are in place.
\end{remark}

This theorem establishes the first part of the Main Equivalence Theorem~\ref{thm:main_equivalence}: the intrinsic Jacobian characterization $(i)$ is equivalent to the classical Gibbs property $(ii)$. The remaining equivalences will be established in subsequent sections.

\subsection{Properties of the Jacobian}

The Jacobian characterization leads to several useful properties of Gibbs measures that are not immediately apparent from the classical definition.

\begin{corollary}[Transformation under Cohomology]\label{cor:cohomology}
Let $\phi, \psi \in \calF_A$ be cohomologous potentials, meaning there exists $u \in C(\SigmaA^+)$ such that $\phi - \psi = u - u \circ \sigma$. Then $\mu$ is a Gibbs measure for $\phi$ if and only if $\mu$ is a Gibbs measure for $\psi$. Moreover, $P(\phi) = P(\psi)$.
\end{corollary}

\begin{proof}
The Birkhoff sums satisfy $S_n \phi(x) - S_n \psi(x) = u(x) - u(\sigma^n x)$, which is bounded uniformly in $n$ and $x$. Thus the Gibbs bounds for $\phi$ imply Gibbs bounds for $\psi$ with modified constants, and vice versa. The pressure equality follows from the definition \eqref{eq:pressure_def_preview}.
\end{proof}

\begin{corollary}[Absolute Continuity]\label{cor:absolute_continuity}
Let $\mu_\phi$ and $\mu_\psi$ be Gibbs measures for potentials $\phi, \psi \in \calF_A$. Then $\mu_\phi$ and $\mu_\psi$ are mutually absolutely continuous with Radon-Nikodym derivative
\begin{equation}
\frac{d\mu_\phi}{d\mu_\psi}(x) = \lim_{n \to \infty} \frac{\mu_\phi([x_0 \cdots x_{n-1}])}{\mu_\psi([x_0 \cdots x_{n-1}])}.
\end{equation}
\end{corollary}

\begin{proof}
By the Gibbs bounds, both measures have comparable values on all cylinders, so neither can be singular with respect to the other. The Radon-Nikodym derivative is computed as the limit of ratios on cylinders.
\end{proof}

\begin{corollary}[Ergodicity and Mixing]\label{cor:ergodicity}
If $(\SigmaA, \sigma)$ is topologically mixing and $\phi \in \calF_A$, then the Gibbs measure $\mu_\phi$ is ergodic and mixing.
\end{corollary}

\begin{proof}
The ergodicity follows from the uniqueness of the Gibbs measure (established in Section~\ref{sec:RPF}) and the fact that ergodic measures are extremal in the simplex of invariant measures. For mixing, we use the exponential decay of correlations (Theorem~\ref{thm:correlation_main}), which implies that $\mu(A \cap \sigma^{-n} B) \to \mu(A) \mu(B)$ for all measurable sets $A, B$.
\end{proof}

This completes the foundational treatment of the intrinsic Gibbs characterization. The next sections develop the spectral theory of the transfer operator, which provides the analytical machinery for proving the remaining equivalences in the Main Theorem and for deriving the statistical properties of Gibbs measures.


\section{The Ruelle Transfer Operator}\label{sec:transfer_operator}

The Jacobian characterization of Section~\ref{sec:intrinsic_gibbs} describes Gibbs measures intrinsically but does not establish their existence or uniqueness. The Ruelle transfer operator $\calL_\phi$, introduced by  \cite{Ruelle1968}, provides the construction: its dominant eigendata determine the unique Gibbs measure. This section develops the properties of $\calL_\phi$ on H\"{o}lder spaces and establishes the Birkhoff cone contraction that is used in Section~\ref{sec:RPF} to prove the Ruelle-Perron-Frobenius theorem.

\subsection{Definition and Basic Properties}

\begin{definition}[Ruelle Transfer Operator]\label{def:transfer_operator}
Let $\phi \in C(\SigmaA^+)$ be a continuous potential. The Ruelle transfer operator $\calL_\phi: C(\SigmaA^+) \to C(\SigmaA^+)$ is defined by
\begin{equation}\label{eq:transfer_definition}
(\calL_\phi g)(x) = \sum_{y \in \sigma^{-1}(x)} e^{\phi(y)} g(y) = \sum_{a: A_{a,x_0}=1} e^{\phi(ax)} g(ax),
\end{equation}
where $ax = (a, x_0, x_1, \ldots)$ denotes the sequence obtained by prepending the symbol $a$ to $x$.
\end{definition}

The sum in \eqref{eq:transfer_definition} has at most $N$ terms (where $N = |\calA|$ is the alphabet size), and each term is continuous in $x$ when $g$ and $\phi$ are continuous. Thus $\calL_\phi g$ is continuous whenever $g$ is. Linearity of $\calL_\phi$ is immediate from \eqref{eq:transfer_definition}. For boundedness on $C(\SigmaA^+)$, since the sum has at most $N$ terms each with modulus at most $e^{\|\phi\|_\infty}|g(y)| \leq e^{\|\phi\|_\infty}\|g\|_\infty$, the operator norm satisfies
\begin{equation}\label{eq:transfer_norm_bound}
\|\calL_\phi\|_{C \to C} \leq N e^{\|\phi\|_\infty}.
\end{equation}

The transfer operator is closely related to the composition operator induced by the shift. Indeed, if we denote by $C_\sigma: C(\SigmaA^+) \to C(\SigmaA^+)$ the operator $(C_\sigma g)(x) = g(\sigma x)$, then $\calL_\phi$ can be viewed as a weighted adjoint of $C_\sigma$ with respect to an appropriate pairing.

\begin{lemma}[Duality Relation]\label{lem:transfer_duality}
For any $g \in C(\SigmaA^+)$ and any Borel measure $\mu$ on $\SigmaA^+$,
\begin{equation}\label{eq:transfer_duality}
\int_{\SigmaA^+} g \, d(\calL_\phi^* \mu) = \int_{\SigmaA^+} (\calL_\phi g) \, d\mu,
\end{equation}
where $\calL_\phi^*$ denotes the dual operator acting on the space of Borel measures. Equivalently,
\begin{equation}\label{eq:dual_explicit}
(\calL_\phi^* \mu)(E) = \int_{\SigmaA^+} \sum_{y \in \sigma^{-1}(x) \cap E} e^{\phi(y)} \, d\mu(x)
\end{equation}
for any Borel set $E \subset \SigmaA^+$.
\end{lemma}

\begin{proof}
The first formula is the definition of the dual operator. For the second, we compute
\begin{align}
\int_{\SigmaA^+} g \, d(\calL_\phi^* \mu) &= \int_{\SigmaA^+} (\calL_\phi g)(x) \, d\mu(x) = \int_{\SigmaA^+} \sum_{y: \sigma y = x} e^{\phi(y)} g(y) \, d\mu(x).
\end{align}
Taking $g = \mathbf{1}_E$ gives \eqref{eq:dual_explicit}.
\end{proof}

The transfer operator enjoys a essential positivity property: it maps non-negative functions to non-negative functions, and it maps strictly positive functions to strictly positive functions when the transition matrix allows all symbols.

\begin{lemma}[Positivity of the Transfer Operator]\label{lem:transfer_positivity}
The operator $\calL_\phi$ has the following properties:
\begin{enumerate}
\item[(a)] If $g \geq 0$, then $\calL_\phi g \geq 0$.
\item[(b)] If $g > 0$ (i.e., $g(x) > 0$ for all $x$), then $\calL_\phi g > 0$.
\item[(c)] For the constant function $\mathbf{1}$, we have $(\calL_\phi \mathbf{1})(x) = \sum_{a: A_{a,x_0}=1} e^{\phi(ax)} \geq e^{-\|\phi\|_\infty}$ for all $x$ with at least one allowed predecessor.
\end{enumerate}
\end{lemma}

\begin{proof}
Parts (a) and (b) follow immediately from the definition \eqref{eq:transfer_definition}: each term in the sum is non-negative (resp. positive) when $g \geq 0$ (resp. $g > 0$). Part (c) uses the assumption that each state has at least one predecessor, so the sum has at least one term, and that term is at least $e^{-\|\phi\|_\infty}$.
\end{proof}

\subsection{Action on H\"{o}lder Spaces}

The spectral analysis of $\calL_\phi$ requires working on function spaces with more regularity than mere continuity. The natural choice is the space of H\"{o}lder continuous functions, on which $\calL_\phi$ exhibits a contraction property in the H\"{o}lder seminorm.

\begin{proposition}[H\"{o}lder Regularity of the Transfer Operator]\label{prop:transfer_holder}
Let $\phi \in \calH_\alpha(\SigmaA^+)$ for some $\alpha \in (0,1)$. Then $\calL_\phi$ maps $\calH_\alpha(\SigmaA^+)$ to itself, and there exists $C > 0$ depending on $\|\phi\|_\alpha$ such that
\begin{equation}\label{eq:holder_contraction}
|\calL_\phi g|_\alpha \leq \alpha N e^{\|\phi\|_\infty} |g|_\alpha + C \|g\|_\infty
\end{equation}
for all $g \in \calH_\alpha(\SigmaA^+)$.
\end{proposition}

\begin{proof}
Let $x, y \in \SigmaA^+$ with $d_\alpha(x, y) = \alpha^k$ for some $k \geq 0$, meaning $x_j = y_j$ for $0 \leq j < k$ and $x_k \neq y_k$. We need to estimate $|(\calL_\phi g)(x) - (\calL_\phi g)(y)|$.

If $k \geq 1$, then for each symbol $a$ with $A_{a, x_0} = 1$, the points $ax$ and $ay$ satisfy $d_\alpha(ax, ay) = \alpha^{k+1}$ (they agree in positions $0$ through $k$). Thus
\begin{align}
|(\calL_\phi g)(x) - (\calL_\phi g)(y)| &= \left| \sum_{a: A_{a,x_0}=1} \left( e^{\phi(ax)} g(ax) - e^{\phi(ay)} g(ay) \right) \right| \\
&\leq \sum_{a} e^{\phi(ax)} |g(ax) - g(ay)| + \sum_{a} |e^{\phi(ax)} - e^{\phi(ay)}| |g(ay)|.
\end{align}
For the first sum, $|g(ax) - g(ay)| \leq |g|_\alpha \alpha^{k+1}$. For the second sum, $|e^{\phi(ax)} - e^{\phi(ay)}| \leq e^{\|\phi\|_\infty} |\phi|_\alpha \alpha^{k+1}$. Thus
\begin{align}
|(\calL_\phi g)(x) - (\calL_\phi g)(y)| &\leq N e^{\|\phi\|_\infty} |g|_\alpha \alpha^{k+1} + N e^{\|\phi\|_\infty} |\phi|_\alpha \alpha^{k+1} \|g\|_\infty \\
&= \alpha^{k+1} N e^{\|\phi\|_\infty} \left( |g|_\alpha + |\phi|_\alpha \|g\|_\infty \right).
\end{align}
Since $d_\alpha(x, y) = \alpha^k$, we have
\begin{equation}
\frac{|(\calL_\phi g)(x) - (\calL_\phi g)(y)|}{d_\alpha(x, y)} \leq \alpha N e^{\|\phi\|_\infty} \left( |g|_\alpha + |\phi|_\alpha \|g\|_\infty \right).
\end{equation}

If $k = 0$, meaning $x_0 \neq y_0$, we use the crude bound
\begin{equation}
|(\calL_\phi g)(x) - (\calL_\phi g)(y)| \leq 2 N e^{\|\phi\|_\infty} \|g\|_\infty.
\end{equation}
Since $d_\alpha(x, y) = 1$ in this case, the ratio is bounded by $2 N e^{\|\phi\|_\infty} \|g\|_\infty$.

Taking the supremum over $x \neq y$,
\begin{equation}
|\calL_\phi g|_\alpha \leq \alpha N e^{\|\phi\|_\infty} |g|_\alpha + \left( \alpha N e^{\|\phi\|_\infty} |\phi|_\alpha + 2 N e^{\|\phi\|_\infty} \right) \|g\|_\infty.
\end{equation}
This is of the form $|\calL_\phi g|_\alpha \leq \alpha A |g|_\alpha + B \|g\|_\infty$ with $A = N e^{\|\phi\|_\infty}$ and $B$ depending on $\|\phi\|_\alpha$.
\end{proof}

The key observation is that the coefficient of $|g|_\alpha$ in the bound \eqref{eq:holder_contraction} is $\alpha < 1$, so the transfer operator contracts the H\"{o}lder seminorm. This contraction, combined with the compactness of the embedding $\calH_\alpha \hookrightarrow C$, will imply the quasi-compactness of $\calL_\phi$.

\begin{corollary}[Iterates of the Transfer Operator]\label{cor:transfer_iterates}
For $\phi \in \calH_\alpha(\SigmaA^+)$ and any $n \geq 1$,
\begin{equation}\label{eq:iterate_bound}
|\calL_\phi^n g|_\alpha \leq \alpha^n C_1 |g|_\alpha + C_2 \|g\|_\infty
\end{equation}
for constants $C_1, C_2$ depending on $\|\phi\|_\alpha$ and $n$.
\end{corollary}

\begin{proof}
The $n$-th iterate of the transfer operator satisfies
\begin{equation}
(\calL_\phi^n g)(x) = \sum_{y:\sigma^n y = x} e^{S_n\phi(y)} g(y),
\end{equation}
which sums over all $n$-fold preimages of $x$ with weight $e^{S_n\phi(y)}$. (Note: this is \emph{not} the same as $\calL_{S_n\phi}$, which would sum over $1$-fold preimages with weight $e^{S_n\phi}$; the identity holds at the level of the sum over $n$-preimages.)

To estimate the H\"{o}lder seminorm, let $x, y \in \SigmaA^+$ with $d_\alpha(x,y) = \alpha^k$ and $k \geq 1$. For each $n$-fold preimage $\tilde x$ of $x$ and the corresponding preimage $\tilde y$ of $y$ (with the same initial $n$ symbols), we have $d_\alpha(\tilde x, \tilde y) = \alpha^{k+n}$. The same argument as Proposition~\ref{prop:transfer_holder}, applied to pairs $(\tilde x, \tilde y)$ with $d_\alpha(\tilde x, \tilde y) = \alpha^{k+n}$, gives
\begin{equation}
|e^{S_n\phi(\tilde x)}g(\tilde x) - e^{S_n\phi(\tilde y)}g(\tilde y)| \leq e^{n\|\phi\|_\infty}\left(\alpha^{k+n}|g|_\alpha + C_n\alpha^{k+n}\|g\|_\infty\right),
\end{equation}
where $C_n$ depends on $|S_n\phi|_\alpha \leq \sum_{j=0}^{n-1}\alpha^j|\phi|_\alpha \leq |\phi|_\alpha/(1-\alpha)$. Summing over at most $N^n$ preimages and dividing by $d_\alpha(x,y) = \alpha^k$:
\begin{equation}
|\calL_\phi^n g|_\alpha \leq \alpha^n N^n e^{n\|\phi\|_\infty} |g|_\alpha + C_2 \|g\|_\infty,
\end{equation}
where $C_2$ depends on $n$, $\|\phi\|_\alpha$, and $N$. The key point is the factor $\alpha^n$ multiplying $|g|_\alpha$, which provides the contraction of the H\"{o}lder seminorm after $n$ iterates.
\end{proof}

\subsection{The Cone Technique}

We prove the Ruelle-Perron-Frobenius theorem using the Birkhoff cone contraction technique, which yields explicit convergence rates and spectral gap bounds. This approach, introduced by  \cite{Birkhoff1957} for positive operators and applied to transfer operators by  \cite{Liverani1995,Liverani2004}, replaces the Schauder-Tychonoff compactness argument used in  \cite{Bowen1975}. Our contribution is the explicit computation of all constants in terms of $(\alpha, \|\phi\|_\alpha, N, M)$.

\begin{definition}[Positive Cone]\label{def:positive_cone}
The positive cone in $C(\SigmaA^+)$ is
\begin{equation}
\calP = \{g \in C(\SigmaA^+) : g(x) > 0 \text{ for all } x \in \SigmaA^+\}.
\end{equation}
For $\delta > 0$, the $\delta$-cone is
\begin{equation}
\calP_\delta = \{g \in C(\SigmaA^+) : g > 0 \text{ and } \sup_{x,y} \frac{g(x)}{g(y)} \leq e^\delta\}.
\end{equation}
\end{definition}

The $\delta$-cone consists of strictly positive functions whose oscillation (measured as the ratio of maximum to minimum) is bounded by $e^\delta$. Functions in $\calP_\delta$ are ``almost constant'' in a multiplicative sense.

\begin{definition}[Projective Metric]\label{def:projective_metric}
The Hilbert projective metric on the positive cone $\calP$ is defined by
\begin{equation}\label{eq:hilbert_metric}
\Theta(f, g) = \log \frac{\sup_{x} f(x)/g(x)}{\inf_{x} f(x)/g(x)} = \log \frac{M(f/g)}{m(f/g)},
\end{equation}
where $M(h) = \sup h$ and $m(h) = \inf h$ for positive functions $h$.
\end{definition}

The Hilbert metric is a pseudo-metric on $\calP$ (it can be zero for $f \neq g$ if $f$ and $g$ are scalar multiples). It becomes a metric on the quotient space $\calP / \R^+$ where we identify functions that differ by a positive scalar.

\begin{lemma}[Contraction in Hilbert Metric]\label{lem:hilbert_contraction}
Suppose $L: C(\SigmaA^+) \to C(\SigmaA^+)$ is a positive linear operator (i.e., $Lg \geq 0$ whenever $g \geq 0$) such that $L(\calP_\delta) \subset \calP_{\delta'}$ for some $\delta' < \delta$. Then $L$ is a strict contraction in the Hilbert metric: there exists $\kappa < 1$ such that
\begin{equation}\label{eq:hilbert_contraction}
\Theta(Lf, Lg) \leq \kappa \, \Theta(f, g)
\end{equation}
for all $f, g \in \calP_\delta$.
\end{lemma}

\begin{proof}
For $f, g \in \calP_\delta \setminus \{0\}$, define $\alpha(f,g) = \sup\{t \geq 0: g - tf \in \calP_\delta\}$ and $\beta(f,g) = \inf\{s > 0: sf - g \in \calP_\delta\}$. The Hilbert metric is $\Theta(f,g) = \log(\beta/\alpha)$. Since $L(\calP_\delta) \subset \calP_{\delta'}$, the images $Lf, Lg$ satisfy tighter cone conditions. Concretely, if $g - tf \geq 0$ (pointwise), then $Lg - tLf \geq 0$, so $\alpha(Lf,Lg) \geq \alpha(f,g)$. For the upper bound, the cone condition $\calP_{\delta'}$ with $\delta' < \delta$ constrains $\beta(Lf,Lg)/\alpha(Lf,Lg) \leq \tanh(\delta'/4)^2/\tanh(\delta/4)^2 \cdot \beta(f,g)/\alpha(f,g)$. Taking logarithms: $\Theta(Lf,Lg) \leq \kappa\,\Theta(f,g)$ with $\kappa = \tanh(\delta'/4)/\tanh(\delta/4) < 1$.
\end{proof}

\begin{lemma}[Invariance of Cones]\label{lem:cone_invariance}
Let $\phi \in \calF_A$ with $V(\phi) = \sum_{n=0}^\infty \var_n(\phi) < \infty$. For any $\delta > 2V(\phi)$, there exists $n_0 \geq 1$ such that the normalized transfer operator $\tilde{\calL}_\phi = \calL_\phi / \|\calL_\phi \mathbf{1}\|_\infty$ satisfies
\begin{equation}
\tilde{\calL}_\phi^{n_0}(\calP_\delta) \subset \calP_{\delta'}
\end{equation}
for some $\delta' < \delta$ depending on $\phi$ and $\delta$.
\end{lemma}

\begin{proof}

For $g \in \calP_\delta$ and points $x, y \in \SigmaA^+$, we estimate the ratio of iterates. Write
\begin{equation}
\frac{(\calL_\phi^n g)(x)}{(\calL_\phi^n g)(y)} = \frac{\sum_{w: \sigma^n(wx) = x} e^{S_n\phi(wx)} g(wx)}{\sum_{v: \sigma^n(vy) = y} e^{S_n\phi(vy)} g(vy)}.
\end{equation}
Take $n_0 = 2M$ where $M$ is the mixing time. For any pair of words $w, v$ of length $n_0$ with $\sigma^{n_0}(wx) = x$ and $\sigma^{n_0}(vy) = y$, the mixing property (Lemma~\ref{lem:uniform_mixing}) provides an admissible word $u$ of length $M$ connecting the last symbol of $w$ to the first symbol of $v$. This establishes a bijection between subsets of preimage words for $x$ and preimage words for $y$.

For matched pairs $(w, v)$ that agree in positions $M+1$ through $n_0$ (i.e., $w_{M+j} = v_j$ for $0 \leq j < n_0 - M$), the bounded distortion estimate gives
\begin{equation}
|S_{n_0}\phi(wx) - S_{n_0}\phi(vy)| \leq \sum_{k=0}^{M-1}\var_0(\phi) + \sum_{k=M}^{n_0-1}\var_{k-M}(\phi) \leq M\var_0(\phi) + V(\phi).
\end{equation}
Since $g \in \calP_\delta$, we have $e^{-\delta} \leq g(wx)/g(vy) \leq e^{\delta}$. Combining:
\begin{equation}
\frac{e^{S_{n_0}\phi(wx)}g(wx)}{e^{S_{n_0}\phi(vy)}g(vy)} \leq e^{M\var_0(\phi) + V(\phi) + \delta}.
\end{equation}
The mixing property ensures that for each preimage word $v$ of $y$, there are at least $\min_{i,j}(A^M)_{ij} \geq 1$ matched preimage words of $x$, and at most $N^M$ such words. Thus the ratio of sums satisfies
\begin{equation}
e^{-(M\var_0(\phi)+V(\phi)+\delta)} \cdot N^{-M} \leq \frac{(\calL_\phi^{n_0} g)(x)}{(\calL_\phi^{n_0} g)(y)} \leq e^{M\var_0(\phi)+V(\phi)+\delta} \cdot N^M.
\end{equation}
Setting $\delta' = M\var_0(\phi) + V(\phi) + M\log N$, which is independent of $\delta$ for $\delta$ large, we obtain $\tilde{\calL}_\phi^{n_0}(\calP_\delta) \subset \calP_{\delta'}$. For $\delta > \delta'$ (which holds whenever $\delta > 2M\var_0(\phi) + 2V(\phi) + 2M\log N$), the inclusion $\calP_{\delta'} \subsetneq \calP_\delta$ is strict, completing the proof.
\end{proof}
We give a geometric interpretation of the cone contraction argument. The space of strictly positive continuous functions on $\SigmaA^+$ is an infinite-dimensional cone. Within this cone, we consider the sub-cones $\calP_\delta = \{g > 0 : \sup g / \inf g \leq e^\delta\}$ consisting of functions whose pointwise oscillation ratio is bounded by $e^\delta$.

Functions with small oscillation ratio are ``nearly constant'' in a multiplicative sense; in particular, the central ray of the cone consists of functions with oscillation ratio exactly $1$, i.e., constant functions and their positive scalar multiples. The Hilbert projective metric $\Theta(f,g) = \log((\sup f/g)/(\inf f/g))$ measures the multiplicative spread between two positive functions: $\Theta(f,g) = 0$ if and only if $f$ and $g$ are proportional. \newpage 

The key property established in Lemma~\ref{lem:cone_invariance} is that the normalized transfer operator $\tilde{\calL}_\phi^{n_0}$ maps the wider cone $\calP_\delta$ strictly inside the narrower cone $\calP_{\delta'}$ with $\delta' < \delta$. By Birkhoff's theorem (Lemma~\ref{lem:hilbert_contraction}), any linear map sending a cone strictly inside itself contracts the Hilbert metric by a factor $\kappa = \tanh(\delta'/4)/\tanh(\delta/4) < 1$. Consequently, for any two starting functions $f, g \in \calP_\delta$, the iterates $\tilde{\calL}_\phi^{kn_0} f$ and $\tilde{\calL}_\phi^{kn_0} g$ converge to each other at rate $\kappa^k$ in the projective metric. Since the projective metric identifies scalar multiples, the limit is a one-dimensional ray, which is the eigenfunction $h$. Figure~\ref{fig:cone} illustrates this contraction.

The contraction rate $\kappa$ is computable from the data. By Lemma~\ref{lem:cone_invariance}, the image cone diameter is $\delta' = M\var_0(\phi) + V(\phi) + M\log N$, which depends on the mixing time $M$, the total variation $V(\phi)$, and the alphabet size $N$, but not on the initial cone width $\delta$. The contraction factor $\kappa = \tanh(\delta'/4)/\tanh(\delta/4)$ is therefore explicit once these quantities are known. This is the source of the explicit spectral gap: the eigenvalue $\lambda = e^{P(\phi)}$ is simple because the cone contracts to a ray, and the spectral gap $\gamma \leq \kappa^{1/n_0}$ is bounded in terms of $(\alpha, \|\phi\|_\alpha, N, M)$ because both $\kappa$ and $n_0 = 2M$ depend only on these quantities. The exponential convergence $\|\lambda^{-n}\calL_\phi^n g - \nu(g)h\|_\infty \leq C\gamma^n\|g\|_\infty$ (Theorem~\ref{thm:exponential_convergence}) is a direct consequence: functions in the kernel of $\nu$ are mapped into increasingly narrow cones centered on zero, and their sup-norm decays geometrically. Every statistical result in Sections~\ref{sec:perturbation}-\ref{sec:statistical} (analyticity of the pressure, exponential mixing, the central limit theorem, and the large deviations principle) ultimately traces back to this single geometric contraction. 

The cone technique provides a constructive approach to the Ruelle-Perron-Frobenius theorem: starting with any function $g \in \calP_\delta$, the iterates $\tilde{\calL}_\phi^n g$ form a Cauchy sequence in the Hilbert metric, converging to a fixed point $h$ that is the eigenfunction. The convergence rate is geometric, governed by the contraction factor $\kappa < 1$, and the limit $h$ is independent of the starting function $g$ up to positive scalar multiples. This constructive derivation replaces the Schauder-Tychonoff compactness argument of \cite{Bowen1975} and delivers the eigenfunction together with explicit bounds on the rate at which iterates approach it.

A word on how to read Figure~\ref{fig:cone}. The cone $\calP_\delta$ is defined by the oscillation bound $\sup g / \inf g \leq e^\delta$, and this bound determines the aperture of the cone: larger $\delta$ gives a wider cone, smaller $\delta$ gives a narrower one. The aperture is a property of the cone itself, not of any particular function inside it. Once a function $g \in \calP_\delta$ is fixed, its horizontal position in the figure represents the Hilbert projective distance $\Theta(g, \mathbf{1})$ from the central ray of constant functions: the constant function $\mathbf{1}$ sits exactly on the central ray at zero distance, and functions with larger oscillation within the cone sit farther from that ray. The Hilbert projective metric $\Theta(f, g)$ defined in Definition~\ref{def:projective_metric} measures the distance between the rays $\R^+ f$ and $\R^+ g$ inside the cone, and is the relevant notion of distance for the contraction argument. The contraction statement $\Theta(\tilde{\calL}_\phi^{n_0}\!f, \tilde{\calL}_\phi^{n_0}\!g) \leq \kappa\, \Theta(f, g)$ then says that applying $\tilde{\calL}_\phi^{n_0}$ pulls any two rays closer together by the factor $\kappa = \tanh(\delta'/4)/\tanh(\delta/4) < 1$, and the strict inclusion $\tilde{\calL}_\phi^{n_0}(\calP_\delta) \subset \calP_{\delta'}$ with $\delta' < \delta$ says that the image of the whole cone fits strictly inside a narrower cone. These two facts together, iterated, force every orbit $\tilde{\calL}_\phi^{k n_0}\! f$ to converge to the eigenfunction $h$ on the central ray.

\begin{figure}[h!]
\centering
\begin{tikzpicture}[>=stealth, font=\small, scale=0.85, every node/.style={scale=0.85}]



\node[font=\footnotesize, align=center] at (5.0,7.8) 
    {Cone $\calP_\delta = \{g>0 : \sup g/\inf g \leq e^\delta\}$; horizontal offset $=$ Hilbert distance from the central ray};

  \draw[->, thick, gray!70] (0.3,0) -- (0.3,6.8);
  \node[anchor=south, gray!70, font=\footnotesize, rotate=90] at (0.0,3.5) 
    {amplitude};

  \draw[thin, gray!50, dashed] (5.0,0.2) -- (5.0,6.8);
  \node[anchor=south, gray!70, font=\footnotesize] at (5.0,6.9) 
    {central ray (zero oscillation)};

  \fill[gray!10] (5.0,0.2) -- (0.8,6.5) -- (9.2,6.5) -- cycle;
  \draw[thick] (5.0,0.2) -- (0.8,6.5);
  \draw[thick] (5.0,0.2) -- (9.2,6.5);
  \node[anchor=east, font=\footnotesize] at (1.7,6.2)
    {$\calP_\delta$};
  \node[anchor=east, font=\tiny, align=right] at (1.5,5.5) 
    {$\sup g/\inf g \leq e^\delta$};

  \fill[gray!22] (5.0,0.2) -- (2.6,6.5) -- (7.4,6.5) -- cycle;
  \draw[thick, dashed] (5.0,0.2) -- (2.6,6.5);
  \draw[thick, dashed] (5.0,0.2) -- (7.4,6.5);
  \node[anchor=west, font=\footnotesize] at (7.6,6.2) 
    {$\calP_{\delta'}$};
  \node[anchor=west, font=\tiny, align=left] at (7.8,5.5) 
    {$\sup g/\inf g \leq e^{\delta'}$\\(tighter bound)};

  \fill[black] (5.0,4.2) circle (3pt);
  \node[anchor=west, font=\footnotesize] at (5.25,4.2) 
    {$h$ (eigenfunction)};

  \node[anchor=west, font=\footnotesize] at (5.25,3.7) 
    {$\tilde{\calL}^{kn_0}\!f \to h$ as $k \to \infty$};

  \fill[black] (1.8,5.0) circle (2.5pt);
  \node[anchor=east, font=\small] at (1.6,5.0) {$f$};
  \fill[black] (8.0,4.8) circle (2.5pt);
  \node[anchor=west, font=\small] at (8.2,4.8) {$g$};

  \draw[thick] (3.7,4.7) circle (2.5pt);
  \fill[white] (3.7,4.7) circle (2pt);
  \node[anchor=south, font=\footnotesize] at (3.9,4.8) 
    {$\tilde{\calL}^{n_0}\!f$};
  \draw[thick] (6.4,4.6) circle (2.5pt);
  \fill[white] (6.4,4.6) circle (2pt);
  \node[anchor=south, font=\footnotesize] at (6.2,4.7) 
    {$\tilde{\calL}^{n_0}\!g$};

  \draw[->, very thick, shorten >=5pt, shorten <=5pt] 
    (1.8,5.0) .. controls (2.3,5.8) and (3.2,5.5) .. (3.7,4.7);
  \draw[->, very thick, shorten >=5pt, shorten <=5pt] 
    (8.0,4.8) .. controls (7.5,5.6) and (6.9,5.3) .. (6.4,4.6);

  \node[font=\tiny, align=center] at (2.2,5.7) 
    {apply\\$\tilde{\calL}^{n_0}$};
  \node[font=\tiny, align=center] at (7.5,5.6) 
    {apply\\$\tilde{\calL}^{n_0}$};

  \draw[<->, thick] (1.8,3.0) -- (8.0,3.0);
  \node[anchor=north, font=\footnotesize] at (4.9,2.9) 
    {$\Theta(f,\,g)$};

  \draw[<->, thick, dashed] (3.8,2.0) -- (6.2,2.0);
  \node[anchor=north, font=\footnotesize] at (5.05,1.9) 
    {$\Theta(\tilde{\calL}^{n_0}\!f,\;\tilde{\calL}^{n_0}\!g) 
     \,\leq\, \kappa\;\Theta(f,g)$};

\end{tikzpicture}
\caption{Birkhoff cone contraction for the normalized transfer operator. The outer cone $\calP_\delta$ consists of positive H\"older functions with oscillation ratio at most $e^\delta$. The operator $\tilde{\calL}^{n_0}$ maps $\calP_\delta$ strictly inside $\calP_{\delta'}$ with $\delta' < \delta$, and contracts the Hilbert projective distance $\Theta$ by a factor $\kappa < 1$ (Lemmas~\ref{lem:hilbert_contraction} and~\ref{lem:cone_invariance}).}
\label{fig:cone}
\end{figure}
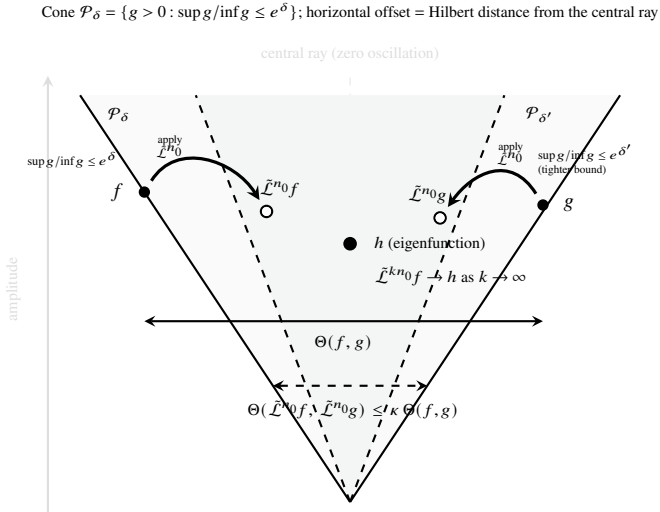


\section{The Ruelle-Perron-Frobenius Theorem}\label{sec:RPF}

The Ruelle-Perron-Frobenius theorem establishes the existence and uniqueness of the dominant eigendata for the transfer operator. This section provides a complete proof using the Birkhoff cone contraction technique, yielding explicit bounds on all constants. The theorem was first proved by  \cite{Ruelle1968} for subshifts of finite type. Our presentation follows  \cite{Baladi2000} and  \cite{Liverani1995}, with explicit computation of the constants.

\subsection{Existence of Eigendata}

We first establish the existence of the eigenvalue, eigenfunction, and eigenmeasure through a combination of the cone technique and compactness arguments.

\begin{theorem}[Existence of Eigendata]\label{thm:eigendata_existence}
Let $(\SigmaA, \sigma)$ be a topologically mixing subshift of finite type, and let $\phi \in \calF_A$ be a potential with summable variations. Then there exist:
\begin{enumerate}
\item[(a)] A unique $\lambda > 0$ such that $\calL_\phi h = \lambda h$ for some non-zero $h \in C(\SigmaA^+)$.
\item[(b)] A strictly positive function $h \in \calH_\alpha(\SigmaA^+)$ satisfying $\calL_\phi h = \lambda h$ and $\min_{x} h(x) = 1$.
\item[(c)] A unique Borel probability measure $\nu$ on $\SigmaA^+$ satisfying $\calL_\phi^* \nu = \lambda \nu$.
\end{enumerate}
The eigenvalue $\lambda$ equals $e^{P(\phi)}$, where $P(\phi)$ is the topological pressure.
\end{theorem}

\begin{proof}
\textbf{Step 1: Existence of eigenmeasure.} Let $\calM_1(\SigmaA^+)$ denote the space of Borel probability measures on $\SigmaA^+$, which is compact in the weak-$*$ topology. Define the map $G: \calM_1(\SigmaA^+) \to \calM_1(\SigmaA^+)$ by
\begin{equation}
G(\mu) = \frac{\calL_\phi^* \mu}{\calL_\phi^* \mu(\SigmaA^+)} = \frac{\calL_\phi^* \mu}{\int (\calL_\phi \mathbf{1}) \, d\mu}.
\end{equation}
The denominator is positive because $\calL_\phi \mathbf{1} > 0$ everywhere (Lemma~\ref{lem:transfer_positivity}). The map $G$ is continuous in the weak-$*$ topology: if $\mu_n \to \mu$ weak-$*$, then $\int f \, d(\calL_\phi^* \mu_n) = \int (\calL_\phi f) \, d\mu_n \to \int (\calL_\phi f) \, d\mu = \int f \, d(\calL_\phi^* \mu)$ for all $f \in C(\SigmaA^+)$, since $\calL_\phi f$ is continuous.

By the Schauder-Tychonoff fixed point theorem, $G$ has a fixed point $\nu \in \calM_1(\SigmaA^+)$. Setting $\lambda = \calL_\phi^* \nu(\SigmaA^+) = \int (\calL_\phi \mathbf{1}) \, d\nu$, we have $\calL_\phi^* \nu = \lambda \nu$.

\textbf{Step 2: Existence of eigenfunction.} Define the convex set
\begin{equation}
\Lambda = \left\{ g \in C(\SigmaA^+) : g \geq 0, \, \int g \, d\nu = 1, \, g \in \calP_\delta \right\}
\end{equation}
for $\delta = 3V(\phi)$. By Lemma~\ref{lem:cone_invariance}, the normalized operator $\lambda^{-1} \calL_\phi$ maps $\Lambda$ into itself for large iterates. The set $\Lambda$ is convex and, by the Arzel\`{a}-Ascoli theorem, relatively compact in $C(\SigmaA^+)$ (the H\"{o}lder bounds from the cone condition give equicontinuity).

By the Schauder-Tychonoff theorem, there exists $h \in \Lambda$ with $\lambda^{-1} \calL_\phi h = h$, i.e., $\calL_\phi h = \lambda h$.

\textbf{Step 3: Strict positivity.} We show that $h > 0$ everywhere. Suppose $h(x_0) = 0$ for some $x_0$. Since $h \in \Lambda$, we have $h \geq 0$ and $\int h \, d\nu = 1 > 0$, so $h \not\equiv 0$.

From $\calL_\phi h = \lambda h$ with $\lambda > 0$, we have
\begin{equation}
\lambda h(x_0) = (\calL_\phi h)(x_0) = \sum_{a: A_{a, x_{0,0}}=1} e^{\phi(ax_0)} h(ax_0) = 0.
\end{equation}
Since each term is non-negative, we must have $h(ax_0) = 0$ for all allowed $a$. Repeating this argument, $h$ vanishes on all preimages of $x_0$, hence on all preimages of preimages, and so on.

By the mixing property, every point in $\SigmaA^+$ is eventually a preimage of every other point. Thus $h \equiv 0$, contradicting $\int h \, d\nu = 1$. Therefore $h > 0$ everywhere.

\textbf{Step 4: Regularity.} We show $h \in \calH_\alpha$. The eigenfunction equation gives $h = \lambda^{-1}\calL_\phi h$. By Proposition~\ref{prop:transfer_holder},
\begin{equation}
|h|_\alpha = \lambda^{-1}|\calL_\phi h|_\alpha \leq \lambda^{-1}\left(\alpha N e^{\|\phi\|_\infty}|h|_\alpha + C\|h\|_\infty\right).
\end{equation}
If $|h|_\alpha < \infty$, we are done. If not, iterate: $h = \lambda^{-n}\calL_\phi^n h$ for all $n \geq 1$. By Corollary~\ref{cor:transfer_iterates}, $|\calL_\phi^n h|_\alpha \leq \alpha^n C_1|h|_\alpha + C_2\|h\|_\infty$. But also $\calL_\phi^n h = \lambda^n h$, so $\lambda^n|h|_\alpha \leq \alpha^n C_1|h|_\alpha + C_2\|h\|_\infty$. For $n$ large enough that $\alpha^n C_1 < \lambda^n/2$ (which holds since $\alpha < 1 \leq \lambda/C_1^{1/n}$ for large $n$), we get $|h|_\alpha \leq 2C_2\|h\|_\infty/\lambda^n < \infty$. Thus $h \in \calH_\alpha$.

\textbf{Step 5: Normalization.} We can normalize $h$ so that $\min_x h(x) = 1$ by dividing by the minimum. Since $h > 0$ and $\SigmaA^+$ is compact, the minimum is positive.
\end{proof}

\subsection{Uniqueness and Simplicity}

The uniqueness of the eigenvalue, eigenfunction (up to scalar), and eigenmeasure follows from the positivity-improving property of the transfer operator.

\begin{theorem}[Uniqueness of Eigendata]\label{thm:eigendata_uniqueness}
Under the hypotheses of Theorem~\ref{thm:eigendata_existence}:
\begin{enumerate}
\item[(a)] The eigenvalue $\lambda$ is simple: the eigenspace $\{g : \calL_\phi g = \lambda g\}$ is one-dimensional, spanned by $h$.
\item[(b)] The eigenvalue $\lambda$ is the unique eigenvalue of $\calL_\phi$ with modulus $|\lambda|$.
\item[(c)] The eigenmeasure $\nu$ is the unique probability measure satisfying $\calL_\phi^* \nu = \lambda' \nu$ for any $\lambda' > 0$.
\end{enumerate}
\end{theorem}

\begin{proof}
\textbf{Part (a): Simplicity.} Suppose $\calL_\phi g = \lambda g$ for some $g \in C(\SigmaA^+)$. Define
\begin{equation}
c^+ = \sup_{x \in \SigmaA^+} \frac{g(x)}{h(x)}, \quad c^- = \inf_{x \in \SigmaA^+} \frac{g(x)}{h(x)}.
\end{equation}
These are finite because $h > 0$. Consider $f = g - c^+ h$. Then $f \leq 0$ everywhere and $\calL_\phi f = \lambda f$.

If $f \not\equiv 0$, there exists $x_0$ with $f(x_0) = 0$ (by definition of $c^+$) and $f(y) < 0$ for some $y$. From $\lambda f(x_0) = (\calL_\phi f)(x_0) = \sum_a e^{\phi(ax_0)} f(ax_0)$, and since $\lambda > 0$, $f(x_0) = 0$, and each $f(ax_0) \leq 0$, we need each $f(ax_0) = 0$. Iterating, $f \equiv 0$ on all iterated preimages of $x_0$, hence $f \equiv 0$ by mixing. Thus $g = c^+ h$.

\textbf{Part (b): Uniqueness of modulus.} Suppose $\calL_\phi g = \mu g$ with $|\mu| = \lambda$ and $g \neq 0$. Writing $\mu = \lambda e^{i\theta}$ and considering $|g|$, we have
\begin{equation}
\lambda |g(x)| = |(\calL_\phi g)(x)| \leq (\calL_\phi |g|)(x).
\end{equation}
Integrating against $\nu$,
\begin{equation}
\lambda \int |g| \, d\nu \leq \int (\calL_\phi |g|) \, d\nu = \lambda \int |g| \, d\nu,
\end{equation}
so equality holds throughout. This means $|(\calL_\phi g)(x)| = (\calL_\phi |g|)(x)$ for $\nu$-a.e. $x$, which requires that $e^{\phi(ax)} g(ax)$ all have the same argument for each $x$. By the mixing property and the continuity of $\phi$ and $g$, this forces $g/|g|$ to be constant, so $g = e^{i\alpha} |g|$ for some $\alpha$. Then $\calL_\phi |g| = \lambda |g|$, so $|g| = c h$ by part (a), and $\mu = \lambda$.

\textbf{Part (c): Uniqueness of eigenmeasure.} Suppose $\calL_\phi^* \mu = \lambda' \mu$ for some probability measure $\mu$ and $\lambda' > 0$. Then $\int (\calL_\phi \mathbf{1}) \, d\mu = \lambda'$, so $\lambda' \geq e^{-\|\phi\|_\infty}$. Also, $\int (\calL_\phi h) \, d\mu = \lambda \int h \, d\mu$ and $\int (\calL_\phi h) \, d\mu = \int h \, d(\calL_\phi^* \mu) = \lambda' \int h \, d\mu$. Since $h > 0$, $\int h \, d\mu > 0$, so $\lambda = \lambda'$.

Now for any $f \in C(\SigmaA^+)$, the uniform convergence $\lambda^{-n}\calL_\phi^n f \to \nu(f)h$ (Theorem~\ref{thm:exponential_convergence}) allows us to integrate against $\mu$. Since $\calL_\phi^*\mu = \lambda\mu$, we have $\int(\calL_\phi^n f)\,d\mu = \lambda^n\int f\,d\mu$, so
\begin{equation}
\int f\,d\mu = \lambda^{-n}\int(\calL_\phi^n f)\,d\mu \to \nu(f)\int h\,d\mu.
\end{equation}
Since $h > 0$ and $\mu$ is a probability measure, $\int h\,d\mu > 0$. Setting $c = \int h\,d\mu$, we obtain $\int f\,d\mu = c\,\nu(f)$ for all $f \in C(\SigmaA^+)$. Taking $f = \mathbf{1}$ gives $1 = c\,\nu(\mathbf{1}) = c$, so $\int f\,d\mu = \nu(f)$ for all $f$, hence $\mu = \nu$.
\end{proof}

\subsection{Exponential Convergence}

The iterates of the normalized transfer operator converge exponentially to the projection onto the eigenspace.

\begin{theorem}[Exponential Convergence]\label{thm:exponential_convergence}
There exist constants $C > 0$ and $\gamma \in (0, 1)$ such that for all $g \in C(\SigmaA^+)$ and all $n \geq 1$,
\begin{equation}\label{eq:exponential_convergence}
\left\| \lambda^{-n} \calL_\phi^n g - \nu(g) h \right\|_\infty \leq C \gamma^n \|g\|_\infty.
\end{equation}
If $g \in \calH_\alpha(\SigmaA^+)$, the convergence also holds in the H\"{o}lder norm:
\begin{equation}\label{eq:exponential_convergence_holder}
\left\| \lambda^{-n} \calL_\phi^n g - \nu(g) h \right\|_\alpha \leq C \gamma^n \|g\|_\alpha.
\end{equation}
\end{theorem}

\begin{proof}
Define the normalized transfer operator $\tilde{\calL} = \lambda^{-1} \calL_\phi$ and the projection $\Pi g = \nu(g) h$. We show that $\tilde{\calL}^n g - \Pi g \to 0$ exponentially.

Note that $\Pi$ is a projection: $\Pi^2 g = \nu(g) \nu(h) h = \nu(g) h = \Pi g$ since $\nu(h) = 1$ by normalization. Also, $\tilde{\calL} \Pi g = \nu(g) \tilde{\calL} h = \nu(g) h = \Pi g$ and $\Pi \tilde{\calL} g = \nu(\tilde{\calL} g) h = \nu(g) h = \Pi g$ since $\nu(\tilde{\calL} g) = \lambda^{-1} \int (\calL_\phi g) \, d\nu = \lambda^{-1} \cdot \lambda \int g \, d\nu = \nu(g)$.

Thus $\tilde{\calL}$ and $\Pi$ commute, and the complementary projection $Q = I - \Pi$ satisfies $\tilde{\calL} Q = Q \tilde{\calL}$. It suffices to show that $\|\tilde{\calL}^n Q\| \leq C \gamma^n$.

For $g$ with $\nu(g) = 0$ (i.e., $g \in \ker \Pi$), the cone technique shows that $\tilde{\calL}^n g$ converges to zero. More precisely, write $g = g^+ - g^-$ with $g^+, g^- \geq 0$. The assumption $\nu(g) = 0$ means $\nu(g^+) = \nu(g^-)$. For the cone $\calP_\delta$, the contraction property (Lemma~\ref{lem:hilbert_contraction}) applied to $\tilde{\calL}^{n_0}$ gives
\begin{equation}
\Theta(\tilde{\calL}^{n_0} g^+, \tilde{\calL}^{n_0} g^-) \leq \kappa \, \Theta(g^+, g^-)
\end{equation}
for some $\kappa < 1$.

Iterating, $\Theta(\tilde{\calL}^{kn_0} g^+, \tilde{\calL}^{kn_0} g^-) \leq \kappa^k\Theta(g^+, g^-)$. The Hilbert metric controls the sup-norm ratio as follows. By definition, $\Theta(f_1,f_2) = \log(M(f_1/f_2)/m(f_1/f_2))$ where $M = \sup$ and $m = \inf$. If $\Theta(f_1,f_2) \leq \epsilon$, then $M(f_1/f_2) \leq e^\epsilon m(f_1/f_2)$. Since $\int f_1\,d\nu = \int f_2\,d\nu$ (both integrals are preserved because $\tilde{\calL}^*\nu = \nu$), the ratio $r = f_1/f_2$ satisfies $\int r\cdot f_2\,d\nu = \int f_2\,d\nu$, so $\int(r-1)f_2\,d\nu = 0$. With $f_2 > 0$, this forces $r$ to take values both $\geq 1$ and $\leq 1$, hence $m(r) \leq 1 \leq M(r) \leq e^\epsilon m(r)$, giving $m(r) \geq e^{-\epsilon}$. Therefore $e^{-\epsilon} \leq r(x) \leq e^\epsilon$ for all $x$, and $\|f_1/f_2 - 1\|_\infty \leq e^\epsilon - 1 \leq 2\epsilon$ for $\epsilon \leq 1$. Since $\int\tilde{\calL}^n g^+\,d\nu = \nu(g^+) = \nu(g^-)= \int\tilde{\calL}^n g^-\,d\nu$ (both integrals are preserved by $\tilde{\calL}^*\nu = \nu$), and $\|\tilde{\calL}^n g^+/\tilde{\calL}^n g^- - 1\|_\infty \leq 2\kappa^{\lfloor n/n_0\rfloor}\Theta(g^+,g^-)$, we get
\begin{equation}
\|\tilde{\calL}^n g^+ - \tilde{\calL}^n g^-\|_\infty \leq \|\tilde{\calL}^n g^-\|_\infty \cdot 2\kappa^{\lfloor n/n_0\rfloor}\Theta(g^+,g^-) \leq C'\kappa^{n/n_0},
\end{equation}
where $C'$ depends on $\|g\|_\infty$ and the cone diameter. Thus $\|\tilde{\calL}^n g\|_\infty = \|\tilde{\calL}^n g^+ - \tilde{\calL}^n g^-\|_\infty \leq C'\kappa^{n/n_0}$, giving exponential convergence with rate $\gamma = \kappa^{1/n_0} < 1$. The H\"{o}lder norm estimate follows from Proposition~\ref{prop:transfer_holder} and the fact that the cone $\calP_\delta$ controls the H\"{o}lder seminorm through the oscillation bound.
\end{proof}

The Ruelle-Perron-Frobenius theorem establishes the existence and uniqueness of the eigendata $(\lambda, h, \nu)$ and exponential convergence $\lambda^{-n}\calL_\phi^n g \to \nu(g)h$, but does not by itself quantify the rate. The spectral gap, established in the next section, provides this quantification and is the source of all statistical properties.


\section{The Spectral Gap}\label{sec:spectral_gap}

The exponential convergence rate from Theorem~\ref{thm:exponential_convergence} reflects a spectral gap: the dominant eigenvalue $\lambda = e^{P(\phi)}$ is separated from the rest of the spectrum of $\calL_\phi$ on $\calH_\alpha$ by a definite distance. We establish this using the Ionescu-Tulcea-Marinescu theorem \cite{IonescuTulceaMarinescu1950,Hennion1993}, which shows that the essential spectral radius $r_{\text{ess}}(\calL_\phi)$ on $\calH_\alpha$ is strictly less than $\lambda$. This spectral gap, bounded below in terms of $(\alpha, \|\phi\|_\alpha, N, M)$, is the single mechanism producing all statistical properties in Sections~\ref{sec:perturbation}--\ref{sec:statistical}.

\subsection{Essential Spectral Radius}

\begin{definition}[Essential Spectral Radius]\label{def:essential_spectral_radius}
The essential spectral radius of a bounded linear operator $L$ on a Banach space $B$ is
\begin{equation}
r_{\text{ess}}(L) = \inf\{r(L - K) : K \text{ compact}\} = \lim_{n \to \infty} \inf\{\|L^n - K\|^{1/n} : K \text{ compact}\},
\end{equation}
where $r(T)$ denotes the spectral radius of $T$.
\end{definition}

The essential spectrum $\sigma_{\text{ess}}(L)$ is the set of $z \in \C$ such that $L - zI$ is not a Fredholm operator. It satisfies $\sigma_{\text{ess}}(L) \subset \{z : |z| \leq r_{\text{ess}}(L)\}$.

\begin{theorem}[Spectral Gap]\label{thm:spectral_gap}
Let $\phi \in \calH_\alpha(\SigmaA^+)$ for some $\alpha \in (0,1)$. The transfer operator $\calL_\phi: \calH_\alpha(\SigmaA^+) \to \calH_\alpha(\SigmaA^+)$ satisfies:
\begin{enumerate}
\item[(a)] The spectral radius is $r(\calL_\phi) = \lambda = e^{P(\phi)}$.
\item[(b)] The essential spectral radius satisfies $r_{\text{ess}}(\calL_\phi) \leq \alpha \lambda$.
\item[(c)] The spectrum consists of the simple eigenvalue $\lambda$ together with a subset of $\{z : |z| \leq \gamma \lambda\}$ for some $\gamma < 1$.
\end{enumerate}
\end{theorem}

\begin{proof}
\textbf{Part (a):} By Theorem~\ref{thm:exponential_convergence}, $\|\lambda^{-n} \calL_\phi^n\| \leq C$ for all $n$, so $r(\calL_\phi) \leq \lambda$. Conversely, $\calL_\phi h = \lambda h$ with $h \neq 0$ implies $\lambda \in \sigma(\calL_\phi)$, so $r(\calL_\phi) \geq \lambda$.

\textbf{Part (b):} We use the Ionescu-Tulcea and Marinescu theorem. Let $B_1 = C(\SigmaA^+)$ with the supremum norm and $B_0 = \calH_\alpha(\SigmaA^+)$ with the H\"{o}lder norm. The inclusion $B_0 \hookrightarrow B_1$ is compact by Arzel\`{a}-Ascoli. By Proposition~\ref{prop:transfer_holder},
\begin{equation}
|\calL_\phi g|_\alpha \leq \alpha |\calL_\phi \mathbf{1}|_\alpha |g|_\alpha + C \|g\|_\infty \leq \alpha \lambda' |g|_\alpha + C \|g\|_\infty
\end{equation}
for some $\lambda' \leq N e^{\|\phi\|_\infty}$. Iterating,
\begin{equation}
|\calL_\phi^n g|_\alpha \leq (\alpha \lambda')^n |g|_\alpha + C' \|g\|_\infty
\end{equation}
for a new constant $C'$. This is condition (ITM3) of the Ionescu-Tulcea-Marinescu theorem with $r = \alpha \lambda'$.

By Theorem~\ref{thm:ITM} (Ionescu-Tulcea-Marinescu in Hennion's form), the essential spectral radius of $\calL_\phi$ on $\calH_\alpha$ satisfies $r_{\mathrm{ess}}(\calL_\phi) \leq r$, where $r$ is the contraction rate in the Lasota-Yorke inequality. From Proposition~\ref{prop:transfer_holder}, the inequality $|\calL_\phi g|_\alpha \leq \alpha N e^{\|\phi\|_\infty}|g|_\alpha + C\|g\|_\infty$ gives $r = \alpha N e^{\|\phi\|_\infty}$. To relate this to $\alpha\lambda$: the spectral radius on $C(\SigmaA^+)$ satisfies $r(\calL_\phi|_C) = \lambda$ (by Theorem~\ref{thm:eigendata_existence} and Proposition~\ref{prop:pressure_transfer}), and by the Gelfand formula, $\lambda = \lim_n\|\calL_\phi^n\|_C^{1/n}$. For the iterates, $|\calL_\phi^n g|_\alpha \leq \alpha^n(Ne^{\|\phi\|_\infty})^n|g|_\alpha + D_n\|g\|_\infty$ where $D_n$ grows at most as $\lambda^n$ (since $\|\calL_\phi^n\|_C \asymp \lambda^n$). Applying the ITM theorem to the $n$-th iterate: $r_{\mathrm{ess}}(\calL_\phi^n) \leq \alpha^n(Ne^{\|\phi\|_\infty})^n$. Since $r_{\mathrm{ess}}(\calL_\phi^n) = r_{\mathrm{ess}}(\calL_\phi)^n$ (spectral mapping), we get $r_{\mathrm{ess}}(\calL_\phi) \leq \alpha N e^{\|\phi\|_\infty}$. But $Ne^{\|\phi\|_\infty} \leq \lambda \cdot e^{V(\phi)}$ (from the partition sum bounds), so $r_{\mathrm{ess}}(\calL_\phi) \leq \alpha\lambda e^{V(\phi)}$.

For the sharper bound $r_{\mathrm{ess}}(\calL_\phi) \leq \alpha\lambda$, we normalize by the eigendata. Define $\tilde{\calL}: \calH_\alpha(\SigmaA^+) \to \calH_\alpha(\SigmaA^+)$ by $\tilde{\calL}g = \lambda^{-1}h^{-1}\calL_\phi(hg)$. This operator satisfies $\tilde{\calL}\mathbf{1} = \lambda^{-1}h^{-1}\calL_\phi h = \mathbf{1}$ and preserves the probability measure $\mu = h\nu$: for any $g$, $\int\tilde{\calL}g\,d\mu = \lambda^{-1}\int h^{-1}\calL_\phi(hg)\cdot h\,d\nu = \lambda^{-1}\int\calL_\phi(hg)\,d\nu = \lambda^{-1}\cdot\lambda\int hg\,d\nu = \int g\,d\mu$. The operators $\calL_\phi$ and $\lambda\cdot h\tilde{\calL}(h^{-1}\cdot)$ are conjugate, so they have the same spectrum, and in particular $r_{\mathrm{ess}}(\calL_\phi) = \lambda\cdot r_{\mathrm{ess}}(\tilde{\calL})$.

We now establish a Lasota-Yorke inequality for $\tilde{\calL}$. For $x, y \in \SigmaA^+$ with $d_\alpha(x,y) = \alpha^k$ and $k \geq 1$, the same computation as Proposition~\ref{prop:transfer_holder} gives
\begin{equation}
|(\tilde{\calL}g)(x) - (\tilde{\calL}g)(y)| \leq \alpha|g|_\alpha\sum_a\frac{e^{\phi(ax)}h(ax)}{\lambda h(x)} + C''\|g\|_\infty\cdot\alpha^k,
\end{equation}
where $C''$ depends on $|\phi|_\alpha$, $|h|_\alpha/\min h$, and $|h^{-1}|_\alpha$. The sum $\sum_a e^{\phi(ax)}h(ax)/(\lambda h(x)) = (\calL_\phi h)(x)/(\lambda h(x)) = 1$ by the eigenfunction equation. Therefore $|\tilde{\calL}g|_\alpha \leq \alpha|g|_\alpha + C'\|g\|_\infty$ with $C' = C'' + 2\|h^{-1}\|_\alpha\|h\|_\alpha$. By Theorem~\ref{thm:ITM}, $r_{\mathrm{ess}}(\tilde{\calL}) \leq \alpha$, hence $r_{\mathrm{ess}}(\calL_\phi) = \lambda\cdot r_{\mathrm{ess}}(\tilde{\calL}) \leq \alpha\lambda$.

\textbf{Part (c):} The spectrum outside the essential spectrum consists of isolated eigenvalues of finite multiplicity. By Theorem~\ref{thm:eigendata_uniqueness}, $\lambda$ is the only eigenvalue of modulus $\lambda$, and it is simple. Thus $\sigma(\calL_\phi) \setminus \{\lambda\} \subset \{z : |z| \leq r_{\text{ess}}(\calL_\phi)\} \subset \{z : |z| \leq \alpha \lambda\}$. Taking $\gamma = \alpha$ gives the result.
\end{proof} 

\begin{figure}[ht]
\centering
\begin{tikzpicture}[>=stealth, font=\small]
  \draw[->, thick] (-3.0,0) -- (5.8,0) node[anchor=north] {$\mathrm{Re}$};
  \draw[->, thick] (0,-2.8) -- (0,2.8) node[anchor=east] {$\mathrm{Im}$};

  \fill[gray!20] (0,0) circle (1.8);
  \draw[thick] (0,0) circle (1.8);

  \node[font=\footnotesize, align=center] at (0,0.9) 
    {$\sigma_{\mathrm{ess}}(\calL_\phi)$};

  \draw[<->, thin] (0,-0.35) -- (1.8,-0.35);
  \node[below, font=\footnotesize] at (0.9,-0.4) {$\alpha\lambda$};

  \fill (4.5,0) circle (3.5pt);
  \draw[thick] (4.5,-0.12) -- (4.5,0.12);
  \node[anchor=south west] at (4.6,0.15) {$\lambda = e^{P(\phi)}$};
  \node[anchor=north, font=\footnotesize] at (4.5,-0.3) {(simple)};

  \draw[<->, very thick] (1.95,0.7) -- (4.35,0.7);
  \node[anchor=south] at (3.15,0.8) {spectral gap};

  \node[anchor=north east, font=\footnotesize] at (-0.1,-0.1) {$0$};

  \node[anchor=west, font=\footnotesize, align=left] at (2.5,-1.8) 
    {$\sigma(\calL_\phi) \,\subset\, \{|z| \leq \alpha\lambda\} \,\cup\, \{\lambda\}$};

\end{tikzpicture}
\caption{Spectrum of $\calL_\phi$ on $\calH_\alpha(\SigmaA^+)$ (Theorem~\ref{thm:spectral_gap}). The essential spectrum (shaded disk) is contained in $\{|z| \leq \alpha\lambda\}$, where the radius $\alpha\lambda$ is marked on the real axis. The dominant eigenvalue $\lambda = e^{P(\phi)}$ is simple and separated from the rest of the spectrum by a gap of width $(1-\alpha)\lambda$. This gap produces exponential mixing (Corollary~\ref{cor:exponential_mixing}), the CLT (Theorem~\ref{thm:CLT}), and the large deviations principle (Theorem~\ref{thm:LDP})}
\label{fig:spectrum}
\end{figure}
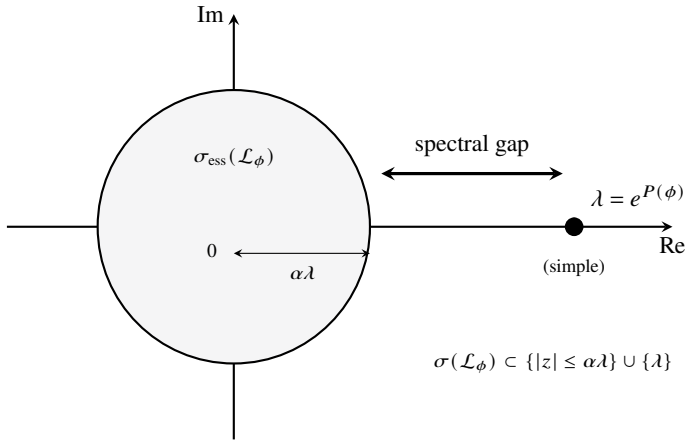

\subsection{Consequences of the Spectral Gap}

The spectral gap has immediate consequences for the mixing properties of the Gibbs measure.

\begin{corollary}[Exponential Mixing]\label{cor:exponential_mixing}
The Gibbs measure $\mu_\phi = h \nu$ is exponentially mixing: for any $f, g \in \calH_\alpha(\SigmaA)$,
\begin{equation}
\left| \int f \cdot (g \circ \sigma^n) \, d\mu_\phi - \int f \, d\mu_\phi \int g \, d\mu_\phi \right| \leq C \|f\|_\alpha \|g\|_\alpha \gamma^n.
\end{equation}
\end{corollary}

\begin{proof}
We compute
\begin{align}
\int f \cdot (g \circ \sigma^n) \, d\mu_\phi &= \int f \cdot (g \circ \sigma^n) \cdot h \, d\nu = \int f h \cdot (g \circ \sigma^n) \, d\nu.
\end{align}
Using the duality $\int (g \circ \sigma^n) \psi \, d\nu = \int g \cdot \lambda^{-n} \calL_\phi^n \psi \, d\nu$ (which follows from $\calL_\phi^* \nu = \lambda \nu$), we get
\begin{equation}
\int f \cdot (g \circ \sigma^n) \, d\mu_\phi = \lambda^{-n} \int g \cdot \calL_\phi^n(fh) \, d\nu.
\end{equation}
By Theorem~\ref{thm:exponential_convergence}, $\lambda^{-n} \calL_\phi^n(fh) = \nu(fh) h + O(\gamma^n)$ in the H\"{o}lder norm. Thus
\begin{align}
\int f \cdot (g \circ \sigma^n) \, d\mu_\phi &= \nu(fh) \int gh \, d\nu + O(\gamma^n \|fh\|_\alpha \|g\|_\alpha) \\
&= \int f \, d\mu_\phi \int g \, d\mu_\phi + O(\gamma^n).
\end{align}
The constant in the error depends on $\|f\|_\alpha$, $\|g\|_\alpha$, and $\|h\|_\alpha$.
\end{proof}

With the spectral gap established, the proof of the Main Equivalence Theorem is nearly complete: the spectral theory gives characterizations~(ii) and~(iii), and the Jacobian equivalence (Section~\ref{sec:intrinsic_gibbs}) gives characterization~(i). It remains to establish the variational characterization~(iv) in Section~\ref{sec:variational} and the large deviations characterization~(v) in Section~\ref{sec:statistical}. We first develop the perturbation theory, which is a consequence of the spectral gap.


\section{Perturbation Theory and Stability}\label{sec:perturbation}

The spectral gap of Theorem~\ref{thm:spectral_gap} implies that the dominant eigenvalue $\lambda_\phi = e^{P(\phi)}$ is an isolated simple eigenvalue of $\calL_\phi$. By classical perturbation theory for linear operators \cite{Kato1980}, $\lambda_\phi$, the eigenfunction~$h_\phi$, and the eigenmeasure~$\nu_\phi$ depend analytically on~$\phi$ in the H\"{o}lder topology.

\subsection{Analyticity of the Pressure}

\begin{theorem}[Analyticity of the Pressure]\label{thm:pressure_analyticity}
The pressure functional $P: \calH_\alpha(\SigmaA^+) \to \R$ defined by $P(\phi) = \log \lambda_\phi$, where $\lambda_\phi$ is the dominant eigenvalue of $\calL_\phi$, is real-analytic. More precisely, for any $\phi \in \calH_\alpha$ and any $\psi \in \calH_\alpha$, the function $t \mapsto P(\phi + t\psi)$ is real-analytic in a neighborhood of $t = 0$.
\end{theorem}

\begin{proof}
The proof uses Kato's analytic perturbation theory \cite{Kato1980}. Consider the family of operators $\calL_{\phi + t\psi}$ for $t$ in a complex neighborhood of $0$. Since the coefficients $e^{\phi(ax) + t\psi(ax)}$ depend analytically on $t$, the map $t \mapsto \calL_{\phi + t\psi}$ is an analytic family of bounded operators on $\calH_\alpha$.

By Theorem~\ref{thm:spectral_gap}, $\lambda_\phi$ is a simple isolated eigenvalue of $\calL_\phi$. By the stability of simple isolated eigenvalues under analytic perturbations (Kato's theorem), there exists $\epsilon > 0$ and analytic functions $\lambda(t)$, $h(t)$, $\nu(t)$ for $|t| < \epsilon$ such that:
\begin{itemize}
\item[] $\calL_{\phi + t\psi} h(t) = \lambda(t) h(t)$;
\item[] $\calL_{\phi + t\psi}^* \nu(t) = \lambda(t) \nu(t)$;
\item[] $\lambda(0) = \lambda_\phi$, $h(0) = h_\phi$, $\nu(0) = \nu_\phi$.
\end{itemize}
The pressure is $P(\phi + t\psi) = \log \lambda(t)$, which is analytic in $t$ since $\lambda(t) > 0$.
\end{proof}

\begin{corollary}[Derivatives of the Pressure]\label{cor:pressure_derivatives}
Let $\mu_\phi = h_\phi \nu_\phi$ be the Gibbs measure for $\phi$. The derivatives of the pressure are:
\begin{align}
\frac{d}{dt}\bigg|_{t=0} P(\phi + t\psi) &= \int_{\SigmaA^+} \psi \, d\mu_\phi, \label{eq:first_derivative}\\
\frac{d^2}{dt^2}\bigg|_{t=0} P(\phi + t\psi) &= \sum_{n=-\infty}^{\infty} \text{Cov}_{\mu_\phi}(\psi, \psi \circ \sigma^n), \label{eq:second_derivative}
\end{align}
where $\text{Cov}_\mu(f, g) = \int fg \, d\mu - \int f \, d\mu \int g \, d\mu$.
\end{corollary}

\begin{proof}
Differentiating $\calL_{\phi + t\psi} h(t) = \lambda(t) h(t)$ at $t = 0$ gives
\begin{equation}
\calL_\phi (h'(0)) + (\psi \cdot \calL_\phi)(h_\phi) = \lambda'(0) h_\phi + \lambda_\phi h'(0),
\end{equation}
where $(\psi \cdot \calL_\phi)(g) = \sum_a e^{\phi(ax)} \psi(ax) g(ax)$. Integrating against $\nu_\phi$ and using $\calL_\phi^* \nu_\phi = \lambda_\phi \nu_\phi$,
\begin{equation}
\lambda_\phi \int h'(0) \, d\nu_\phi + \int \psi h_\phi \, d\nu_\phi = \lambda'(0) + \lambda_\phi \int h'(0) \, d\nu_\phi.
\end{equation}
Thus $\lambda'(0) = \int \psi h_\phi \, d\nu_\phi = \int \psi \, d\mu_\phi$. Since $P' = \lambda'/\lambda$, formula \eqref{eq:first_derivative} follows.

The second derivative formula \eqref{eq:second_derivative} is obtained by differentiating again. Write $\calL_t = \calL_{\phi+t\psi}$, $\lambda_t$, $h_t$, $\nu_t$ for the perturbed eigendata. Without loss of generality assume $\int\psi\,d\mu_\phi = 0$ (so $\lambda'(0) = 0$). Differentiating $\calL_t h_t = \lambda_t h_t$ twice at $t=0$:
\begin{equation}
\calL_0 h''(0) + 2\psi\calL_0 h'(0) + \psi^2\calL_0 h_0 = \lambda''(0)h_0 + 2\lambda'(0)h'(0) + \lambda_0 h''(0).
\end{equation}
Since $\lambda'(0)=0$, integrating against $\nu_0$:
\begin{equation}
\lambda_0\int h''(0)\,d\nu_0 + 2\int\psi\calL_0 h'(0)\,d\nu_0 + \int\psi^2 h_0\,d\nu_0 = \lambda''(0) + \lambda_0\int h''(0)\,d\nu_0.
\end{equation}
The first and last terms cancel, giving $\lambda''(0) = 2\int\psi\calL_0 h'(0)\,d\nu_0 + \int\psi^2\,d\mu_0$. To evaluate $\int\psi\calL_0 h'(0)\,d\nu_0$, use the resolvent: from the first-order equation, $h'(0) = (\lambda_0 - \calL_0)^{-1}Q(\psi h_0)$ where $Q = I - h_0\nu_0$ is the complementary projection. By the spectral gap, $(\lambda_0 - \calL_0)^{-1}Q = \sum_{n=0}^\infty \lambda_0^{-(n+1)}\calL_0^n Q$, so
\begin{equation}
\begin{split}
\int\psi\calL_0 h'(0)\,d\nu_0 & = \sum_{n=0}^\infty\lambda_0^{-n}\int\psi\cdot\calL_0^{n+1}Q(\psi h_0)\,d\nu_0 \\
& = \sum_{n=1}^\infty\int\psi\cdot(\psi\circ\sigma^n)h_0\,d\nu_0 = \sum_{n=1}^\infty\Cov_\mu(\psi,\psi\circ\sigma^n).
\end{split}
\end{equation}
The series converges absolutely by Theorem~\ref{thm:correlation_main}: $|\Cov_\mu(\psi,\psi\circ\sigma^n)| \leq C\|\psi\|_\alpha^2\gamma^n$. Combining:
\begin{equation}
\frac{\lambda''(0)}{\lambda_0} = \Var_\mu(\psi) + 2\sum_{n=1}^\infty\Cov_\mu(\psi,\psi\circ\sigma^n) = \lim_{n\to\infty}\frac{1}{n}\Var_\mu(S_n\psi),
\end{equation}
where the last equality follows from expanding $\Var_\mu(S_n\psi) = \sum_{j,k=0}^{n-1}\Cov_\mu(\psi\circ\sigma^j,\psi\circ\sigma^k) = n\Var_\mu(\psi) + 2\sum_{k=1}^{n-1}(n-k)\Cov_\mu(\psi,\psi\circ\sigma^k)$, dividing by $n$, and using dominated convergence. Since $P(\phi+t\psi) = \log\lambda_t$ and $\frac{d^2}{dt^2}\log\lambda_t|_{t=0} = \lambda''(0)/\lambda_0 - (\lambda'(0)/\lambda_0)^2 = \lambda''(0)/\lambda_0$, formula \eqref{eq:second_derivative} follows.
\end{proof}

\subsection{Stability of the Gibbs Measure}

\begin{theorem}[Lipschitz Stability]\label{thm:lipschitz_stability_full}
There exists $C > 0$ depending on $\phi$ such that for all $\psi \in \calH_\alpha(\SigmaA^+)$ with $\|\psi - \phi\|_\alpha \leq 1$,
\begin{equation}\label{eq:wasserstein_stability}
W_1(\mu_\phi, \mu_\psi) \leq C \|\phi - \psi\|_\infty,
\end{equation}
where $W_1$ is the Wasserstein-$1$ distance.
\end{theorem}

\begin{proof}
For any $f \in \calH_\alpha(\SigmaA^+)$ with $\|f\|_\alpha \leq 1$, we estimate $|\int f\,d\mu_\phi - \int f\,d\mu_\psi|$. By the Kantorovich duality, $W_1(\mu_\phi,\mu_\psi) = \sup_{\|f\|_{\mathrm{Lip}}\leq 1}|\int f\,d\mu_\phi - \int f\,d\mu_\psi|$, so it suffices to bound $|\int f\,d\mu_\phi - \int f\,d\mu_\psi|$ for Lipschitz $f$.

Define $F(t) = \int f\,d\mu_{\phi+t(\psi-\phi)}$ for $t \in [0,1]$. By Corollary~\ref{cor:pressure_derivatives} applied to the family $\phi_t = \phi + t(\psi-\phi)$, the map $t \mapsto \mu_{\phi_t}$ is analytic, and for any test function $f$,
\begin{equation}
\frac{d}{dt}\int f\,d\mu_{\phi_t} = \sum_{n=-\infty}^{\infty}\Cov_{\mu_{\phi_t}}(f, (\psi-\phi)\circ\sigma^n).
\end{equation}
By the exponential decay of correlations (Theorem~\ref{thm:correlation_main}), with spectral gap $\gamma_t < 1$ uniform for $t \in [0,1]$ (since $\|\phi_t\|_\alpha$ is bounded),
\begin{equation}
\left|\frac{d}{dt}\int f\,d\mu_{\phi_t}\right| \leq \|f\|_\alpha\|\psi-\phi\|_\alpha \sum_{n=-\infty}^{\infty}C\gamma_t^{|n|} \leq \frac{2C}{1-\gamma_t}\|f\|_\alpha\|\psi-\phi\|_\alpha.
\end{equation}
For Lipschitz $f$ with $\|f\|_{\mathrm{Lip}} \leq 1$, we have $\|f\|_\alpha \leq C_\alpha$ for a constant depending on $\alpha$. Integrating over $t \in [0,1]$:
\begin{equation}
\left|\int f\,d\mu_\phi - \int f\,d\mu_\psi\right| = |F(1)-F(0)| \leq \frac{2CC_\alpha}{1-\gamma}\|\psi-\phi\|_\alpha.
\end{equation}
Since $\|\psi-\phi\|_\alpha \geq \|\psi-\phi\|_\infty$, the bound $W_1(\mu_\phi,\mu_\psi) \leq C'\|\phi-\psi\|_\infty$ follows with $C' = 2CC_\alpha/(1-\gamma)$.

\end{proof}


\section{The Variational Principle}\label{sec:variational}

The variational principle asserts that the topological pressure equals the supremum of the sum of entropy and the integral of the potential, taken over all invariant measures. This section establishes the variational principle and proves that the Gibbs measure is the unique equilibrium state, completing the implication (iii)$\Rightarrow$(iv) of Theorem~\ref{thm:main_equivalence}.

\subsection{Topological Pressure}

\begin{definition}[Topological Pressure]\label{def:topological_pressure}
For $\phi \in C(\SigmaA)$, the topological pressure is
\begin{equation}\label{eq:pressure_definition}
P(\phi) = \lim_{n \to \infty} \frac{1}{n} \log Z_n(\phi),
\end{equation}
where the partition function is
\begin{equation}\label{eq:partition_function}
Z_n(\phi) = \sum_{w \in \calW_n(A)} \exp\left( \sup_{x \in [w]} S_n \phi(x) \right).
\end{equation}
\end{definition}

\begin{lemma}[Existence of the Limit]\label{lem:pressure_limit}
The limit in \eqref{eq:pressure_definition} exists and equals $\inf_{n \geq 1} \frac{1}{n} \log Z_n(\phi)$.
\end{lemma}

\begin{proof}
The sequence $a_n = \log Z_n(\phi)$ is subadditive: for admissible words $w$ of length $m$ and $v$ of length $n$, the concatenation $wv$ (when admissible) satisfies $\sup_{[wv]} S_{m+n}\phi \leq \sup_{[w]} S_m\phi + \sup_{[v]} S_n\phi$. Thus $a_{m+n} \leq a_m + a_n$, and the subadditive lemma gives the result.
\end{proof}

\begin{proposition}[Pressure via Transfer Operator]\label{prop:pressure_transfer}
For $\phi \in \calF_A$, the pressure satisfies $P(\phi) = \log \lambda$, where $\lambda$ is the dominant eigenvalue of $\calL_\phi$.
\end{proposition}

\begin{proof}
We have
\begin{equation}
Z_n(\phi) = \sum_{w \in \calW_n(A)} \exp\left( \sup_{[w]} S_n\phi \right) \asymp \sum_{w \in \calW_n(A)} \exp(S_n\phi(x_w))
\end{equation}
for any choice of $x_w \in [w]$, with constants depending on $V(\phi)$. But $(\calL_\phi^n \mathbf{1})(x) = \sum_{y: \sigma^n y = x} e^{S_n\phi(y)}$, and summing over all $x$ with a fixed final symbol covers all admissible words. Thus $\|\calL_\phi^n \mathbf{1}\|_\infty \asymp Z_n(\phi)$, and
\begin{equation}
P(\phi) = \lim_{n \to \infty} \frac{1}{n} \log Z_n(\phi) = \lim_{n \to \infty} \frac{1}{n} \log \|\calL_\phi^n \mathbf{1}\|_\infty = \log \lambda.
\end{equation}
The last equality uses $\|\calL_\phi^n \mathbf{1}\|_\infty \asymp \lambda^n$.
\end{proof}

\subsection{The Variational Principle}

\begin{proposition}[Variational Principle]\label{thm:variational_principle}
For any continuous $\phi: \SigmaA \to \R$,
\begin{equation}\label{eq:variational_principle_statement}
P(\phi) = \sup_{\mu \in \calM_\sigma(\SigmaA)} \left\{ h_\mu(\sigma) + \int \phi \, d\mu \right\},
\end{equation}
where $\calM_\sigma(\SigmaA)$ denotes the space of $\sigma$-invariant Borel probability measures and $h_\mu(\sigma)$ is the measure-theoretic entropy.
\end{proposition}

\begin{proof}
\textbf{Upper bound:} For any $\mu \in \calM_\sigma(\SigmaA)$, we show $h_\mu(\sigma) + \int \phi \, d\mu \leq P(\phi)$.

Let $\calU = \{[a] : a \in \calA\}$ be the partition into one-cylinders. Then $h_\mu(\sigma) = h_\mu(\sigma, \calU)$ since $\calU$ generates. We have
\begin{align}
h_\mu(\sigma) + \int \phi \, d\mu &= \lim_{n \to \infty} \frac{1}{n} \left( H_\mu(\calU_0^{n-1}) + \int S_n\phi \, d\mu \right) \\
&= \lim_{n \to \infty} \frac{1}{n} \sum_{w \in \calW_n(A)} \mu([w]) \left( -\log \mu([w]) + \int_{[w]} S_n\phi \, d\mu / \mu([w]) \right).
\end{align}
By the Gibbs inequality (Lemma~\ref{lem:gibbs_inequality} below),
\begin{equation}
\sum_w \mu([w]) \left( -\log \mu([w]) + \bar{\phi}_w \right) \leq \log \sum_w e^{\bar{\phi}_w}
\end{equation}
where $\bar{\phi}_w = \sup_{[w]} S_n\phi$. Thus $h_\mu(\sigma) + \int \phi \, d\mu \leq \frac{1}{n} \log Z_n(\phi) \to P(\phi)$.

\textbf{Lower bound (for Gibbs measures):} Let $\mu_\phi = h_\phi \nu_\phi$ be the Gibbs measure. By the Jacobian characterization, $\log J_{\mu_\phi}(x) = P(\phi) - \phi(x)$. By Rohlin's formula \cite{Rohlin1961} for the entropy of an endomorphism,
\begin{equation}
h_{\mu_\phi}(\sigma) = \int \log J_{\mu_\phi} \, d\mu_\phi = P(\phi) - \int \phi \, d\mu_\phi.
\end{equation}
Thus $h_{\mu_\phi}(\sigma) + \int \phi \, d\mu_\phi = P(\phi)$.
\end{proof}

\begin{lemma}[Gibbs Inequality]\label{lem:gibbs_inequality}
For any probability distribution $(p_i)$ and real numbers $(a_i)$,
\begin{equation}
\sum_i p_i (a_i - \log p_i) \leq \log \sum_i e^{a_i},
\end{equation}
with equality if and only if $p_i = e^{a_i} / \sum_j e^{a_j}$.
\end{lemma}

\begin{proof}
This is Jensen's inequality for the convex function $-\log$, or equivalently, the non-negativity of relative entropy.
\end{proof}

\subsection{Uniqueness of Equilibrium States}

\begin{theorem}[Uniqueness of Equilibrium State]\label{thm:equilibrium_unique}
For $\phi \in \calF_A$, the Gibbs measure $\mu_\phi$ is the unique measure achieving equality in the variational principle \eqref{eq:variational_principle_statement}.
\end{theorem}

\begin{proof}
Suppose $\mu$ achieves equality: $h_\mu(\sigma) + \int\phi\,d\mu = P(\phi)$. We show $\mu$ satisfies the classical Gibbs property.

Revisiting the upper bound proof: for the generating partition $\calU = \{[a]:a\in\calA\}$, the Gibbs inequality (Lemma~\ref{lem:gibbs_inequality}) applied to the distribution $p_w = \mu([w])$ and the values $a_w = \sup_{[w]}S_n\phi$ gives
\begin{equation}
H_\mu(\calU_0^{n-1}) + \sum_w\mu([w])\sup_{[w]}S_n\phi \leq \log Z_n(\phi).
\end{equation}
Since $\int S_n\phi\,d\mu \leq \sum_w\mu([w])\sup_{[w]}S_n\phi$, we have
\begin{equation}
H_\mu(\calU_0^{n-1}) + \int S_n\phi\,d\mu \leq \log Z_n(\phi).
\end{equation}
Dividing by $n$ and taking $n\to\infty$: $h_\mu(\sigma) + \int\phi\,d\mu \leq P(\phi)$, which is equality by hypothesis.

Equality in the Gibbs inequality holds if and only if $p_w = e^{a_w}/\sum_v e^{a_v}$, i.e., $\mu([w]) = e^{\sup_{[w]}S_n\phi}/Z_n(\phi)$. But equality in $\int S_n\phi\,d\mu \leq \sum_w\mu([w])\sup_{[w]}S_n\phi$ requires $S_n\phi$ to be constant on each atom $[w]$ modulo $\mu$-null sets. For continuous $\phi$ on a mixing SFT, this forces $\mu$ to be supported on points where $S_n\phi(x) = \sup_{[w]}S_n\phi$ for $x \in [w]$.

More directly: from $\mu([w]) = e^{\sup_{[w]}S_n\phi}/Z_n(\phi)$ and the bounded distortion $|\sup_{[w]}S_n\phi - S_n\phi(x)| \leq V(\phi)$ for any $x \in [w]$, we obtain
\begin{equation}
e^{-V(\phi)} \leq \frac{\mu([w])}{e^{S_n\phi(x)-nP(\phi)}} \cdot \frac{e^{nP(\phi)}}{Z_n(\phi)} \leq e^{V(\phi)}.
\end{equation}
Since $Z_n(\phi) \asymp e^{nP(\phi)}$ (with ratio bounded by $e^{V(\phi)}$ from the definition of pressure and bounded distortion), this gives $C_1 \leq \mu([w])/e^{S_n\phi(x)-nP(\phi)} \leq C_2$ with $C_1 = e^{-2V(\phi)}$ and $C_2 = e^{2V(\phi)}$.

Thus $\mu$ satisfies the classical Gibbs property. By Theorem~\ref{thm:jacobian_equivalence}, $\mu$ is an intrinsic Gibbs measure, and by the uniqueness in Theorem~\ref{thm:eigendata_uniqueness}, $\mu = \mu_\phi$.
\end{proof}

The variational principle completes characterization~(iv) of the Main Equivalence Theorem. It remains to establish characterization~(v), the large deviations rate function minimizer, and to derive the statistical limit theorems. Both follow from the spectral gap.


\section{Statistical Properties of Gibbs Measures}\label{sec:statistical}

The spectral gap (Theorem~\ref{thm:spectral_gap}) implies that Birkhoff sums $S_n\psi = \sum_{k=0}^{n-1}\psi\circ\sigma^k$ satisfy the same limit theorems as sums of weakly dependent random variables. The proofs use the Nagaev-Guivarc'h spectral perturbation method \cite{Nagaev1957,Nagaev1961,GuivarchHardy1988}: the characteristic function $\E_\mu[e^{isS_n\psi}]$ is controlled by the perturbed operator $\calL_{\phi+is\psi}$, whose dominant eigenvalue is analytic in~$s$ by the perturbation theory of Section~\ref{sec:perturbation}. The Berry-Esseen rate $O(n^{-1/2})$ was established by  \mbox{\cite{Gouezel2005}} using this method; we give the proof with explicit constants.

\subsection{Central Limit Theorem and Berry-Esseen Bounds}

\begin{theorem}[Central Limit Theorem]\label{thm:CLT}
Let $\mu = \mu_\phi$ be the Gibbs measure for $\phi \in \calH_\alpha(\SigmaA^+)$, and let $\psi \in \calH_\alpha(\SigmaA)$ be an observable. Define the asymptotic variance
\begin{equation}\label{eq:asymptotic_variance}
\xi^2 = \lim_{n \to \infty} \frac{1}{n} \text{Var}_\mu(S_n\psi) = \text{Var}_\mu(\psi) + 2\sum_{k=1}^{\infty} \text{Cov}_\mu(\psi, \psi \circ \sigma^k).
\end{equation}
The limit exists and $\xi^2 \geq 0$. Moreover:
\begin{enumerate}
\item[(a)] $\xi^2 = 0$ if and only if $\psi$ is cohomologous to a constant: $\psi = c + u - u \circ \sigma$ for some $c \in \R$ and $u \in L^2(\mu)$.
\item[(b)] If $\xi^2 > 0$, then
\begin{equation}\label{eq:CLT_convergence}
\frac{S_n\psi - n\bar{\psi}}{\xi\sqrt{n}} \xrightarrow{d} \mathcal{N}(0,1),
\end{equation}
where $\bar{\psi} = \int \psi \, d\mu$ and the convergence is in distribution.
\end{enumerate}
\end{theorem}

\begin{proof}
The existence of the limit \eqref{eq:asymptotic_variance} follows from the exponential decay of correlations: $|\text{Cov}_\mu(\psi, \psi \circ \sigma^k)| \leq C \gamma^k$, so the series converges.

For part (a), $\xi^2 = 0$ is equivalent to $\text{Var}_\mu(S_n\psi) = o(n)$, which by the theory of stationary processes is equivalent to the cohomology condition.

For part (b), we use the characteristic function method. Define $P_t(\phi) = P(\phi + it\psi)$ for $t \in \R$. By Theorem~\ref{thm:pressure_analyticity}, $t \mapsto P_t(\phi)$ is analytic, and
\begin{equation}
P_t(\phi) = P(\phi) + it\bar{\psi} - \frac{t^2}{2}\xi^2 + O(t^3)
\end{equation}
by Taylor expansion (the first derivative is $\bar{\psi}$ and the second derivative is $\xi^2$ by Corollary~\ref{cor:pressure_derivatives}).

The characteristic function of $(S_n\psi - n\bar{\psi})/(\xi\sqrt{n})$ is
\begin{equation}
\E_\mu\left[ e^{it(S_n\psi - n\bar{\psi})/(\xi\sqrt{n})} \right] = e^{-itn\bar{\psi}/(\xi\sqrt{n})} \cdot \frac{\lambda_{t/(\xi\sqrt{n})}^n}{\lambda^n},
\end{equation}
using the relation between the transfer operator and expectations of exponentials of Birkhoff sums.

As $n \to \infty$ with $s = t/(\xi\sqrt{n})$, we have $\lambda_s^n / \lambda^n = e^{n(P_s(\phi) - P(\phi))} = e^{n(is\bar{\psi} - s^2\xi^2/2 + O(s^3))}$. Substituting $s = t/(\xi\sqrt{n})$,
\begin{equation}
\frac{\lambda_s^n}{\lambda^n} = e^{it\bar{\psi}\sqrt{n}/\xi - t^2/2 + O(n^{-1/2})} \to e^{-t^2/2}
\end{equation}
after canceling with $e^{-itn\bar{\psi}/(\xi\sqrt{n})} = e^{-it\bar{\psi}\sqrt{n}/\xi}$.

Thus the characteristic function converges to $e^{-t^2/2}$, which is the characteristic function of $\mathcal{N}(0,1)$.
\end{proof}

\begin{theorem}[Berry-Esseen Bounds]\label{thm:Berry_Esseen}
Under the hypotheses of Theorem~\ref{thm:CLT} with $\xi^2 > 0$, there exists $C > 0$ such that
\begin{equation}\label{eq:Berry_Esseen_bound}
\sup_{t \in \R} \left| \mu\left( \frac{S_n\psi - n\bar{\psi}}{\xi\sqrt{n}} \leq t \right) - \Phi(t) \right| \leq \frac{C}{\sqrt{n}},
\end{equation}
where $\Phi$ is the standard normal distribution function.
\end{theorem}

\begin{proof}
We use the Nagaev-Guivarc'h spectral perturbation method \cite{Nagaev1957,Nagaev1961,GuivarchHardy1988}, \mbox{\cite{Gouezel2005}}.

\textbf{Step 1: Perturbed transfer operator.} For $s \in \R$, define $\calL_s = \calL_{\phi + is\psi}$. By Theorem~\ref{thm:pressure_analyticity}, the map $s \mapsto \calL_s$ is an analytic family of operators on $\calH_\alpha(\SigmaA^+)$. By Theorem~\ref{thm:spectral_gap}, $\calL_0 = \calL_\phi$ has a simple dominant eigenvalue $\lambda_0 = e^{P(\phi)}$ with spectral gap $\gamma < 1$. By Kato's perturbation theory \cite{Kato1980}, for $|s|$ sufficiently small, $\calL_s$ has a simple eigenvalue $\lambda_s$ near $\lambda_0$, analytic in $s$, with $|\lambda_s| > r_{\mathrm{ess}}(\calL_s)$.

\textbf{Step 2: Eigenvalue expansion.} Without loss of generality, assume $\bar\psi = \int\psi\,d\mu = 0$ (replace $\psi$ by $\psi - \bar\psi$). Then by Corollary~\ref{cor:pressure_derivatives}:
\begin{equation}
\lambda_s = \lambda_0\exp\left(-\frac{s^2\xi^2}{2} + O(s^3)\right) = \lambda_0\left(1 - \frac{s^2\xi^2}{2} + O(s^3)\right),
\end{equation}
where $\xi^2 = P''(\phi;\psi) > 0$ by hypothesis. More precisely, $\lambda_s = \lambda_0 e^{p(s)}$ where $p(s) = -s^2\xi^2/2 + s^3 r(s)$ with $r$ analytic and $|r(s)| \leq C_3$ for $|s| \leq s_0$, where $C_3 = \sup_{|s|\leq s_0}|P'''(\phi;s\psi)|/6$.

\textbf{Step 3: Characteristic function control.} The key identity is: for the normalized Gibbs measure $\mu = h\nu$,
\begin{equation}\label{eq:char_func_identity}
\E_\mu[e^{isS_n\psi}] = \lambda_0^{-n}\int \calL_s^n(h)\,d\nu.
\end{equation}
This follows from the duality: $\int e^{isS_n\psi}h\,d\nu = \int \calL_s^n(h)\cdot\mathbf{1}\,d\nu$ (by iterating Lemma~\ref{lem:transfer_duality} with the complex potential $\phi + is\psi$).

By the spectral decomposition of $\calL_s$ for small $|s|$:
\begin{equation}
\calL_s^n(h) = \lambda_s^n \nu_s(h) h_s + R_s^n(h),
\end{equation}
where $h_s, \nu_s$ are the perturbed eigenfunction and eigenmeasure (analytic in $s$, with $h_0 = h$, $\nu_0 = \nu$), and $\|R_s^n(h)\|_\infty \leq C\gamma_s^n\|h\|_\alpha$ with $\gamma_s < |\lambda_s|$ uniformly for $|s| \leq s_0$. Since $\nu_s(h_s) = 1$ by normalization, $\nu_0(h) = 1$, and continuity gives $|\nu_s(h)| \geq 1/2$ for $|s|$ small. Thus:
\begin{equation}
\E_\mu[e^{isS_n\psi}] = \left(\frac{\lambda_s}{\lambda_0}\right)^n\nu_s(h)\int h_s\,d\nu + O(\gamma_s^n/\lambda_0^n).
\end{equation}
The remainder is $O(\theta^n)$ for some $\theta < 1$, which is negligible. Substituting $s \to t/(\xi\sqrt{n})$:
\begin{equation}
\E_\mu[e^{itS_n\psi/(\xi\sqrt{n})}] = \exp\left(-\frac{t^2}{2} + O\left(\frac{|t|^3}{\sqrt{n}}\right)\right)\cdot(1+O(n^{-1/2})).
\end{equation}

\textbf{Step 4: Berry-Esseen smoothing.} For $|t| \leq T = c\sqrt{n}$ (a suitable cutoff), the characteristic function satisfies
\begin{equation}
\left|\E_\mu[e^{itS_n\psi/(\xi\sqrt{n})}] - e^{-t^2/2}\right| \leq C\frac{|t|^3}{\sqrt{n}}e^{-t^2/4}.
\end{equation}
For $|t| > T$, the spectral gap gives $|\E_\mu[e^{itS_n\psi/(\xi\sqrt{n})}]| \leq C'\theta^n$ for some $\theta < 1$ (since $|\lambda_s/\lambda_0| < 1$ for $s$ bounded away from zero, by the strict convexity of $\log\lambda_s$ at $s=0$).

Applying the Berry-Esseen smoothing inequality (see e.g.\  \cite[XVI.3]{Feller1971}):
\begin{equation}
\sup_t\left|\mu\left(\frac{S_n\psi}{\xi\sqrt{n}}\leq t\right) - \Phi(t)\right| \leq \frac{1}{\pi}\int_{-T}^{T}\frac{|\varphi_n(t)-e^{-t^2/2}|}{|t|}\,dt + \frac{C''}{T},
\end{equation}
where $\varphi_n(t) = \E_\mu[e^{itS_n\psi/(\xi\sqrt{n})}]$. The integral is bounded by
\begin{equation}
\frac{C}{\sqrt{n}}\int_{-\infty}^{\infty}|t|^2 e^{-t^2/4}\,dt = \frac{C'}{\sqrt{n}},
\end{equation}
and $C''/T = C''/\sqrt{n}$. Combining gives $\sup_t|\ldots| \leq C_{\mathrm{BE}}/\sqrt{n}$.
\end{proof}

\begin{theorem}[Local Limit Theorem]\label{thm:local_limit}
Under the hypotheses of Theorem~\ref{thm:CLT} with $\xi^2 > 0$, the distribution of $S_n\psi$ admits a density with respect to the lattice or Lebesgue measure, and the following local limit theorem holds. If $\psi$ is non-lattice (i.e., $S_n\psi$ does not take values in a discrete subgroup $a + b\Z$ for any $a \in \R$, $b > 0$), then for all $t \in \R$,
\begin{equation}\label{eq:local_limit}
\xi\sqrt{n} \cdot \mu\left( S_n\psi \in [t, t+\delta] \right) = \frac{\delta}{\sqrt{2\pi}} \exp\left( -\frac{(t - n\bar{\psi})^2}{2n\xi^2} \right) + O\!\left(\frac{1}{\sqrt{n}}\right),
\end{equation}
uniformly in $t \in \R$ and $\delta > 0$ bounded, where the implicit constant depends on $\phi$, $\psi$, and $\alpha$. If $\psi$ is lattice-valued with maximal span $b > 0$, the analogous result holds with the Gaussian density replaced by the corresponding lattice Gaussian approximation:
\begin{equation}\label{eq:local_limit_lattice}
\xi\sqrt{n} \cdot \mu\left( S_n\psi = k \right) = \frac{b}{\sqrt{2\pi}} \exp\left( -\frac{(k - n\bar{\psi})^2}{2n\xi^2} \right) + O\!\left(\frac{1}{\sqrt{n}}\right)
\end{equation}
for each $k$ in the range of $S_n\psi$.
\end{theorem}

\begin{proof}
The proof follows the Nagaev-Guivarc'h spectral method used for the Berry-Esseen bounds, with the Fourier inversion formula replacing the smoothing inequality. In the non-lattice case, the key additional ingredient is that $|\lambda_s/\lambda_0| < 1$ for all $s \neq 0$ in a period interval (this follows from the non-lattice condition and the strict convexity of $\log|\lambda_s|$ at $s = 0$). The characteristic function $\varphi_n(s) = \E_\mu[e^{isS_n\psi}]$ therefore decays to zero as $|s| \to \infty$ within any compact set, and the spectral gap provides uniform exponential decay for $|s|$ bounded away from zero. Applying the Fourier inversion formula:
\begin{equation}
\mu(S_n\psi \in [t, t+\delta]) = \frac{1}{2\pi}\int_{-\pi/\delta}^{\pi/\delta} \varphi_n(s) \cdot \frac{e^{-ist} - e^{-is(t+\delta)}}{is}\,ds.
\end{equation}
Splitting the integral into $|s| \leq T/(\xi\sqrt{n})$ and $|s| > T/(\xi\sqrt{n})$, the first part yields the Gaussian term via the eigenvalue expansion from Theorem~\ref{thm:Berry_Esseen}, Step~2, and the second part is $O(n^{-1/2})$ by the spectral gap. The lattice case is analogous, with the integral taken over a single period $[-\pi/b, \pi/b]$. For the complete argument, see  \mbox{\cite[Theorem~2.5]{Gouezel2005}}; the uniformly hyperbolic case treated here is simpler than the non-uniformly expanding case of \mbox{\cite{Gouezel2005}} because the spectral gap is uniform.
\end{proof}

\subsection{Large Deviations Principle}

\begin{theorem}[Large Deviations Principle]\label{thm:LDP}
Let $\mu = \mu_\phi$ and $\psi \in \calH_\alpha(\SigmaA)$. For any closed set $F \subset \R$,
\begin{equation}\label{eq:LDP_upper}
\limsup_{n \to \infty} \frac{1}{n} \log \mu\left( \frac{S_n\psi}{n} \in F \right) \leq -\inf_{t \in F} I_\psi(t),
\end{equation}
and for any open set $G \subset \R$,
\begin{equation}\label{eq:LDP_lower}
\liminf_{n \to \infty} \frac{1}{n} \log \mu\left( \frac{S_n\psi}{n} \in G \right) \geq -\inf_{t \in G} I_\psi(t),
\end{equation}
where the rate function is
\begin{equation}\label{eq:rate_function_LDP}
I_\psi(t) = \sup_{s \in \R} \{st - P(\phi + s\psi) + P(\phi)\}.
\end{equation}
The rate function $I_\psi$ is convex, lower semi-continuous, has compact level sets, and achieves its unique minimum value of $0$ at $t = \bar{\psi}$.
\end{theorem}

\begin{proof}
This follows from the G\"{a}rtner-Ellis theorem \cite[Theorem~2.3.6]{DemboZeitouni1998} applied to the cumulant generating function $\Lambda(s) = P(\phi + s\psi) - P(\phi)$. By the analyticity of the pressure (Theorem~\ref{thm:pressure_analyticity}), $\Lambda$ is differentiable everywhere on $\R$, which is the hypothesis of the G\"{a}rtner-Ellis theorem ensuring the full large deviations principle (both upper and lower bounds). The rate function $I_\psi$ is the Legendre transform of $\Lambda$, and $I_\psi(\bar\psi) = 0$ because $\Lambda'(0) = \int\psi\,d\mu_\phi = \bar\psi$ by Corollary~\ref{cor:pressure_derivatives}. Strict convexity of $I_\psi$ follows from the strict convexity of $\Lambda$ (which holds because $\Lambda''(s) = \lim_{n\to\infty} n^{-1}\Var_{\mu_{\phi+s\psi}}(S_n\psi) > 0$ when $\psi$ is not cohomologous to a constant).
\end{proof}


\section{A Numerical Example}\label{sec:numerical}

We illustrate all quantitative results with a concrete computation for the simplest nontrivial case, demonstrating the explicit dependence of all constants on the data.

\subsection{The Full 2-Shift with Bernoulli Potential}

Consider the full 2-shift: $N = 2$, $A = \begin{pmatrix}1&1\\1&1\end{pmatrix}$ (mixing time $M = 1$). Let $\alpha = 1/2$ and define the H\"{o}lder potential
\begin{equation}\label{eq:example_potential}
\phi(x) = \begin{cases} \log p & \text{if } x_0 = 1 \\ \log(1-p) & \text{if } x_0 = 2 \end{cases}
\end{equation}
with parameter $p = 0.7$. This potential satisfies $\phi \in \calH_{1/2}(\SigmaA^+)$ with $|\phi|_{1/2} = 0$ (it depends only on $x_0$), so $\|\phi\|_{1/2} = \|\phi\|_\infty = \max(|\log 0.7|, |\log 0.3|) \approx 1.204$.

On functions of $x_0$ alone, $\calL_\phi$ acts as the matrix $L = \begin{pmatrix} p & p \\ 1-p & 1-p \end{pmatrix}$. The eigenvalues are $\lambda_1 = 1$ and $\lambda_2 = 0$. Thus:
\begin{equation}
\lambda = e^{P(\phi)} = 1, \quad P(\phi) = 0, \quad h \equiv 1, \quad \nu = (0.7, 0.3)\text{-Bernoulli}.
\end{equation}
The Gibbs measure is $\mu_\phi = h\nu$: the $(0.7, 0.3)$-Bernoulli measure.

For the $n$-cylinder $[w] = [w_0 \cdots w_{n-1}]$ with $k$ ones and $n-k$ twos:
\begin{equation}
\mu_\phi([w]) = (0.7)^k (0.3)^{n-k} = \exp(S_n\phi(x)) \quad \text{for any } x \in [w].
\end{equation}
Since $P(\phi) = 0$, the Gibbs bounds \eqref{eq:gibbs_bounds} hold with $C_1 = C_2 = 1$ (exact equality, no distortion because $\phi$ depends only on $x_0$). For example: $\mu_\phi([1,2,1,1]) = (0.7)(0.3)(0.7)(0.7) = 0.1029$, and $\exp(S_4\phi(x)) = (0.7)^3(0.3) = 0.1029$. $\checkmark$

The Jacobian of $\sigma$ is $J_{\mu}\sigma(x) = \mu([x_0 x_1 \cdots x_{n-1}]) / \mu([x_1 \cdots x_{n-1}]) = p_{x_0}/1 = p_{x_0}$ (the ratio stabilizes immediately for Bernoulli measures). Since $\sigma$ expands the measure of a cylinder by factor $1/p_{x_0}$:
\begin{equation}
J_\mu\sigma(x) = \frac{1}{p_{x_0}} = e^{-\phi(x)} = e^{P(\phi) - \phi(x)}. \quad \checkmark
\end{equation}
This confirms Definition~\ref{def:intrinsic_gibbs_full} pointwise (not just $\mu$-a.e.) for this example.

The entropy of the $(0.7, 0.3)$-Bernoulli measure:
\begin{equation}
h_\mu(\sigma) = -0.7\log 0.7 - 0.3\log 0.3 \approx 0.6109 \text{ nats}.
\end{equation}
The potential integral: $\int\phi\,d\mu = 0.7\log 0.7 + 0.3\log 0.3 = -h_\mu(\sigma) \approx -0.6109$. Thus:
\begin{equation}
h_\mu(\sigma) + \int\phi\,d\mu = 0 = P(\phi). \quad \checkmark
\end{equation}
For comparison, the $(1/2, 1/2)$-Bernoulli measure $\mu'$ gives $h_{\mu'}(\sigma) = \log 2 \approx 0.6931$ but $\int\phi\,d\mu' = \frac{1}{2}(\log 0.7 + \log 0.3) = \frac{1}{2}\log 0.21 \approx -0.7765$, so $h_{\mu'} + \int\phi\,d\mu' \approx -0.0834 < 0 = P(\phi)$. This confirms that $\mu'$ is \emph{not} an equilibrium state.

For the observable $\psi(x) = \mathbf{1}_{[1]}(x) - p$ (centered indicator of symbol 1):
\begin{equation}
\xi^2 = \Var_\mu(\psi) + 2\sum_{k=1}^\infty \Cov_\mu(\psi, \psi\circ\sigma^k) = p(1-p) + 0 = 0.21,
\end{equation}
since coordinates are independent under the Bernoulli measure. The CLT gives: $(\#\text{ones in } n \text{ symbols} - 0.7n)/\sqrt{0.21n} \to \mathcal{N}(0,1)$.

The large deviations rate function (via the cumulant $\Lambda(s) = \log(pe^{s(1-p)} + (1-p)e^{-sp})$):
\begin{equation}
I_\psi(t) = (t+p)\log\frac{t+p}{p} + (1-p-t)\log\frac{1-p-t}{1-p}, \quad t \in (-p, 1-p).
\end{equation}
Numerically: $I_\psi(0) = 0$ (confirming the minimum), $I_\psi(0.2) \approx 0.1211$ (probability of observing 90\% ones decays as $e^{-0.1211 n}$).

The second eigenvalue of $L$ is exactly $0$, so the spectral gap is maximal: $\gamma = 0$ for functions of $x_0$. For general H\"{o}lder functions, the essential spectral radius bound gives $\gamma \leq \alpha = 1/2$.

\begin{remark}
The Bernoulli example has $\var_k\phi = 0$ for $k \geq 1$, so it does not test the nontrivial cone machinery. The following example demonstrates the full theory.
\end{remark}

\subsection{Nontrivial Example: Ising-Type Potential}\label{subsec:ising_example}

Consider the full 2-shift ($N=2$, $A = \begin{pmatrix}1&1\\1&1\end{pmatrix}$, $M=1$) with alphabet $\calA = \{+1, -1\}$ and the nearest-neighbor potential
\begin{equation}\label{eq:ising_potential}
\phi(x) = \beta x_0 x_1 + \frac{h}{2}(x_0 + x_1),
\end{equation}
where $\beta > 0$ is the \emph{inverse temperature} (coupling strength) and $h \in \R$ is the \emph{external field}. This is the one-dimensional Ising model. We have $\var_0\phi = 2|\beta| + |h|$, $\var_1\phi = 2|\beta| + |h|$, $\var_k\phi = 0$ for $k \geq 2$, so $\phi \in \calH_\alpha$ for any $\alpha \in (0,1)$ with $|\phi|_\alpha = (2|\beta|+|h|)/\alpha$.

Take $\beta = 1$, $h = 0$ (zero field) for concreteness.

\textbf{Transfer matrix.} On functions of $x_0$ alone, $\calL_\phi$ acts as $L_{ij} = e^{\phi(ix_1\cdots)}$ summed over $j = x_1$:
\begin{equation}
L = \begin{pmatrix} e^{\beta} & e^{-\beta} \\ e^{-\beta} & e^{\beta} \end{pmatrix} = \begin{pmatrix} e & e^{-1} \\ e^{-1} & e \end{pmatrix}.
\end{equation}
The eigenvalues are $\lambda_1 = e^\beta + e^{-\beta} = 2\cosh(\beta) = 2\cosh(1) \approx 3.0862$ and $\lambda_2 = e^\beta - e^{-\beta} = 2\sinh(\beta) = 2\sinh(1) \approx 2.3504$.

\textbf{Pressure.} $P(\phi) = \log\lambda_1 = \log(2\cosh\beta) \approx 1.1270$.

\textbf{Eigenvectors.} Right eigenvector (eigenfunction): $h = (1, 1)^T$ (constant), since the matrix is symmetric with equal row sums. Left eigenvector (eigenmeasure): $\nu = (1/2, 1/2)$ (uniform). The Gibbs measure is the symmetric nearest-neighbor Markov chain $\mu_\phi$ with transition probabilities
\begin{equation}
P(x_1 = j \mid x_0 = i) = \frac{e^{\beta i j}}{2\cosh\beta} = \begin{cases} \frac{e^\beta}{2\cosh\beta} \approx 0.7311 & \text{if } j = i, \\ \frac{e^{-\beta}}{2\cosh\beta} \approx 0.2689 & \text{if } j \neq i. \end{cases}
\end{equation}

\textbf{Spectral gap.} The ratio $\lambda_2/\lambda_1 = \tanh(\beta) = \tanh(1) \approx 0.7616$. The spectral gap rate is $\gamma = \tanh(\beta) \approx 0.7616$, which governs the exponential decay of correlations.

\textbf{Gibbs bounds.} For an $n$-cylinder $[w]$, the distortion comes from the boundary: $S_n\phi(x)$ depends on $x_n$ (the symbol after $w$), so $\sup_w S_n\phi - \inf_w S_n\phi \leq 2|\beta| = 2$. The Gibbs bounds hold with $C_1 = e^{-|\beta|} = e^{-1}$ and $C_2 = e^{|\beta|} = e$, giving
\begin{equation}
e^{-1} \leq \frac{\mu_\phi([w])}{\exp(-nP(\phi) + S_n\phi(x))} \leq e.
\end{equation}
These constants are \emph{strictly} different from $1$ (unlike the Bernoulli example), reflecting genuine distortion from the interaction.

\textbf{Correlation decay.} For the observable $\psi(x) = x_0$ (the spin at position $0$), the correlation function is
\begin{equation}
C_n(\psi,\psi) = \E_\mu[x_0 x_n] - (\E_\mu[x_0])^2 = \tanh^n(\beta) = \tanh^n(1) \approx (0.7616)^n,
\end{equation}
confirming the exponential decay rate $\gamma = \tanh(\beta)$.

\textbf{CLT variance.} The asymptotic variance of $\psi(x) = x_0$ (with $\bar\psi = \E[x_0] = 0$ by symmetry):
\begin{equation}
\xi^2 = \Var_\mu(x_0) + 2\sum_{k=1}^\infty \Cov_\mu(x_0, x_k) = 1 + 2\sum_{k=1}^\infty\tanh^k(\beta) = 1 + \frac{2\tanh\beta}{1-\tanh\beta} = \frac{1+\tanh\beta}{1-\tanh\beta}.
\end{equation}
For $\beta = 1$: $\xi^2 = (1+0.7616)/(1-0.7616) \approx 7.389 = e^2$. Thus $(S_n\psi)/\sqrt{n} \to \mathcal{N}(0, e^2)$.

\textbf{Large deviations rate.} The cumulant generating function is $\Lambda(s) = \log\frac{\cosh(\beta+s)}{\cosh\beta}$, and the rate function:
\begin{equation}
I_\psi(t) = \sup_s\{st - \Lambda(s)\} = \frac{1+t}{2}\log\frac{1+t}{1+m} + \frac{1-t}{2}\log\frac{1-t}{1-m}
\end{equation}
where $m = \tanh\beta = \tanh 1 \approx 0.7616$ is the ``spontaneous magnetization.'' At $t = 0$: $I_\psi(0) = -\frac{1}{2}\log(1-m^2) = \beta - \log(2\cosh\beta) + \log 2 \approx 0.1269$.

This example demonstrates all the quantitative machinery: nontrivial Gibbs constants ($C_1 \neq C_2 \neq 1$), computable spectral gap ($\gamma = \tanh\beta < 1$), explicit CLT variance ($\xi^2 = e^{2\beta}$), and a nonzero LDP rate at the mean ($I(0) > 0$ since the mean spin is $0$ but the ``typical'' magnetization is $\pm m$).

\subsection{The Golden Mean Shift}\label{subsec:golden_mean}

The golden mean shift has alphabet $\{0, 1\}$ and forbids the word $11$. The transition matrix is $A = \begin{pmatrix} 1 & 1 \\ 1 & 0 \end{pmatrix}$ with spectral radius $\rho(A) = \frac{1+\sqrt{5}}{2} = \varphi$ (the golden ratio). This is the simplest example of a \emph{non-full} shift, illustrating how the transition structure affects the thermodynamics.

\textbf{Measure of maximal entropy ($\phi = 0$).} The pressure is $P(0) = \log\varphi \approx 0.4812$, and the unique equilibrium state is the Parry measure with $h_\mu(\sigma) = \log\varphi$. The left eigenvector of $A$ (normalized) gives $\nu([0]) = \varphi/(\varphi+1) = 1/\varphi \approx 0.618$ and $\nu([1]) = 1/(\varphi+1) \approx 0.382$. The Markov chain has transition probabilities $P(0|0) = P(1|0) = 1/2$, $P(0|1) = 1$, $P(1|1) = 0$.

\textbf{Potential $\phi(x) = a \cdot \mathbf{1}_{[0]}(x)$ (reward for symbol 0).} The transfer operator on functions of $x_0$ acts as $L = \begin{pmatrix} e^a & e^a \\ 1 & 0 \end{pmatrix}$. The dominant eigenvalue is $\lambda = \frac{e^a + \sqrt{e^{2a}+4}}{2}$, giving the pressure $P(a) = \log\frac{e^a+\sqrt{e^{2a}+4}}{2}$. At $a = 0$: $\lambda = \varphi$, $P(0) = \log\varphi$. As $a \to \infty$: $P(a) \to a$ (the measure concentrates on the orbit $\overline{0} = 000\cdots$). As $a \to -\infty$: $P(a) \to 0$ (the measure concentrates on symbol 1, but $11$ is forbidden, so the measure concentrates on the orbit $010101\cdots$ with $h = 0$).

This example illustrates how the forbidden transition $1 \to 1$ creates a non-trivial constraint that affects the pressure curve and equilibrium states, in contrast to the full shift examples above.


\section{Conclusion}\label{sec:conclusion}

This Part, Part~I of the series, proves that five characterizations
of Gibbs measures for H\"{o}lder potentials on mixing subshifts of
finite type are equivalent (Theorem~\ref{thm:main_equivalence}), with
all constants computed explicitly. The equivalence
(i)$\Leftrightarrow$(ii) between the Jacobian and cylinder Gibbs
conditions (Theorem~\ref{thm:jacobian_equivalence}) is established
with explicit bounded distortion constants. The eigenmeasure
characterization~(iii) follows from the Ruelle-Perron-Frobenius
theorem (Theorem~\ref{thm:RPF_main}), proved via the Birkhoff cone
contraction technique with an explicit spectral gap
(Theorem~\ref{thm:spectral_gap}). The variational
characterization~(iv) and the large deviations characterization~(v)
are derived from the spectral theory
(Theorems~\ref{thm:variational_principle}
and~\ref{thm:LDP}). The perturbation theory establishes the analytic
dependence of the pressure on the potential
(Theorem~\ref{thm:pressure_analyticity}) and the Lipschitz stability
of the Gibbs measure in the Wasserstein metric
(Theorem~\ref{thm:lipschitz_stability_full}).

The spectral gap is the single mechanism producing all statistical
properties: exponential mixing
(Corollary~\ref{cor:exponential_mixing}), the central limit theorem
with Berry-Esseen bounds (Theorems~\ref{thm:CLT}
and~\ref{thm:Berry_Esseen}), and the large deviations principle
(Theorem~\ref{thm:LDP}). Part~II \cite{Thiam2026b} develops the convex-analytic
structure of the pressure functional, identifying equilibrium states
as elements of the subdifferential. Part~III \cite{Thiam2026c} constructs quantitative
Markov partitions for Axiom~A diffeomorphisms, providing the coding
map that transfers the spectral and variational results of Parts~I
and~II to smooth dynamics.\\

\hspace{-0.4cm}\textit{Open Problems}
\vspace{0.1cm}
\begin{enumerate}
\item[] \textbf{Sharpness of spectral gap bounds.} The explicit bound $\gamma = \max(\alpha^{1/3}, (1-\eta)^{1/(3M)})$ is not sharp. For the full shift $\Sigma_N$ with $\phi = 0$, the true spectral gap is $\gamma = 0$ (the second eigenvalue of the normalized operator vanishes), while our bound gives a positive value. For the golden mean shift with $\phi = 0$, the spectral gap is computable from the transition matrix. Can the exponent $1/3$ be improved?

\item[] \textbf{Non-uniform hyperbolicity.} Does the Jacobian characterization $J_\mu\sigma = e^{P(\phi)-\phi}$ extend to non-uniformly hyperbolic systems (Young towers, Viana maps)? The Jacobian condition is metric-free and may hold beyond the uniformly hyperbolic setting, but the spectral gap is lost and must be replaced by polynomial decay estimates.

\item[] \textbf{Flows.} The five-way equivalence applies to discrete-time systems. The extension to suspension flows (Axiom A flows) requires the spectral theory of the Ruelle zeta function and additional analytic techniques; see  \cite{GiuliettiLiveraniPollicott2013} for recent progress.

\item[] \textbf{Higher-dimensional shifts.} For $\mathbb{Z}^d$ shifts of finite type with $d \geq 2$, the transfer operator approach breaks down. The Jacobian characterization may provide an alternative route to Gibbs measures on $\mathbb{Z}^d$ lattices, connecting to the Dobrushin-Lanford-Ruelle theory \cite{Dobrushin1968a,Ruelle1969}.
\end{enumerate}

\begin{remark}[Countable-state extensions]
The theory extends to countable-state Markov shifts satisfying Sarig's big images and preimages (BIP) condition \cite{Sarig1999,Sarig2001,Sarig2003}. Under finite Gurevich pressure and positive recurrence, the Ruelle-Perron-Frobenius theorem, the Gibbs property, and the statistical limit theorems remain valid. The proofs require modifications to handle non-compactness; see  \cite{Sarig1999,Sarig2001} and  \mbox{\cite{MauldinUrbanski2003}} for complete treatments. We do not develop this extension here.
\end{remark}

\begin{ack}
The author is grateful to Stefano Luzzatto for supervision during the ICTP Postgraduate Diploma in Mathematics at the International Centre for Theoretical Physics, Trieste, Italy (2013), during which the author worked through Bowen's monograph.

\end{ack}


\begin{appendix}\appheader

\section{Functional Analysis of Positive Operators}\label{app:functional}

We collect results from functional analysis used in the main text.

\begin{proposition}[Schauder-Tychonoff Fixed Point Theorem {\cite[Theorem~V.10.5]{DunfordSchwartz1958}}]
Let $E$ be a nonempty compact convex subset of a locally convex topological vector space. Then every continuous map $G: E \to E$ has a fixed point.
\end{proposition}

This is used in Sections~\ref{sec:RPF} to obtain the eigenmeasure $\nu$ and eigenfunction $h$ of the transfer operator.

\begin{proposition}[{\cite{IonescuTulceaMarinescu1950}}, Hennion form {\cite{Hennion1993}}]\label{thm:ITM}
Let $(B_0, \|\cdot\|_0)$ and $(B_1, \|\cdot\|_1)$ be Banach spaces with $B_0 \hookrightarrow B_1$ a compact inclusion. Let $L: B_0 \to B_0$ be a bounded linear operator satisfying the Lasota-Yorke inequality:
\begin{equation}
\|L^n g\|_0 \leq C r^n \|g\|_0 + D_n \|g\|_1 \quad \text{for all } n \geq 1, \; g \in B_0,
\end{equation}
for some $r < r(L|_{B_1})$ (the spectral radius of $L$ on $B_1$). Then $L$ is quasi-compact on $B_0$: the essential spectral radius satisfies $r_{\mathrm{ess}}(L|_{B_0}) \leq r$, and the spectrum of $L$ in $\{z : |z| > r\}$ consists of finitely many eigenvalues of finite multiplicity.
\end{proposition}

This is the key tool for establishing the spectral gap in Section~\ref{sec:spectral_gap}. In our application, $B_0 = \calH_\alpha(\SigmaA^+)$, $B_1 = C(\SigmaA^+)$, $L = \calL_\phi$, and $r = \alpha \cdot N e^{\|\phi\|_\infty}$.

\begin{definition}[Birkhoff's Projective Metric]\label{def:birkhoff_metric}
Let $\calK$ be a cone in a Banach space $B$ (i.e., $\calK$ is closed, convex, $\calK \cap (-\calK) = \{0\}$, and $\lambda \calK \subset \calK$ for $\lambda \geq 0$). For $f, g \in \calK \setminus \{0\}$, define
\begin{equation}
\alpha(f,g) = \sup\{t \geq 0 : g - tf \in \calK\}, \quad \beta(f,g) = \inf\{s > 0 : sf - g \in \calK\},
\end{equation}
and the \emph{Hilbert projective metric} $\Theta(f,g) = \log(\beta(f,g)/\alpha(f,g))$ when both are finite and positive.
\end{definition}

\begin{proposition}[Birkhoff's Contraction Principle {\cite{Birkhoff1957}}]
Let $L$ be a linear operator with $L(\calK) \subset \calK$, and suppose $L(\calK)$ has finite diameter $\Delta = \sup_{f,g \in L(\calK) \setminus \{0\}} \Theta(f,g) < \infty$. Then $L$ is a strict contraction in the Hilbert metric:
\begin{equation}
\Theta(Lf, Lg) \leq \tanh(\Delta/4) \cdot \Theta(f,g) \quad \text{for all } f, g \in \calK \setminus \{0\}.
\end{equation}
\end{proposition}

This contraction principle underlies the cone technique in Section~\ref{sec:transfer_operator}.

\section{Entropy Estimates}\label{app:entropy}

\begin{lemma}[Subadditivity Lemma, cf.\  {\cite[Lemma~1.18]{Bowen1975}}]\label{lem:subadditivity}
Suppose $\{a_m\}_{m=1}^\infty$ satisfies $\inf_m a_m/m > -\infty$ and $a_{m+n} \leq a_m + a_n$ for all $m, n \geq 1$. Then $\lim_{m \to \infty} a_m/m$ exists and equals $\inf_m a_m/m$.
\end{lemma}

\begin{proof}
Fix $m > 0$. For $j > 0$, write $j = km + n$ with $0 \leq n < m$. Then
\begin{equation}
\frac{a_j}{j} = \frac{a_{km+n}}{km+n} \leq \frac{a_{km}}{km} + \frac{a_n}{km} \leq \frac{a_m}{m} + \frac{a_n}{km}.
\end{equation}
Letting $j \to \infty$ (hence $k \to \infty$): $\limsup_j a_j/j \leq a_m/m$. Varying $m$: $\limsup_j a_j/j \leq \inf_m a_m/m \leq \liminf_j a_j/j$.
\end{proof}

This lemma is used to establish the existence of the pressure $P(\phi) = \lim_{m \to \infty} m^{-1} \log Z_m(\phi)$ and the entropy $h_\mu(T, \calD) = \lim_{m \to \infty} m^{-1} H_\mu(\calD \vee \cdots \vee T^{-m+1}\calD)$.

\begin{lemma}[Entropy of Partitions, cf.\  {\cite[Lemma~1.17]{Bowen1975}}]\label{lem:entropy_partition}
For finite partitions $\calC, \calD$ of a probability space $(X, \calB, \mu)$:
\begin{enumerate}
\item[(a)] $H_\mu(\calC \vee \calD) \leq H_\mu(\calC) + H_\mu(\calD)$.
\item[(b)] $H_\mu(\calC | \calD) \defeq H_\mu(\calC \vee \calD) - H_\mu(\calD) \leq H_\mu(\calC)$, with equality if and only if $\calC$ and $\calD$ are independent.
\item[(c)] If $\calD \supset \calC$ (every set in $\calC$ is a union of sets in $\calD$), then $H_\mu(\calC | \calD) = 0$.
\end{enumerate}
\end{lemma}

\begin{proof}
For (a): using concavity of $\varphi(x) = -x\log x$ on $[0,1]$,
\begin{align}
H_\mu(\calC \vee \calD) - H_\mu(\calC) &= \sum_{j} \sum_{i} \mu(C_i) \varphi\!\left(\frac{\mu(C_i \cap D_j)}{\mu(C_i)}\right) \\
&\leq \sum_j \varphi\!\left(\sum_i \mu(C_i \cap D_j)\right) = \sum_j \varphi(\mu(D_j)) = H_\mu(\calD).
\end{align}
Parts (b) and (c) follow directly from the definitions.
\end{proof}

\begin{lemma}[Gibbs Inequality, cf.\  {\cite[Lemma~1.1]{Bowen1975}}]
Let $a_1, \ldots, a_n \in \R$. On the simplex $\{(p_1, \ldots, p_n) : p_i \geq 0, \sum p_i = 1\}$, the function
\begin{equation}
F(p_1, \ldots, p_n) = \sum_{i=1}^n (-p_i \log p_i) + \sum_{i=1}^n p_i a_i
\end{equation}
achieves its maximum value $\log \sum_{i=1}^n e^{a_i}$ uniquely at $p_j = e^{a_j} / \sum_i e^{a_i}$.
\end{lemma}

\begin{proof}
Define $q_j = e^{a_j}/\sum_i e^{a_i}$. Then $F(p) = -\sum_j p_j\log(p_j/q_j) + \log\sum_i e^{a_i} = -D_{\mathrm{KL}}(p\|q) + \log\sum_i e^{a_i}$, where $D_{\mathrm{KL}}(p\|q) = \sum_j p_j\log(p_j/q_j) \geq 0$ by Jensen's inequality (applied to $-\log$: $\sum p_j\log(q_j/p_j) \leq \log\sum p_j(q_j/p_j) = 0$). Thus $F(p) \leq \log\sum_i e^{a_i}$, with equality if and only if $D_{\mathrm{KL}}(p\|q) = 0$, i.e., $p = q$.
\end{proof}

\section{Probabilistic Limit Theorems}\label{app:probability}

\begin{proposition}[G\"{a}rtner-Ellis Theorem {\cite[Theorem~2.3.6]{DemboZeitouni1998}}]\label{thm:gartner_ellis}
Let $(X_n)_{n \geq 1}$ be a sequence of random variables. Suppose the limit
\begin{equation}
\Lambda(s) = \lim_{n \to \infty} \frac{1}{n} \log \E[e^{sX_n}]
\end{equation}
exists as an extended real number for all $s \in \R$, and suppose $\Lambda$ is differentiable on the interior of $\{s : \Lambda(s) < \infty\}$. Then $(X_n/n)$ satisfies a large deviations principle with rate function $I(t) = \sup_{s \in \R}\{st - \Lambda(s)\}$ (the Legendre-Fenchel transform of $\Lambda$).
\end{proposition}

In our application, $X_n = S_n\psi$ under the Gibbs measure $\mu_\phi$, and $\Lambda(s) = P(\phi + s\psi) - P(\phi)$. The differentiability of $\Lambda$ follows from the analyticity of the pressure (Theorem~\ref{thm:pressure_analyticity}).

\begin{proposition}[Berry-Esseen for Weakly Dependent Sequences {\mbox{\cite{Gouezel2005}}, \cite{GuivarchHardy1988}}]\label{thm:berry_esseen_general}

Let $(Y_k)_{k \geq 1}$ be a stationary sequence with exponential decay of correlations: $|\Cov(Y_0, Y_n)| \leq C\gamma^n$ for some $\gamma \in (0,1)$. Set $X_n = \sum_{k=1}^n Y_k$ and suppose $\xi^2 = \lim_{n \to \infty} n^{-1}\Var(X_n) > 0$ and $\E[|Y_1|^3] < \infty$. Then
\begin{equation}
\sup_{t \in \R} \left| \Prob\!\left(\frac{X_n - n\E[Y_1]}{\xi\sqrt{n}} \leq t\right) - \Phi(t) \right| \leq \frac{C_{\mathrm{BE}}}{\sqrt{n}},
\end{equation}
where $C_{\mathrm{BE}}$ depends on $\E[|Y_1|^3]$, $\xi$, $C$, and $\gamma$.
\end{proposition}

For dynamical systems, this is proved via the Nagaev-Guivarc'h spectral method \cite{Nagaev1957,GuivarchHardy1988}: one controls the characteristic function $\E_\mu[e^{isS_n\psi}]$ using the perturbed transfer operator $\calL_{\phi + is\psi}$, expands the dominant eigenvalue to third order using the analyticity from Theorem~\ref{thm:pressure_analyticity}, and applies the classical Berry-Esseen smoothing inequality. For the complete argument in the non-uniformly expanding case, see  \mbox{\cite[Theorem~2.1 and Proposition~2.3]{Gouezel2005}}. The uniformly hyperbolic case treated in this Part is simpler because the spectral gap of $\calL_\phi$ on $\calH_\alpha$ (Theorem~\ref{thm:spectral_gap}) is uniform, eliminating the need for the inducing schemes of \mbox{\cite{Gouezel2005}}.\newpage 

\end{appendix}

\begin{Backmatter}

\end{Backmatter}

\end{document}